\documentclass[a4paper,twoside,11pt]{article}

\usepackage{a4wide, fancyhdr, amsmath, amssymb, mathtools, yfonts}
\usepackage{mathrsfs}
\usepackage{graphicx}
\usepackage{tikz}
\usepackage[all]{xy}
\usepackage[utf8]{inputenc}
\usepackage{amsthm}
\usepackage[english]{babel}
\usepackage{chngcntr}
\usepackage{ifthen}
\usepackage{calc}
\usepackage{hyperref}
\usepackage{authblk}
\numberwithin{equation}{section}

\newcommand{\Z}{\mathbb{Z}}
\newcommand{\Q}{\mathbb{Q}}

\newcommand\FF{\mathbb{F}}

\newcommand\Gal{\mathrm{Gal}}

\newtheorem{lemma}{Lemma}[section]
\newtheorem{theorem}[lemma]{Theorem}
\newtheorem{proposition}[lemma]{Proposition}
\newtheorem{corollary}[lemma]{Corollary}

\newtheorem{remark}{Remark}

\DeclareMathOperator{\diff}{d}
\title{\vspace{ - \baselineskip} \sffamily\bfseries On Malle's conjecture for nilpotent groups, I}
\author[1]{Peter Koymans\thanks{Vivatsgasse 7, 53111 Bonn, Germany, koymans@mpim-bonn.mpg.de}}
\author[1,2]{Carlo Pagano\thanks{Vivatsgasse 7, 53111 Bonn, Germany, carlein90@gmail.com}}
\affil[1]{Max Planck Institute for Mathematics, Bonn}
\affil[2]{University of Glasgow, Glasgow}

\date{\today}

\begin{document}
\maketitle

\begin{abstract}
We develop an abstract framework for studying the strong form of Malle's conjecture \cite{Malle1, Malle2} for nilpotent groups $G$ in their regular representation. This framework is then used to prove the strong form of Malle's conjecture for any nilpotent group $G$ such that all elements of order $p$ are central, where $p$ is the smallest prime divisor of $\# G$. 

We also give an upper bound for any nilpotent group $G$ tight up to logarithmic factors, and tight up to a constant factor in case all elements of order $p$ pairwise commute. Finally, we give a new heuristical argument supporting Malle's conjecture in the case of nilpotent groups in their regular representation.
\end{abstract}

\section{Introduction}
Let $G$ be a non-trivial, finite, nilpotent group and let $K$ be a number field. In this paper we are interested in the counting function
\[
N(X, G, K) := \#\{L/K \text{ Galois} : \Gal(L/K) \cong G, \ |N_{K/\Q}(\Delta(L/K))| \leq X\},
\]
where $\Delta(L/K)$ denotes the relative discriminant of $L/K$. As part of a broader conjecture, Malle \cite{Malle1, Malle2} conjectured that there exists a constant $c(G, K) > 0$ such that
\begin{align}
\label{eStrongMalle}
N(X, G, K) \sim c(G, K) X^{a(G)} (\log X)^{b(G, K) - 1},
\end{align}
where $a(G)$ and $b(G, K)$ are explicit constants that can be computed as follows. Let $p$ be the smallest prime divisor of $\# G$. Then
\[
a(G) = \frac{p}{(p - 1) \# G}, \quad b(G, K) = \#\{C \in \textup{Conj}(G) : C \textup{ non-trivial and } c^p = \text{id} \ \forall c \in C\}/\sim,
\]
where $\text{Conj}(G)$ denotes the set of conjugacy classes of $G$ and where two conjugacy classes $C$ and $C'$ are equivalent if there exists $\sigma \in \Gal(\overline{K}/K)$ such that $\sigma \ast C = C'$. Here the action is given by $\sigma \ast C = C^{\chi(\sigma)}$ with $\chi : \Gal(\overline{K}/K) \rightarrow \hat{\Z}^\ast$ the cyclotomic character.

We shall refer to equation (\ref{eStrongMalle}) as the strong form of Malle's conjecture. In case $G$ is allowed to be an arbitrary finite group, counterexamples are known to the strong form of Malle's conjecture, see the work of Kl\"uners \cite{KlunersCounter}. This led T\"urkelli \cite{Turkelli} to propose a corrected version of equation (\ref{eStrongMalle}). The strong form of Malle's conjecture (including the more general situation where $G$ is not necessarily considered in its regular representation) has been verified in a limited amount of cases, see \cite{DatskovskyWright, DH} for $S_3$, \cite{Wright} for $G$ abelian, \cite{CDO} for $D_4 \subseteq S_4$, \cite{KlunersHab} for generalized quaternion groups, \cite{BhargavaI} for $S_4$, \cite{BhargavaII} for $S_5$, \cite{BhargavaWood} for $S_3 \subseteq S_6$, \cite{FK} for nonic Heisenberg extensions, \cite{MTTW, Wang} for direct products $G \times A$ with $G = S_3, S_4, S_5$ and $A$ abelian. There is also the recent work \cite{ASVW}, which counts $D_4 \subseteq S_4$ when the extensions are ordered by Artin conductor instead of discriminant.

The aim of this paper is to prove the strong form of Malle's conjecture for a large family of nilpotent groups. In the process, we give a parametrization of $G$-extensions that may prove fruitful for future investigations. The authors hope to use these techniques to deal with various other groups in future work.

\begin{theorem}
\label{tMalle}
Let $K$ be a number field and let $G$ be a non-trivial, finite, nilpotent group. Let $p$ be the smallest prime divisor of $\# G$ and assume that all elements of order $p$ are central. Then there exists a constant $c > 0$ such that
\[
N(X, G, K) \sim c X^{\frac{p}{(p - 1) \# G}} (\log X)^{b(G, K) - 1},
\]
where $b(G, K)$ is the Malle constant, which equals the number of elements of order $p$ divided by $[K(\zeta_p) : K]$ in this case.
\end{theorem}

It is easy to construct many $2$-groups $G$ satisfying the hypotheses of Theorem \ref{tMalle}. Take for example any finite, abelian $2$-group $A$. Then there is an action of $\Z/2\Z$ on $A$ by inversion, which gives an action of $\Z/4\Z$ on $A$ by projecting first to $\Z/2\Z$. If one takes $G = A \rtimes \Z/4\Z$, then $G$ fulfills all the conditions of the theorem. Note that such $G$ can have arbitrarily large nilpotency class.

We remark that our techniques do not give an explicit handle on the constant $c > 0$ guaranteed by Theorem \ref{tMalle}. Even in the case where $G$ is cyclic of prime order it is a rather non-trivial task to provide an explicit value of the constant $c$, see \cite{CDO2}. The following classical result of Wright \cite{Wright} is an immediate corollary of the above theorem. Our proof is substantially shorter and makes only very limited use of class field theory.

\begin{corollary}
Let $K$ be a number field and let $A$ be a non-trivial, finite, abelian group. Let $p$ be the smallest prime divisor of $\# A$. Then there exists a constant $c > 0$ such that
\[
N(X, A, K) \sim c X^{\frac{p}{(p - 1) \# A}} (\log X)^{b(A, K) - 1},
\]
where $b(A, K)$ is the number of elements of order $p$ divided by $[K(\zeta_p) : K]$.
\end{corollary}

It is worth mentioning that the weak form of Malle's conjecture, which asserts that
\[
X^{a(G)} \ll N(X, G, K) \ll_\epsilon X^{a(G) + \epsilon},
\]
is much better understood. There are no known counterexamples to the weak form, even when $G$ is allowed to be an arbitrary finite group. The weak form is known for nilpotent groups by the work of Kl\"uners--Malle \cite{KlunersMalle}. Alberts \cite{Alberts1, Alberts2} and Alberts--O'Dorney \cite{AD} made further progress in the solvable case. Our parametrization of $G$-extensions immediately implies the following theorem, which improves, in case of nilpotent groups in the regular representation, on recent work of Kl\"uners--Wang \cite{KlunersWang} and Kl\"uners \cite{KlunersUpper}.

\begin{theorem}
\label{tUpper}
Let $K$ be a number field and let $G$ be a non-trivial, finite, nilpotent group. Let $p$ be the smallest prime divisor of $\# G$. Then
\[
N(X, G, K) \ll X^{\frac{p}{(p - 1) \# G}}  (\log X)^{i(G, K) - 1},
\]
where $i(G, K)$ is the number of elements of order $p$ in $G$ divided by $[K(\zeta_p) : K]$.
\end{theorem}

The result in Theorem \ref{tUpper} is sharp up to logarithmic factors. It should be possible to use our techniques to prove a more general version of Theorem \ref{tUpper} valid for arbitrary representations of nilpotent groups, but we shall not pursue this further here. Note that the upper bound in Theorem \ref{tUpper} matches Malle's prediction precisely when we are in the situation of Theorem \ref{tMalle}. 

Our next theorem shows that we can achieve a (conjecturally) sharp upper bound, up to a constant factor, provided that the elements of order $p$ commute with each other.

\begin{theorem}
\label{tUpper2}
Let $K$ be a number field and let $G$ be a non-trivial, finite, nilpotent group. Let $p$ be the smallest prime divisor of $\# G$. Suppose that all elements of order $p$ commute with each other. Then
\[
N(X, G, K) \ll X^{\frac{p}{(p - 1) \# G}}  (\log X)^{b(G, K) - 1}.
\]
\end{theorem}

To give examples of groups $G$ where Theorem \ref{tUpper2} applies (but Theorem \ref{tMalle} does not), consider an $\FF_2$ vector space $V$ of dimension $2^n$ and pick an ordered basis $\{b_0, \dots, b_{2^n - 1}\}$. Let $\Z/2^n\Z$ act on $V$ by cycling the ordered basis and extending linearly. Then we can take $G = V \rtimes \Z/2^{n + 1}\Z$, where $\Z/2^{n + 1}\Z$ acts on $V$ by first projecting to $\Z/2^n\Z$. Note that such $G$ can again have arbitrarily large nilpotency class.

There is also the related problem of counting the number of degree $n$ extensions with bounded discriminant, which was first treated by Schmidt \cite{Schmidt}. His upper bound was drastically improved by Ellenberg--Venkatesh \cite{EV}, Couveignes \cite{Couveignes} and Lemke Oliver--Thorne \cite{OlivierThorne}.

Malle's conjecture has strong ties with the Cohen--Lenstra conjectures \cite{cohen--lenstra}. There is the classical work of Davenport--Heilbronn \cite{DH} on $3$-torsion of class groups of quadratic fields, which was later extended by \cite{BST} and \cite{TT} in the form of a secondary main term. Davenport and Heilbronn obtain their results by counting certain $S_3$-extensions.

Fouvry and Kl\"uners \cite{FK1, fouvry--kluners} dealt with the $4$-rank of quadratic fields building on earlier work of Gerth \cite{Gerth}, and Heath-Brown \cite{HB} on $2$-Selmer groups. There is also a rich literature on upper bounds for $2$-torsion elements in class groups of which we mention \cite{KP2, KP3} for multiquadratic extensions, \cite{BSTTTZ} for $S_n$-extensions and \cite{Siad, Siad2} for monogenic $S_n$-extensions. Furthermore, the average size of the $2$-torsion of $S_n$-extensions has been determined in \cite{HSV} conditional on a tail estimate. Over function fields Malle's conjecture and the Cohen--Lenstra conjectures are better understood due to the results in \cite{ETW, EVW}.

Recently, Smith \cite{Smith1, Smith} dealt with the $2$-part of class groups of quadratic fields, which was extended by the authors to the $\ell$-part of class groups of degree $\ell$ cyclic fields \cite{KP1}. His techniques can be adapted to give a lower bound for the number of $D_{2^n}$-extensions of $\Q$ of the correct order of magnitude: this is another instance that highlights the clear ties between Malle's conjecture and the Cohen--Lenstra conjectures. It seems plausible that this can be extended to an asymptotic once one extends the results on ray class groups of Pagano--Sofos \cite{Pagano--Sofos}.

The paper is divided as follows. We start with some preliminaries in Section \ref{sPrelim}. We have opted to first treat the case of $2$-groups over $\Q$, which avoids some of the technical issues that we will face in the general case. The core of the paper is Section \ref{Parametrization}, where we provide a new parametrization of $2$-extensions over $\Q$. From this we deduce Theorem \ref{tUpper} and Theorem \ref{tMalle} in respectively Section \ref{sUpper} and Section \ref{sMainQ} still in the special case of $2$-extensions over $\Q$. 

We then generalize Section \ref{Parametrization} to arbitrary nilpotent groups and arbitrary number fields in Section \ref{sMain}. In Section \ref{sLocal conditions} we give a new heuristic in support of Malle's conjecture for nilpotent groups in their regular representation. Our main theorems are proven in full generality in Section \ref{sTheorems} with Section \ref{sAna} providing some analytic tools.

\subsection*{Acknowledgements}
The authors wish to thank the Max Planck Institute for Mathematics in Bonn for its great work conditions and an inspiring atmosphere. We are also grateful to user $2734364041$ of MathOverflow for answering a question related to Theorem \ref{tCheb}. The first named author gratefully acknowledges the Max Planck for financial support. The second named author gratefully acknowledges financial support through EPSRC Fellowship EP/P019188/1, ``Cohen--Lenstra heuristics, Brauer relations, and low-dimensional manifolds''. 

\section{Preliminaries and setup}
\label{sPrelim}
We use the abbreviation $[n] := \{1, \dots, n\}$ throughout the paper. For a set $S$ and a prime number $l$, we write $\FF_l^S$ for the free $\FF_l$ vector space on the set $S$. Let us fix a separable closure $\mathbb{Q}^{\text{sep}}$ of $\mathbb{Q}$ once and for all. All our number fields are implicitly taken inside this fixed separable closure. Furthermore, we write $G_K := \Gal(\mathbb{Q}^{\text{sep}}/K)$ for the absolute Galois group of a number field $K$. Similarly, for each prime $p$ we fix a separable closure $\mathbb{Q}_p^{\text{sep}}$ of $\mathbb{Q}_p$ (which allows us to define $G_K := \Gal(\mathbb{Q}_p^{\text{sep}}/K)$ for any extension $K$ of $\mathbb{Q}_p$ inside $\mathbb{Q}_p^{\text{sep}}$) together with an embedding
$$
i_p: \mathbb{Q}^{\text{sep}} \to \mathbb{Q}_p^{\text{sep}}.
$$
This yields an inclusion
$$
i_p^{*}: G_{\mathbb{Q}_p} \to G_{\mathbb{Q}}.
$$
We will denote by
$$
I_p := \Gal(\mathbb{Q}_p^{\text{sep}}/\mathbb{Q}_p^{\text{unr}})
$$
the inertia subgroup, where $\mathbb{Q}_p^{\text{unr}}$ is the maximal unramified extension of $\mathbb{Q}_p$. Write $\mathcal{P}$ for the collection of odd prime numbers and $\mathbb{Q}^{\text{pro}-2}$ for the compositum of all finite Galois extensions of degree a power of $2$. Let
$$
\mathcal{G}_{\mathbb{Q}}^{\text{pro}-2} := \text{Gal}(\mathbb{Q}^{\text{pro}-2}/\mathbb{Q})
$$
be the corresponding Galois group. Equipped with the Krull topology the group $\mathcal{G}_{\mathbb{Q}}^{\text{pro}-2}$ is by construction a pro-$2$ group. Let us denote by
$$
I_p(2) := \text{proj}(G_{\mathbb{Q}} \to \mathcal{G}_{\mathbb{Q}}^{\text{pro}-2}) \circ i_p^{*}(I_p).
$$
The next proposition describes $I_p(2)$ for $p \in \mathcal{P}$ and provides us with a convenient set of topological generators of $\mathcal{G}_{\mathbb{Q}}^{\text{pro}-2}$.

\begin{proposition} 
\label{cyclic inertia generators}
The group $I_p(2)$ is pro-cyclic for each $p \in \mathcal{P}$. Furthermore
$$
\overline{\langle \{I_p(2)\}_{p \in \mathcal{P} \cup \{2\}} \rangle} = \mathcal{G}_{\mathbb{Q}}^{\textup{pro}-2}.
$$
\end{proposition}

\begin{proof}
The first part follows from the description of the maximal tame extension of $\mathbb{Q}_p$ for $p \in \mathcal{P}$, which immediately implies that the maximal pro-$2$-quotient of $I_p$ is isomorphic to $\Z_2$ for $p \in \mathcal{P}$. Then $I_p(2)$ is the image of a pro-cyclic group, hence pro-cyclic.

For the second part, recall that if $\mathcal{G}$ is a pro-$2$ group and $S \subseteq \mathcal{G}$ is a subset, then $S$ topologically generates $\mathcal{G}$ if and only if $\bigcup_{g \in \mathcal{G}} gSg^{-1}$ topologically generates $\mathcal{G}$. Indeed, this is a consequence of the following well-known facts
\begin{itemize}
\item a subset $S$ of a profinite group $\mathcal{G}$ topologically generates $\mathcal{G}$ if and only if $S$ generates in every continuous finite quotient of $\mathcal{G}$;
\item a subset $S$ of a finite group $G$ generates if and only if $S$ generates modulo the Frattini subgroup of $G$;
\item the Frattini subgroup equals $G^2 [G, G]$ for a finite $2$-group $G$.
\end{itemize}
By Minkowksi's theorem we know that $\mathbb{Q}^{\text{pro}-2}/\mathbb{Q}$ does not possess non-trivial unramified subextensions. Hence 
\[
\bigcup_{p \in \mathcal{P} \cup \{2\}} \bigcup_{\sigma \in I_p(2)} \bigcup_{g \in \mathcal{G}_{\mathbb{Q}}^{\text{pro}-2}} g \sigma g^{-1}
\]
topologically generates $\mathcal{G}_{\mathbb{Q}}^{\text{pro}-2}$. We conclude that $\{I_p(2)\}_{p \in \mathcal{P} \cup \{2\}}$ topologically generates $\mathcal{G}_{\mathbb{Q}}^{\text{pro}-2}$ as desired.
\end{proof}

We remark that the above proof in fact shows that if we find a subset $S \subseteq \mathcal{G}_{\mathbb{Q}}^{\textup{pro}-2}$ such that $S$ topologically generates modulo the Frattini subgroup of  $\mathcal{G}_{\mathbb{Q}}^{\text{pro}-2}$, then $S$ is a set of topological generators for $\mathcal{G}_{\mathbb{Q}}^{\text{pro}-2}$. We fix once and for all a topological generator $\sigma_p \in I_p(2)$ for all $p \in \mathcal{P}$. For $a \in \mathbb{Q}^\ast$, we write $\chi_a: G_\Q \rightarrow \FF_2$ for the quadratic character corresponding to $\Q(\sqrt{a})$. Then we pick $\sigma_2(1), \sigma_2(2) \in I_2(2)$ such that
\[
\chi_2(\sigma_2(1)) = 1, \quad \chi_2(\sigma_2(2)) = 0, \quad \chi_{-1}(\sigma_2(1)) = 0, \quad \chi_{-1}(\sigma_2(2)) = 1.
\]
Proposition \ref{cyclic inertia generators} together with the above remark tells us that the set $\{\sigma_p\}_{p \in \mathcal{P}} \cup \{\sigma_2(1), \sigma_2(2)\}$ is a set of topological generators for the Galois group $\mathcal{G}_{\mathbb{Q}}^{\text{pro}-2}$.

Let $L/\mathbb{Q}$ now be a finite Galois extension with Galois group $G := \text{Gal}(L/\mathbb{Q})$. Fix a $2$-cocycle $\theta$
$$
\theta: G^2 \to \mathbb{F}_2
$$
with the requirement that $\theta(\text{id}, \text{id}) = 0$. Note that every class in $H^2(G, \FF_2)$ contains such a $2$-cocycle, since if we let $\phi: G \rightarrow \FF_2$ be the constant non-trivial map, then for every $2$-cocycle $\theta'$, exactly one of $\theta'$ and $\theta'+ \diff \phi$ has the required property. Then one can form the group 
$$
(\mathbb{F}_2 \times G,*_{\theta}),
$$
where the group law is
$$
(a_1, g_2)*_{\theta}(a_2, g_2)=(a_1+a_2+\theta(g_1, g_2), g_1g_2).
$$
Our assumption on $\theta$ makes sure that $(0, \text{id})$ is the trivial element of $(\mathbb{F}_2 \times G,*_{\theta})$: for a general $2$-cocycle the identity element is $(\theta(\text{id},\text{id}),\text{id})$. The following proposition plays a key role in the parametrization of $G$-extensions.

\begin{proposition} 
\label{defining normalized phi}
Let $L/\Q$ be a Galois extension with $G := \Gal(L/\Q)$. Suppose that $\theta$ is non-trivial in $H^2(G, \mathbb{F}_2)$. \\
$(a)$ The natural projection map $\pi: G_{\mathbb{Q}} \twoheadrightarrow G$ can be lifted to a surjective homomorphism
$$
\psi: G_{\mathbb{Q}} \to (\mathbb{F}_2 \times G,*_{\theta})
$$
if and only if $\textup{inv}_p(\theta) = 0$ for each prime $p$ ramifying in $L/\mathbb{Q}$. Furthermore, for all lifts $\psi$, the $\mathbb{F}_2$-coordinate of $\psi$ is a continuous $1$-cochain $\phi(\psi): G_\Q \rightarrow \FF_2$ with
$$
\diff(-\phi(\psi)) = \theta,
$$
where $\diff(\phi(\psi))(\sigma, \tau) := \phi(\psi)(\sigma) + \phi(\psi)(\tau) - \phi(\psi)(\sigma \tau)$. Conversely, for any continuous $1$-cochain $\phi: G_\Q \rightarrow \FF_2$ with $\diff(-\phi) = \theta$ the assignment
$$
\psi(\phi)(g) = (\phi(g), \pi(g))
$$
is an epimorphism lifting the canonical projection $\pi: G_{\mathbb{Q}} \to G$ to an epimorphism $G_{\mathbb{Q}} \to (\mathbb{F}_2 \times G,*_{\theta})$. The two assignments are mutual inverses. \\
$(b)$ In case one has a lift $\psi$ as in part $(a)$, then there is a unique one satisfying
$$
\phi(\psi)(\sigma_p) = 0 \textup{ for all } p \in \mathcal{P} \textup{ and } \phi(\psi)(\sigma_2(1)) = \phi(\psi)(\sigma_2(2)) = 0.
$$
\end{proposition}

\begin{proof}
For part $(a)$, observe that a map $\phi(\psi): G_{\mathbb{Q}} \to \mathbb{F}_2$ is the first coordinate of a homomorphism
$$
\psi: G_{\mathbb{Q}} \twoheadrightarrow (\mathbb{F}_2 \times G,*_{\theta}), \quad g \mapsto (\phi(\psi)(g), \pi(g))
$$
if and only if
$$
\diff(-\phi(\psi)) = \theta.
$$
Indeed, in order for $\psi$ to be a homomorphism, it is necessary and sufficient that
$$
(\phi(\psi)(g_1g_2), \pi(g_1g_2)) = \psi(g_1g_2) = \psi(g_1) \psi(g_2) = (\phi(\psi)(g_1) + \phi(\psi)(g_2) + \theta(g_1, g_2), \pi(g_1g_2)),
$$
which is in turn equivalent to
$$
\diff(-\phi(\psi))(g_1, g_2) = \phi(\psi)(g_1g_2) - \phi(\psi)(g_1) - \phi(\psi)(g_2) = \theta(g_1, g_2).
$$
Let $\phi(\psi): G_\Q \rightarrow \FF_2$ be a continuous $1$-cochain such that $\diff(-\phi(\psi)) = \theta$. We claim that $(\phi(\psi), \pi)$ is surjective. Take a non-trivial character $\chi: (\mathbb{F}_2 \times G,*_{\theta}) \rightarrow \mathbb{F}_2$. We further claim that $\chi$ comes from $G$. Indeed, if not, then the kernel of $\chi$ would readily provide a splitting of $\theta$, contrary to our assumption that $\theta$ is non-trivial. Since $\psi$ is surjective, there exists $g \in G$ with $\chi(g) = 1$. This implies that the image of $(\phi(\psi), \pi)$ generates modulo the Frattini subgroup of $(\mathbb{F}_2 \times G,*_{\theta})$ and therefore $(\phi(\psi), \pi)$ is surjective as claimed.

Furthermore, we see that the lifting $\psi$ exists if and only if the inflation of $\theta$ to $H^2(G_{\mathbb{Q}}, \mathbb{F}_2)$ is trivial, which happens if and only if $\theta$ is trivial in $H^2(G_{{\mathbb{Q}}_v}, \mathbb{F}_2)$ for every place $v$ in $\mathcal{P} \cup \{2\}$: the unique archimedean place of $\mathbb{Q}$ is then guaranteed by Hilbert reciprocity. Thanks to \cite[Proposition 4.4]{KP1}, the vanishing at the places unramified in $L/\Q$ is already guaranteed: notice that in \cite[Section 4]{KP1} the number $l$ is assumed to be an odd prime but Proposition \cite[Proposition 4.4]{KP1} also holds for $l = 2$ with an identical proof. This ends the proof of part $(a)$. 

We now turn to part $(b)$. The uniqueness follows immediately from part $(a)$ and the fact that $\{\sigma_p\}_{p \in \mathcal{P}} \cup \{\sigma_2(1),\sigma_2(2)\}$ is a system of topological generators for $\mathcal{G}_{\mathbb{Q}}^{\text{pro}-2}$. Indeed, an epimorphism $\psi$ as in part $(a)$ is entirely determined by its values on a set of topological generators. It remains to establish the existence of $\phi(\psi)$, where we shall use that $\{\sigma_p\}_{p \in \mathcal{P}} \cup \{\sigma_2(1),\sigma_2(2)\}$ is a minimal set of topological generators. Choose a map $\phi: G_\Q \rightarrow \mathbb{F}_2$ satisfying 
$$
\diff(-\phi) = \theta. 
$$
The resulting epimorphism $\psi(\phi):\mathcal{G}_{\mathbb{Q}}^{\text{pro}-2} \twoheadrightarrow (\mathbb{F}_2 \times G, *_{\theta})$ corresponds to a finite extension. Therefore we have $\phi(\sigma_{p}) = 0$ for all but finitely many $p \in \mathcal{P}$. Define for $p \in \mathcal{P}$
\[
p^\ast = (-1)^{\frac{p - 1}{2}} p.
\]
Now the sum
$$
\phi(\sigma_2(1)) \cdot \chi_2 + \phi(\sigma_2(2)) \cdot \chi_{-1} + \sum_{p \in \mathcal{P}} \phi(\sigma_{p}) \cdot \chi_{p^{*}}
$$
is a well-defined element of $H^1(G_{\mathbb{Q}}, \mathbb{F}_2)$. Hence, since $\{\sigma_2(1),\sigma_2(2)\} \cup \{\sigma_p\}_{p \in \mathcal{P}}$ and $\{\chi_2,\chi_{-1}\} \cup \{\chi_{p^*}\}_{p \in \mathcal{P}}$ are dual to each other, we obtain that
$$
\phi + \phi(\sigma_2(1)) \cdot \chi_2+\phi(\sigma_2(2)) \cdot \chi_{-1} + \sum_{p \in \mathcal{P}} \phi(\sigma_{p}) \cdot \chi_{p^{*}}
$$
vanishes at $\sigma_2(1),\sigma_2(2)$ and at $\sigma_p$ for all $p \in \mathcal{P}$, which completes the proof of part $(b)$. 
\end{proof}

We denote the unique $1$-cochain as in part $(b)$ of Proposition \ref{defining normalized phi} by $\phi(G, \theta)$. In case $\theta$ is trivial as a $2$-cocycle, then we define $\phi(G, \theta) := 0$. With this convention for $\phi(G, \theta)$, we observe that it still satisfies the conclusion of Proposition \ref{defining normalized phi} part (b), since there is exactly one quadratic character (namely the trivial character) vanishing on the system of topological generators $\{\sigma_1(2),\sigma_2(2)\} \cup \{\sigma_p\}_{p \in \mathcal{P}}$.

Write $\mathcal{S}$ for the set of squarefree integers (possibly negative). In what follows we identify $\mathcal{S}$ with
$$
H^1(\mathcal{G}_{\mathbb{Q}}^{\text{pro}-2}, \mathbb{F}_2)
$$
via the assignment
$$
d \mapsto \chi_d. 
$$
We denote by $\text{Prim}(\mathcal{S}^{\mathbb{F}_2^A - \{(0, \ldots, 0)\}})$ the subset of $\mathcal{S}^{\mathbb{F}_2^A - \{(0, \ldots, 0)\}}$ consisting of pairwise coprime squarefree integers, where coprimality for a pair of squarefree integers $(a_1, a_2)$ here means also that not both $a_1$ and $a_2$ are negative. We conclude this section by giving a bijection
$$
\text{Pow}(A): \text{Prim}(\mathcal{S}^{\mathbb{F}_2^A - \{(0, \ldots, 0)\}}) \to \mathcal{S}^A,
$$
which sends a vector $(v_B)_{\emptyset \neq B \subseteq A}$ to
$$
\text{Pow}(A)((v_B)_{\emptyset \neq B \subseteq A}) := \left(\prod_{j \in B} v_B\right)_{j \in A}.
$$
Here we implicitly identify the space $\mathbb{F}_2^A - \{(0, \ldots, 0)\}$ with the non-empty subsets of $A$. To see that this map is a bijection, we construct the inverse map. For a vector $(v_j)_{j \in A} \in \mathcal{S}^A$ and a prime $p$, we call the $p$-support of $(v_j)_{j \in A}$ to be the subset of $A$ consisting of those $j$ such that $p \mid v_j$. We also define the $-1$-support of $(v_j)_{j \in A}$ to be the subset of $A$ consisting of those $j$ such that $v_j < 0$. We now consider the assignment 
$$
(v_j)_{j \in A} \mapsto (w_B)_{\emptyset \neq B \subseteq A},
$$
where $w_B$ is the product of the elements $t \in \mathcal{P} \cup \{-1, 2\}$ such that $(v_j)_{j \in A}$ has $t$-support equal to $B$. It is readily verified that this assignment and $\text{Pow}(A)$ are inverse to each other.  

\section{The parametrization} 
\label{Parametrization}
Fix a positive integer $r$. We call a sequence of pairs
$$
\{(G_i, \theta_i)\}_{i \in [r]}
$$
an \emph{admissible sequence} if it obeys the following inductive rules. For each $i \in [r]$, we demand that $G_i$ is a $2$-group and that $\theta_i$ is a $2$-cocycle 
$$
\theta_i: G_{i - 1}^2 \to \mathbb{F}_2
$$
with $\theta_i(\text{id}, \text{id}) = 0$ and with $\theta_i$ vanishing in $H^2(G_{i-1},\mathbb{F}_2)$ if and only if $\theta_i$ is the zero map. Here $G_0$ is the trivial group by convention. Furthermore, we require that
$$
G_i = (\mathbb{F}_2 \times G_{i - 1}, *_{\theta_i})
$$
for all $i \in [r]$. Fix for the remainder of this section an admissible sequence $\{(G_i, \theta_i)\}_{i \in [r]}$. We denote in what follows 
$$
G := G_r. 
$$
The goal of this section is to construct a surjective map
$$
P_G: \text{Prim}(\mathcal{S}^{G - \{\text{id}\}}) \twoheadrightarrow \text{Epi}_{\text{top.gr.}}(\mathcal{G}_{\mathbb{Q}}^{\text{pro}-2}, G) \cup \{\bullet\},
$$
which restricts to a bijection between
$$
\text{Prim}(\mathcal{S}^{G - \{\text{id}\}})(\text{solv.}):=P_G^{-1}(\text{Epi}_{\text{top.gr.}}(\mathcal{G}_{\mathbb{Q}}^{\text{pro}-2}, G))
$$
and $\text{Epi}_{\text{top.gr.}}(\mathcal{G}_{\mathbb{Q}}^{\text{pro}-2}, G)$. Furthermore, we explain how to read the ramification data on the right hand side from the left hand side of this parametrization. Here we recall that $\text{Prim}(\mathcal{S}^{G - \{\text{id}\}})$ denotes the subset of $\mathcal{S}^{G - \{\text{id}\}}$ consisting of pairwise coprime integers, where we recall that coprimality of two integers here also excludes that they are both negative.  

We start by defining a map
$$
\tilde{P}_G:\mathcal{S}^{r} \to \text{Epi}_{\text{top.gr.}}(\mathcal{G}_{\mathbb{Q}}^{\text{pro}-2}, G) \cup \{\bullet\}
$$
as follows. Let $v := (v_1, \ldots, v_r)$ be an element of $\mathcal{S}^{r}$. First of all in case $\chi_{v_1}$ is the trivial character, we declare $\tilde{P}_G(v) = \bullet$. So assume that $\chi_{v_1}$ is non-trivial. Equivalently,
$$
\chi_{v_1} \in \text{Epi}_{\text{top.gr.}}(\mathcal{G}_{\mathbb{Q}}^{\text{pro}-2}, G_1).
$$
Hence 
$$
\chi_{v_1}^*(\theta_2) := \theta_2(\chi_{v_1}(\sigma), \chi_{v_1}(\tau))
$$ 
is now a $2$-cocycle on $G_{\mathbb{Q}}$. If $\chi_{v_1}^*(\theta_2)$ is non-trivial in $H^2(G_{\mathbb{Q}}, \mathbb{F}_2)$, then we declare $\tilde{P}_G(v) = \bullet$. Now assume that $\chi_{v_1}^*(\theta_2)$ is zero in $H^2(G_{\mathbb{Q}}, \mathbb{F}_2)$; we distinguish two cases. If $\theta_2$ is already a trivial $2$-cocycle on $G_1$, we have that $\phi(G_1, \chi_{v_1}^{*}(\theta_2)) = 0$ and
$$
(\chi_{v_2}, \chi_{v_1}) \in \text{Epi}_{\text{top.gr.}}(\mathcal{G}_{\mathbb{Q}}^{\text{pro}-2}, G_2)
$$
if and only if $\chi_{v_1}$ and $\chi_{v_2}$ are linearly dependent. In case $\chi_{v_1}$ and $\chi_{v_2}$ are linearly dependent, we set $\tilde{P}_G(v) = \bullet$; otherwise we continue. If instead $\theta_2$ is a non-trivial $2$-cocycle on $G_1$, we always have that
$$
(\phi(G_1, \chi_{v_1}^{*}(\theta_2)) + \chi_{v_2}, \chi_{v_1}) \in \text{Epi}_{\text{top.gr.}}(\mathcal{G}_{\mathbb{Q}}^{\text{pro}-2}, G_2)
$$
by Proposition \ref{defining normalized phi}.

Now we continue in this fashion inductively. At step $i < r$ we have either already assigned $v$ to $\bullet$, or we have obtained an epimorphism $\psi_i \in \text{Epi}_{\text{top.gr.}}(\mathcal{G}_{\mathbb{Q}}^{\text{pro}-2}, G_i)$. Then we get a $2$-cocycle $\psi_i^*(\theta_{i + 1})$, which gives a class in $H^2(G_{\mathbb{Q}}, \mathbb{F}_2)$. If this class is non-trivial in $H^2(G_{\mathbb{Q}}, \mathbb{F}_2)$, we send $v$ to $\bullet$. 

In case $\psi_i^*(\theta_{i + 1})$ is trivial in $H^2(G_{\mathbb{Q}}, \mathbb{F}_2)$, we distinguish two cases. If $\psi_i^*(\theta_{i + 1})$ is already trivial in $H^2(G_i, \mathbb{F}_2)$, we have that $\phi(G_i, \psi_i^*(\theta_{i + 1})) = 0$. Then we have that
\[
(\chi_{v_{i + 1}}, \psi_i) \in \text{Epi}_{\text{top.gr.}}(\mathcal{G}_{\mathbb{Q}}^{\text{pro}-2}, G_{i + 1})
\]
if and only if $\chi_{v_{i + 1}}$ is linearly independent from the characters $\chi_{v_j}$ with $j$ satisfying 
\[
\phi(G_{j - 1}, \psi_{j - 1}^*(\theta_j)) = 0.
\]
Indeed, observe that in this case we have $G_{i+1}=\mathbb{F}_2 \times G_i$, therefore, since we have surjectivity if and only if we have surjectivity modulo the Frattini, we need to have that $\chi_{v_{i+1}}$ is linearly independent from the quadratic characters coming from $G_i$. Thus our claim comes down to the claim that such quadratic characters are precisely spanned by the set
$$
\{\chi_{v_j}\}_{j \leq i: \theta_j = 0}.
$$
This is justified by the following simple observation. Let $H$ be a finite $2$-group and $\theta:H^2 \to \mathbb{F}_2$ a $2$-cocycle. Then $\theta$ is trivial in $H^2(H,\mathbb{F}_2)$ if and only if the dimension of $(\mathbb{F}_2 \times H,*_{\theta})$ modulo its Frattini subgroup is one larger than that of $H$ modulo its Frattini subgroup. To see the non-trivial direction observe that if one takes a quadratic character $\chi:(\mathbb{F}_2 \times H,*_{\theta}) \to \mathbb{F}_2$ that does not come from $H$, then its kernel gives a splitting of the sequence. 

If $\chi_{v_{i + 1}}$ is linearly dependent on these characters $\chi_{v_j}$, we send $v$ to $\bullet$ and otherwise we go to step $i + 1$.

Now suppose that $\theta_{i + 1}$ was a non-trivial class of $H^2(G_i, \mathbb{F}_2)$. Then we always obtain by means of Proposition \ref{defining normalized phi} a new epimorphism
$$
(\phi(G_i, \psi_i^*(\theta_{i + 1})) + \chi_{v_{i + 1}}, \psi_i) \in \text{Epi}_{\text{top.gr.}}(\mathcal{G}_{\mathbb{Q}}^{\text{pro}-2}, G_{i + 1}).
$$
Continuing in this fashion we obtain either $\bullet$ or an element of 
$$
\text{Epi}_{\text{top.gr.}}(\mathcal{G}_{\mathbb{Q}}^{\text{pro}-2}, G),
$$
which is by definition $\tilde{P}_G(v)$. We put
$$
P_G := \tilde{P}_G \circ \text{Pow}([r]). 
$$
We remark that $\phi(G_i, \theta)$ was defined only when $G_i$ was a Galois group. However, this small abuse of notation does not present any issues, since the epimorphism $\psi_i$ realizes the implicit identification between $G_i$ and the corresponding Galois group. 

Also we remark that in the construction of the map $P_G$ we have implicitly used that $G$ is set-theoretically defined to be $\mathbb{F}_2^r$, with the identity element being $(0, \ldots, 0)$, thanks to our convention on $2$-cocycles vanishing on $(\text{id}, \text{id})$. Hence there is no abuse of notation in invoking the map $\text{Pow}([r])$. 

\begin{proposition} 
\label{It is a parametrization}
The map 
$$
P_G: \textup{Prim}(\mathcal{S}^{G - \{\textup{id}\}}) \twoheadrightarrow \textup{Epi}_{\textup{top.gr.}}(\mathcal{G}_{\mathbb{Q}}^{\textup{pro}-2}, G) \cup \{\bullet\}
$$
is a surjection, which restricts to a bijection between $\textup{Prim}(\mathcal{S}^{G - \{\textup{id}\}})(\textup{solv.})$ and surjective homomorphisms $\textup{Epi}_{\textup{top.gr.}}(\mathcal{G}_{\mathbb{Q}}^{\textup{pro}-2}, G)$.
\end{proposition}

\begin{proof}
This is an immediate consequence of the construction of the map $P_G$ and Proposition \ref{defining normalized phi}. 
\end{proof}

The parametrization $P_G$ allows us to read off very neatly the image of the topological generators $\{\sigma_p\}_{p \in \mathcal{P}}$ from the tuples of squarefree integers. 

\begin{proposition} 
\label{Reading off inertia}
Let $(v_g)_{g \in G - \{\textup{id}\}}$ be an element of $\textup{Prim}(\mathcal{S}^{G - \{\textup{id}\}})(\textup{solv.})$. Let $p \in \mathcal{P}$. If $p \mid v_{g_0}$ for a (necessarily) unique $g_0 \in G - \{\textup{id}\}$ then
$$
P_G((v_g)_{g \in G - \{\textup{id}\}})(\sigma_p) = g_0.
$$
If $p$ does not divide any of the elements of the vector $(v_g)_{g \in G - \{\textup{id}\}}$, then 
$$
P_G((v_g)_{g \in G - \{\textup{id}\}})(\sigma_p)=\textup{id}.
$$
\end{proposition}

\begin{proof}
Write $(w_j)_{j \in [r]}$ for the image of $(v_g)_{g \in G - \{\textup{id}\}}$ under $\text{Pow}([r])$. For the first part, consider step $i$ of the admissible sequence so that 
\[
G_i = (\FF_2 \times G_{i - 1}, \ast_{\theta_i}).
\]
Observe that $\sigma_p$ is sent to $1$ (after projecting on the $\FF_2$ component) if and only if $\chi_{w_i}$ ramifies at $p$. Indeed, this follows from the normalization imposed on the $1$-cochains $\phi$ constructed in the parametrization. Therefore $\sigma_p$ is sent to the $p$-support of $(w_j)_{j \in [r]}$, which by construction of the map $P_G$ is also the unique $v_{g_0}$ divisible by $p$. 

For the second part, observe that if $p$ does not divide any of the $v_g$, then at every step of the admissible sequence $\sigma_p$ is mapped to $0$. Hence $\sigma_p$ is sent to $0$, which is by construction the identity element of $G$.  
\end{proof}

We next read off the value of the discriminant under the bijection $P_G$. For any continuous homomorphism $\psi$ of $G_{\mathbb{Q}}$ with values in some finite group, we denote by $\text{Disc}(\psi)$ the absolute discriminant of the corresponding extension of $\mathbb{Q}$.  We write $\text{odd}(n)$ for the largest, positive, odd divisor of an integer $n$.

\begin{proposition} 
\label{reading off discriminants}
Let $(v_g)_{g \in G - \{\textup{id}\}}$ be an element of $\textup{Prim}(\mathcal{S}^{G - \{\textup{id}\}})(\textup{solv.})$. Then
$$
\textup{odd}(\textup{Disc}(P_G((v_g)_{g \in G - \{\textup{id}\}}))) = \textup{odd} \left(\prod_{g \in G - \{\textup{id}\}} |v_g|^{\#G \cdot (1-\frac{1}{\#\langle g \rangle})} \right).
$$
\end{proposition}

\begin{proof}
We need to show that for each $p \in \mathcal{P}$ the $p$-adic valuation of $\textup{Disc}(P_G((v_g)_{g \in G - \{\textup{id}\}}))$ matches with the $p$-adic valuation of the right hand side. Since $p$ is odd and since $G$ is a $2$-group, we know that $p$ is a tame prime. Therefore we conclude that
$$
v_p(\textup{Disc}(P_G((v_g)_{g \in G - \{\textup{id}\}}))) = \frac{\#G}{\# \langle P_G((v_g)_{g \in G - \{\textup{id}\}})(\sigma_p) \rangle} \cdot \big(\# \langle P_G((v_g)_{g \in G - \{\textup{id}\}})(\sigma_p) \rangle-1 \big).
$$
Thanks to Proposition \ref{Reading off inertia} we deduce that $P_G((v_g)_{g \in G - \{\textup{id}\}})(\sigma_p)$ is trivial in case $p$ does not divide any $v_g$ and equals $g_0$ in case $p$ divides $v_{g_0}$. This is precisely the desired conclusion. 
\end{proof}

\section{The upper bound}
\label{sUpper}
For any finite $2$-group $G$ we denote by $\text{Inv}(G)$ the subset of $G$ consisting of involutions, that is, elements $g \in G - \{\text{id}\}$ with $g^2=\text{id}$. In this section we use the parametrization of Section \ref{Parametrization} to establish the following upper bound. 

\begin{theorem} 
\label{Main1}
Let $G$ be a non-trivial, finite $2$-group. Then there exists a constant $c \in \mathbb{R}_{>0}$ such that
$$
\#\{\psi \in \textup{Epi}_{\textup{top.gr.}}(G_{\mathbb{Q}}, G): \textup{Disc}(\psi) \leq X\} \leq c \cdot X^{\frac{2}{\#G}} \cdot \log(X)^{\#\textup{Inv}(G) - 1}
$$
for all $X \in \mathbb{R}_{>2}$. 
\end{theorem}

\begin{proof}
Write
\[
e_g := \#G\left(1-\frac{1}{\#\langle g \rangle} \right).
\]
Thanks to Proposition \ref{reading off discriminants}, it suffices to bound
\[
\#\left\{(v_g)_{g \in G - \{\text{id}\}} \in \text{Prim}(\mathcal{S}^{G - \{\text{id}\}}): \prod_{g \in G - \{\text{id}\}} |v_g|^{e_g} \leq X\right\}.
\]
Indeed, the $2$-adic valuation of the discriminant is at least $e_g$ in case $2 \mid v_g$. The above set has size at most
\[
\ll_G \sum_{\prod_{g \in G - \{\text{id}\}} v_g^{e_g} \leq X} 1,
\]
where the $v_g$ are positive squarefree integers that are pairwise coprime. We pull out the variables $v_g$ for which the order of $g$ is greater than $2$. Then the above sum becomes
\begin{align}
\label{eSquarefreeCount}
\sum_{\substack{\prod_{g \in G - \text{Inv}(G) - \{\text{id}\}} v_g^{e_g} \leq X \\ v_g \text{ positive squarefree} \\ v_g \text{ pairwise coprime}}} \sum_{\substack{\prod_{h \in \text{Inv}(G)} v_h \leq \frac{X^{\frac{2}{\# G}}}{\prod_{g \in G - \text{Inv}(G) - \{\text{id}\}} v_g^{2e_g/\# G}} \\ v_h \text{ positive squarefree} \\ v_h \text{ coprime to } v_g \\ v_h \text{ pairwise coprime}}} 1.
\end{align}
In the inner sum we drop the condition that $v_h$ is coprime to the $v_g$ with $g \in G - \text{Inv}(G) - \{\text{id}\}$, while in the outer sum we drop the condition that the $v_g$ are pairwise coprime. However, we keep the condition that the $v_h$ are pairwise coprime. Then the inner sum is bounded by
\[
\sum_{1 \leq d \leq \frac{X^{\frac{2}{\# G}}}{\prod_{g \in G - \text{Inv}(G) - \{\text{id}\}} v_g^{2e_g/\# G}}} \mu^2(d) \cdot \# \text{Inv}(G)^{\omega(d)} \ll \frac{X^{\frac{2}{\#G}} \cdot \log(X)^{\#\text{Inv}(G)-1}}{\prod_{g \in G - \text{Inv}(G) - \{\text{id}\}} v_g^{2e_g/\# G}},
\]
which follows for example from \cite[Theorem 2.7]{MoVa} and \cite[Corollary 2.15]{MoVa}. Plugging this bound in equation (\ref{eSquarefreeCount}) gives the sum
\begin{multline}
\label{eUpperConverge}
X^{\frac{2}{\#G}} \cdot \log(X)^{\#\text{Inv}(G) - 1} \sum_{\prod_{g \in G - \text{Inv}(G) - \{\text{id}\}} v_g^{e_g} \leq X} \frac{1}{\prod_{g \in G - \text{Inv}(G) - \{\text{id}\}} v_g^{2e_g/\# G}} \ll \\
X^{\frac{2}{\#G}} \cdot \log(X)^{\#\text{Inv}(G) - 1} \prod_{g \in G - \text{Inv}(G) - \{\text{id}\}} \left(\sum_{v_g^{e_g} \leq X} \frac{1}{v_g^{2e_g/\# G}} \right).
\end{multline}
Each inner sum converges since $2e_g/\#G > 1$ for $g \in G - \text{Inv}(G) - \{\text{id}\}$, which completes the proof of the theorem.
\end{proof}

Let $K$ be a number field. The upper bound in Theorem \ref{mt3: Upper bound general}, which is the natural generalization of Theorem \ref{Main1} to number fields and arbitrary nilpotent groups, is always at least as good as the upper bound appearing in \cite{KlunersUpper}, where an upper bound of the shape
\[
c_{G, K} \cdot X^{a(G)} \cdot \log(X)^{d(G, K) - 1}
\]
is established. For $2$-groups we can compute $d(G, K)$ as follows. Take a refinement $R$ of the upper central series of $G$
\[
\{\text{id}\} = G_r \subseteq G_{r - 1} \subseteq \dots \subseteq G_1 \subseteq G_0 = G
\]
such that each quotient is of size $2$ (we caution the reader that these $G_i$ are closely related, but different than the ones in the definition of admissible sequence). Then define $A_i := G_{i - 1} - G_i$ for $1 \leq i \leq r$ and
\[
d(R, K) = \sum_{\substack{1 \leq i \leq r \\ A_i \cap \text{Inv}(G) \neq \emptyset}} |A_i|.
\]
Now $d(G, K)$ is simply the minimum of $d(R, K)$ over all refinements $R$. We conclude this section by showing that for some groups $G$, Theorem \ref{Main1} provides a strictly better upper bound. In particular the next example is a group of size $64$, where \cite{KlunersUpper} always gives (i.e. for any choice of the filtration $\{G_i\}$ as above) at least $4$ logarithms more than Theorem \ref{Main1}. 

\begin{proposition} 
\label{smaller than kluners constant}
Let 
$$
G := \frac{\mathbb{F}_2[x_1,x_2]}{(x_1^2,x_2^2)} \rtimes \mathbb{F}_2^2,
$$
where $(1, 0)$ acts as multiplication by $1+x_1$ and $(0, 1)$ as multiplication by $1+x_2$. Then for any choice of the filtration $\{G_i\}$ as in \cite{KlunersUpper}, one has
$$
d(G, K) \geq \# \textup{Inv}(G) + 4.
$$
\end{proposition}

\begin{proof}
Let
$$
\{\text{id}\} = G_6 \subseteq G_5 \subseteq \dots \subseteq G_1 \subseteq G_0 = G
$$
be a refinement of the upper central series where each group is of index $2$ in the next one. We say that an element $g \in G$ has \emph{weight} $i$ in case $g \in G_i - G_{i + 1}$. We claim that there are at least $4$ elements of order $4$ with the same weight as an involution. 

We start by computing
\[
[G, G] = (x_1, x_2) \rtimes \{0\}, \quad \frac{G}{[G, G]} \cong \FF_2^3,
\]
where $(x_1, x_2)$ is the ideal generated by $x_1$ and $x_2$. From this we deduce that any surjective homomorphism
$$
\pi: G \twoheadrightarrow \mathbb{F}_2^2
$$ 
does not vanish identically on $\{0\} \rtimes \mathbb{F}_2^2$. Next observe that each of the $G_i$'s above are normal in $G$. Therefore, since $[G : G_2] = 4$, we conclude that $\frac{G}{G_2}$ is an abelian group, which implies that 
\begin{align}
\label{eG2GG}
G_2 \supseteq [G, G] =  (x_1, x_2) \rtimes \{0\}, \quad \frac{G}{G_2} \cong \FF_2^2.
\end{align}
Looking at the surjective homomorphism $G \twoheadrightarrow \frac{G}{G_2} \cong \FF_2^2$, we conclude that at least one of the involutions in the set
$$
\{(0,(1,0)),(0,(1,1)),(0,(0,1))\}
$$ 
has weight at most $1$. On the other hand equation (\ref{eG2GG}) shows that every commutator is necessarily of weight at least $2$. Hence at least one of the following three sets
$$
\{(x_2,(1,0)), (x_1+x_2,(1,0)), (x_2+x_1x_2,(1,0)), (x_1+x_2+x_1x_2,(1,0))\},
$$
$$
\{(x_2,(1,1)), (x_1,(1,1)), (x_1+x_1x_2,(1,1)), (x_2+x_1x_2,(1,1))\},
$$
$$
\{(x_1,(0,1)), (x_1+x_2,(0,1)), (x_1+x_2+x_1x_2,(0,1)), (x_1+x_1x_2,(0,1))\},
$$
consist entirely of elements with order $4$ and weight $1$. This shows that the set of elements having the same weight as an involution contains at least $4$ elements that are not involutions, which is precisely the desired conclusion. 
\end{proof}

\section{\texorpdfstring{Asymptotics for $G$ with $\text{Inv}(G) \subseteq Z(G)$}{Asymptotic for G with Inv(G) inside Z(G)}}
\label{sMainQ}
For a group $G$ we denote by $Z(G)$ the center of $G$. The goal of this section is to prove Theorem \ref{Main2}. 

Let $G$ be a finite $2$-group with $\text{Inv}(G) \subseteq Z(G)$. Then in particular $H(G):=\text{Inv}(G) \cup \{\text{id}\}$ is a vector space over $\mathbb{F}_2$, we denote by $h(G)$ its dimension. We filter $G$ by an admissible sequence 
$$
\{(G_i, \theta_i)\}_{i \in [r]}
$$
such that the kernel of the projection from $G = G_r$ to $G_{r - h(G)}$ coincides with $H(G)$. In other words $H(G)$ equals the subset of $\mathbb{F}_2^r$ of vectors with last $r-h(G)$ coordinates equal to $0$. 

Let us denote by
$$
\text{Prim}(\mathcal{S}^{G-H(G)})(\text{solv.})
$$
the subset of $\text{Prim}(\mathcal{S}^{G-H(G)})$ consisting of elements that appear as the last $r-h(G)$ coordinate of a vector in $\text{Prim}(\mathcal{S}^{G - \{\text{id}\}})(\text{solv.})$. The following proposition shows that once the set $\text{Prim}(\mathcal{S}^{G-H(G)})(\text{solv.})$ is given as input, then the set $\text{Prim}(\mathcal{S}^{G - \{\text{id}\}})(\text{solv.})$ admits the following relatively straightforward structure. In what follows we define the set $\textup{Prim}(\mathcal{S}^{G - \{\textup{id}\}})^{\circ}$ to be the set of vectors $(v_g)_{g \in G-\{\text{id}\}} \in \textup{Prim}(\mathcal{S}^{G - \{\textup{id}\}})$ such that $v_h \neq 1$ for each $h \in H(G)$. 

\begin{proposition} 
\label{Product structure}
We have
$$
\textup{Prim}(\mathcal{S}^{G - \{\textup{id}\}})(\textup{solv.}) \supseteq (\textup{Prim}(\mathcal{S}^{H(G) - \{\textup{id}\}}) \times \textup{Prim}(\mathcal{S}^{G - H(G)})(\textup{solv.})) \cap \textup{Prim}(\mathcal{S}^{G - \{\textup{id}\}})^{\circ}
$$
and furthermore
$$
\textup{Prim}(\mathcal{S}^{G - \{\textup{id}\}})(\textup{solv.}) \subseteq (\textup{Prim}(\mathcal{S}^{H(G) - \{\textup{id}\}}) \times \textup{Prim}(\mathcal{S}^{G - H(G)})(\textup{solv.})) \cap \textup{Prim}(\mathcal{S}^{G - \{\textup{id}\}}).
$$
\end{proposition}

\begin{remark}
In case $G = H(G)$, the set $\textup{Prim}(\mathcal{S}^{G - H(G)})(\textup{solv.})$ is by definition the one element set containing the empty tuple, so that
\[
\textup{Prim}(\mathcal{S}^{H(G) - \{\textup{id}\}}) \times \textup{Prim}(\mathcal{S}^{G - H(G)})(\textup{solv.}) = \textup{Prim}(\mathcal{S}^{H(G) - \{\textup{id}\}}).
\]
\end{remark}

\begin{proof}
Let $\mathcal{G}$ be any profinite group and let $\psi \in \text{Hom}_{\text{top.gr.}}(\mathcal{G},G)$ such that $\pi_{\text{mod} \ H(g)} \circ \psi$ is in $\text{Epi}_{\text{top.gr.}}(\mathcal{G}, \frac{G}{H(G)})$. Let $\chi \in \text{Hom}_{\text{top.gr.}}(\mathcal{G}, H(G))$. Note that the assignment
$$
\chi \cdot \psi: \mathcal{G} \to G, \quad g \mapsto \chi(g)\cdot\psi(g)
$$
is an element of $\text{Hom}_{\text{top.gr.}}(\mathcal{G}, G)$. Indeed, this map is clearly continuous and furthermore
\begin{align*}
\chi(g_1g_2) \psi(g_1g_2) &= \chi(g_1)\chi(g_2)\psi(g_1)\psi(g_2) =\chi(g_1)\psi(g_1)\chi(g_2)\psi(g_2) \\
&= (\chi \cdot \psi)(g_1)(\chi \cdot \psi)(g_2).
\end{align*}
Here the second equality uses that $H(G) \subseteq Z(G)$. 

Next observe that $\psi$ and $\chi$ induce natural maps $\psi^\ast: \text{Hom}(G, \FF_2) \rightarrow \text{Hom}_{\text{top.gr.}}(\mathcal{G}, \mathbb{F}_2)$ and $\chi^\ast: \text{Hom}(H(G), \mathbb{F}_2) \rightarrow \text{Hom}_{\text{top.gr.}}(\mathcal{G}, \mathbb{F}_2)$. We define $V$ to be the image of $\psi^\ast$ and $W$ to be the image of $\chi^\ast$, so that $V$ and $W$ are naturally $\FF_2$ vector spaces. We claim that if 
$$
V \cap W = \{0\}
$$
and if $\chi^\ast$ is injective, then
$$
\chi \cdot \psi \in \text{Epi}_{\text{top.gr.}}(\mathcal{G}, G).
$$
Since $G$ is a finite $2$-group, it is enough to show that $\chi \cdot \psi$ surjects modulo the Frattini subgroup or equivalently
\begin{align}
\label{eSurjTest}
\chi' \circ (\chi \cdot \psi) \neq 0
\end{align}
for any non-trivial character $\chi':G \to \mathbb{F}_2$. We aim to establish equation (\ref{eSurjTest}). Let us distinguish two cases. First assume that 
\[
\chi'(H(G)) = \{0\}.
\]
Then $\chi' \circ (\chi \cdot \psi)= \chi' \circ \psi$ and the claim follows from the assumption that $\pi_{\text{mod} \ H(G)} \circ \psi$ is surjective. Suppose now that $\chi'(H(G)) \neq \{0\}$. Since $\chi^*$ is injective it follows that $\chi' \circ \chi \neq 0$. On the other hand we also know that $V \cap W=\{0\}$. This means that there exists $g \in \mathcal{G}$ with $(\chi' \circ \psi)(g) = 0$ and $(\chi' \circ \chi)(g) = 1$. Therefore we find that
$$
\chi' \circ (\chi \cdot \psi)(g) = 1
$$
and we have established equation (\ref{eSurjTest}).

We apply the above to $\mathcal{G} = \mathcal{G}_{\mathbb{Q}}^{\text{pro}-2}$. Fix $(y_g)_{g \in G-\{\text{id}\}} \in \text{Prim}(\mathcal{S}^{G-\{\text{id}\}})(\text{solv.})$. Take now any vector
$$
(y'_g)_{g \in G-\{\text{id}\}} \in \textup{Prim}(\mathcal{S}^{G - \{\textup{id}\}})
$$
with $y'_g = y_g$ for each $g \in G-H(G)$ and $y'_h \neq 1$ for each $h \in H(G)$. We must show that
$$
(y'_g)_{g \in G-\{\text{id}\}} \in \text{Prim}(\mathcal{S}^{G-\{\text{id}\}})(\text{solv.}).
$$

Through the map $\text{Pow}([h(G)])$, we see that $(y_h)_{h \in H(G)-\{\text{id}\}}$ corresponds uniquely to a character $\chi \in \text{Hom}_{\text{top.gr.}}(\mathcal{G},H(G))$. We have that the map
$$
\psi := \chi \cdot P_G((y_g)_{g \in G-\{\text{id}\}})
$$
is an element of $\text{Hom}_{\text{top.gr.}}(\mathcal{G},G)$ such that $\pi_{\text{mod} \ H(G)} \circ \psi$ is surjective. 

Let $\chi'$ be the character from $\mathcal{G}$ to $H(G)$ corresponding to $(y'_h)_{h \in H(G)-\{\text{id}\}}$. Since $y'_h \neq 1$ for each $h \in H(G)-\{\text{id}\}$, it follows that $\chi'$ is surjective and hence $\chi'^\ast$ is injective. Furthermore, $V \cap W = \{0\}$ by construction of $\psi$ and our assumption that $y'_h \neq 1$ for $h \in H(G)-\{\text{id}\}$. Hence
$$
\chi' \cdot \psi \in \text{Epi}_{\text{top.gr.}}(\mathcal{G}, G),
$$
and furthermore
$$
P_G((y'_g)_{g \in G-\{\text{id}\}}) = \chi' \cdot \psi.
$$
This establishes the first part of the proposition. The second part is straightforward.
\end{proof}

We are now ready to show Theorem \ref{tMalle} in the special case of $2$-groups over $\Q$.

\begin{theorem} 
\label{Main2}
Let $G$ be a non-trivial $2$-group with $\textup{Inv}(G) \subseteq Z(G)$. Then there exists a constant $\alpha \in \mathbb{R}_{>0}$ such that 
$$
\#\{\psi \in \textup{Epi}_{\textup{top.gr.}}(G_{\mathbb{Q}}, G): \textup{Disc}(\psi) \leq X\} \sim \alpha \cdot X^{\frac{2}{\#G}} \cdot \log(X)^{\#\textup{Inv}(G) - 1}.
$$
\end{theorem}

\begin{proof}
Thanks to Proposition \ref{It is a parametrization} we have a bijection between the sets $\text{Epi}_{\text{top.gr.}}(G_{\mathbb{Q}}, G)$ and $\text{Prim}(\mathcal{S}^{G - \{\textup{id}\}})(\text{solv.})$. This allows us to define the discriminant $\text{Disc}(y)$ for any vector $y$ in the space $\text{Prim}(\mathcal{S}^{G - \{\textup{id}\}})(\text{solv.})$. We also recall that the odd part of $\text{Disc}(y)$ can be computed by an appeal to Proposition \ref{reading off discriminants}.

Now label the elements of $\text{Prim}(\mathcal{S}^{G - H(G)})(\text{solv.})$ as $x_1, x_2, x_3, \dots$ and write $L$ for the length of the sequence, where we allow $L$ to be infinite. Write $\pi$ for the natural projection map from $\text{Prim}(\mathcal{S}^{G - \{\textup{id}\}})(\text{solv.})$ to $\text{Prim}(\mathcal{S}^{G - H(G)})(\text{solv.})$. Then we have
\[
\#\{\psi \in \text{Epi}_{\text{top.gr.}}(G_{\mathbb{Q}}, G): \text{Disc}(\psi) \leq X\} = \sum_{i = 1}^L \sum_{\substack{y \in \text{Prim}(\mathcal{S}^{G - \{\textup{id}\}})(\text{solv.}) \\ \pi(y) = x_i \\ \text{Disc}(y) \leq X}} 1.
\]
We claim that for all $i$ there exists a constant $b_{G, x_i} > 0$ such that
\begin{align}
\label{eProjAsymptotic}
\sum_{\substack{y \in \text{Prim}(\mathcal{S}^{G - \{\textup{id}\}})(\text{solv.}) \\ \pi(y) = x_i \\ \text{Disc}(y) \leq X}} 1 \sim b_{G, x_i} \cdot X^{\frac{2}{\#G}} \cdot \log(X)^{\#\text{Inv}(G) - 1}.
\end{align}
To prove the claim, write $y = (y_g)_{g \in G - \{\textup{id}\}}$ and split the sum depending on the value of $y_h \bmod 8$ for each $h \in H(G) - \{\text{id}\}$. So let $\mathbf{a} = (a_h)_{h \in H(G) - \{\text{id}\}} \in (\Z/8\Z)^{H(G) - \{\text{id}\}}$. Then $G$, $x_i$ and $\mathbf{a}$ determine the restriction of $\psi$ to $I_2(2)$ and thus the $2$-adic valuation of the discriminant. Hence we can split equation (\ref{eProjAsymptotic}) in finitely many sums of the shape
\begin{align}
\label{eSplitmodulo8}
\sum_{\substack{y = (y_g)_{g \in G - \{\text{id}\}} \in \text{Prim}(\mathcal{S}^{G - \{\textup{id}\}})(\text{solv.}) \\ \pi(y) = x_i \\ y_h \equiv a_h \bmod 8 \text{ for } h \in H(G) - \{\text{id}\} \\ \prod_{h \in H(G) - \{\text{id}\}} y_h^{\#G/2} \leq C(x_i, \mathbf{a}) X}} 1,
\end{align}
where $C(x_i, \mathbf{a})$ is a constant. From Proposition \ref{Product structure} we can lower bound equation (\ref{eSplitmodulo8}) by
\begin{align}
\label{eUpperSplit}
\sum_{\substack{(y_h)_{h \in H(G) - \{\text{id}\}} \\ y_h \neq 1 \\ y_h \text{ coprime to } y_g \text{ for } g \in G - H(G) \\ y_h \text{ squarefree and pairwise coprime} \\ y_h \equiv a_h \bmod 8 \\ \prod_{h \in H(G) - \{\text{id}\}} y_h^{\#G/2} \leq C(x_i, \mathbf{a}) X}} 1
\end{align}
and upper bound equation (\ref{eSplitmodulo8}) by
\begin{align}
\label{eLowerSplit}
\sum_{\substack{(y_h)_{h \in H(G) - \{\text{id}\}} \\ y_h \text{ coprime to } y_g \text{ for } g \in G - H(G) \\ y_h \text{ squarefree and pairwise coprime} \\ y_h \equiv a_h \bmod 8 \\ \prod_{h \in H(G) - \{\text{id}\}} y_h^{\#G/2} \leq C(x_i, \mathbf{a}) X}} 1.
\end{align}
Basic analytic number theory allows one to give matching asymptotic formulas for equation (\ref{eUpperSplit}) and equation (\ref{eLowerSplit}), which together imply the claimed equation (\ref{eProjAsymptotic}).

Let us now recall the statement of Tannery's theorem, which, in modern terms, is just the dominated convergence theorem on $\ell^1$. Let $f_i : \mathbb{Z}_{\geq 1} \rightarrow \mathbb{C}$ be functions and let
\[
S(X) = \sum_{i = 1}^\infty f_i(X)
\]
and suppose that $\lim_{X \rightarrow \infty} f_i(X) = b_i$. If $|f_i(X)| \leq M_i$ and
\begin{align}
\label{eBound}
\sum_{i = 1}^\infty M_i < \infty,
\end{align}
then $\lim_{X \rightarrow \infty} S(X)$ exists, $\sum_{i = 1}^\infty b_i$ converges absolutely and
\begin{align}
\label{eTannery}
\lim_{X \rightarrow \infty} S(X) = \sum_{i = 1}^\infty b_i.
\end{align}
We apply Tannery's theorem with
\[
f_i(X) := \frac{1}{X^{\frac{2}{\#G}} \cdot \log(X)^{\#\text{Inv}(G) - 1}} \cdot \sum_{\substack{y \in \text{Prim}(\mathcal{S}^{G - \{\textup{id}\}})(\text{solv.}) \\ \pi(y) = x_i \\ \text{Disc}(y) \leq X}} 1,
\]
so that $\lim_{X \rightarrow \infty} f_i(X) = b_{G, x_i}$ by equation (\ref{eProjAsymptotic}). It follows from equation (\ref{eUpperConverge}) of Theorem \ref{Main1} that equation (\ref{eBound}) is satisfied with
\[
M_i = \frac{C_G}{\prod_{g \in G - H(G)} y_g^{2e_g/\# G}},
\]
where $C_G$ is a constant. Then equation (\ref{eTannery}) shows that
\[
\lim_{X \rightarrow \infty} \frac{1}{X^{\frac{2}{\#G}} \cdot \log(X)^{\#\text{Inv}(G) - 1}} \cdot \sum_{\substack{y \in \text{Prim}(\mathcal{S}^{G - \{\textup{id}\}})(\text{solv.}) \\ \text{Disc}(y) \leq X}} 1 = \lim_{X \rightarrow \infty} S(X) = \sum_{i = 1}^L b_{G, x_i}.
\]
Since $b_{G, x_i} > 0$ and the sum is non-empty by Shafarevich's theorem, the theorem follows.
\end{proof}

\section{Arbitrary nilpotent groups and number fields}
\label{sMain}
In our first subsection we restrict ourselves to the case that $G$ is a finite $l$-group, but we work with an arbitrary number field $K$. This is then extended to general nilpotent $G$ in the second subsection.

\subsection{\texorpdfstring{Finite $l$-groups}{Finite l-groups}}
Our first goal is to extend Proposition \ref{It is a parametrization} to general number fields and general nilpotent groups. In what follows we shall use the notation introduced in Section \ref{sPrelim}. Let $K$ be a number field, picked inside $\mathbb{Q}^{\text{sep}}$. Denote by $\Omega_K$ the set of all places of $K$. For each finite place $\mathfrak{q}$ in $\Omega_K$, lying above a rational prime $q$, the restriction of the map $i_q^{*}$ provides us with an inclusion
$$
i_{\mathfrak{q}}^{*}: G_{K_{\mathfrak{q}}} \to G_{K}.
$$
Denote by $k_{\mathfrak{q}}$ the residue field of $K$ at $\mathfrak{q}$. Write
$$
I_{\mathfrak{q}} := \text{ker}(G_{K_{\mathfrak{q}}} \to G_{k_{\mathfrak{q}}})
$$ 
for the inertia subgroup. Let now $l$ be any prime number. In what follows the group $\mathbb{F}_l$ will always be implicitly interpreted as a Galois module with trivial action, whenever the notation suggests an implicit action of a group on $\mathbb{F}_l$. We denote by $H_{\text{unr}}^1(G_{K_{\mathfrak{q}}}, \mathbb{F}_l)$ the image, via inflation, of $H^1(G_{k_{\mathfrak{q}}}, \mathbb{F}_l)$ in $H^1(G_{K_{\mathfrak{q}}}, \mathbb{F}_l)$.

For a subset $S \subseteq \Omega_K$, containing all the archimedean places of $K$, we consider the map
$$
\Phi_K(l, S): H^1(G_K, \mathbb{F}_l) \to \bigoplus_{\mathfrak{q} \in \Omega_K-S} \frac{H^1(G_{K_{\mathfrak{q}}}, \mathbb{F}_l)}{H_{\text{unr}}^1(G_{K_{\mathfrak{q}}}, \mathbb{F}_l)}.
$$
In case $S$ consists exactly of the archimedean places, we will denote the resulting map simply by $\Phi_K(l)$. We start by recalling the following classical fact.

\begin{proposition} 
\label{finite kernel cokernel}
The abelian groups $\textup{ker}(\Phi_K(l))$ and $\textup{coker}(\Phi_K(l))$ are finite.
\end{proposition}

\begin{proof}
By class field theory we have a canonical identification
$$
\text{ker}(\Phi_K(l)) = \textup{Cl}(K, m_{\infty})^{\vee}[l],
$$
where $m_{\infty}$ is the modulus consisting of all archimedean places. Therefore we have that
$$
\#\text{ker}(\Phi_K(l))=\#\textup{Cl}(K,m_{\infty})^{\vee}[l] \leq l^{[K:\mathbb{Q}]} \cdot \#\text{Cl}(K)^{\vee}[l], 
$$
where the first factor $l^{[K:\mathbb{Q}]}$ can be dropped when $l$ is odd. Therefore the finiteness of $\text{ker}(\Phi_K(l))$ follows from the finiteness of $\text{Cl}(K)$. The finiteness of $\text{coker}(\Phi_K(l))$ is established in \cite[Theorem 5, eq. (14) and (16)]{Shafarevich}.
\end{proof}

The following important fact falls as an easy consequence of Proposition \ref{finite kernel cokernel}.

\begin{proposition} 
\label{generalized cleaners}
There exists a finite set of places $S$, containing all archimedean places, such that $\Phi_K(l, S)$ is surjective. 
\end{proposition}

\begin{proof}
Write $A$ for the finite subset of archimedean places of $\Omega_K$. Thanks to Proposition \ref{finite kernel cokernel} we can find a finite set of places $A \subseteq S \subseteq \Omega_K$ such that the natural map
$$
\bigoplus_{\mathfrak{q} \in S - A} \frac{H^1(G_{K_\mathfrak{q}}, \mathbb{F}_l)}{H_{\text{unr}}^1(G_{K_\mathfrak{q}}, \mathbb{F}_l)} \to \text{coker}(\Phi_K(l))
$$
is surjective. This implies that for any vector 
\[
v \in \bigoplus_{\mathfrak{q} \in \Omega_K-S} \frac{H^1(G_{K_\mathfrak{q}}, \mathbb{F}_l)}{H_{\text{unr}}^1(G_{K_\mathfrak{q}}, \mathbb{F}_l)}
\]
we can find $w \in \bigoplus_{\mathfrak{q} \in S - A} \frac{H^1(G_{K_\mathfrak{q}}, \mathbb{F}_l)}{H_{\text{unr}}^1(G_{K_\mathfrak{q}}, \mathbb{F}_l)}$ and a global character $\chi \in H^1(G_K, \mathbb{F}_l)$ such that
$$
\Phi_K(l)(\chi) = v + w.
$$
Therefore, since $w$ is entirely supported in $S$, we conclude that
$$
\Phi_K(l, S)(\chi) = v.
$$
Hence we have shown that the map $\Phi_K(l, S)$ is surjective with this choice of $S$, which is precisely the desired conclusion. 
\end{proof}

Of course if a set $S$ as in Proposition \ref{generalized cleaners} works, then any larger set works as well. We fix once and for all a finite set $S_{\text{clean}}(l)$ as in Proposition \ref{generalized cleaners}, making sure that it also contains all places above $l$. We denote by $\widetilde{\Omega}_K(l)$ the subset of $\Omega_K - S_{\text{clean}}(l)$ with
\begin{align}
\label{eLocalChar}
\frac{H^1(G_{K_{\mathfrak{q}}}, \mathbb{F}_l)}{H_{\text{unr}}^1(G_{K_{\mathfrak{q}}}, \mathbb{F}_l)} \neq 0.
\end{align}
For $\mathfrak{q} \in \Omega_K - S_{\text{clean}}(l)$, it follows from local class field theory that equation (\ref{eLocalChar}) is equivalent to the condition
$$
\# \left(\mathcal{O}_K/\mathfrak{q} \right) \equiv 1 \bmod l.
$$
Write $K^{\text{pro}-l}$ for the compositum of all finite Galois extensions $L$ of $K$ with $[L : K]$ a power of $l$.

\begin{proposition} 
\label{only with local roots of unity}
Let $\mathfrak{q} \in \Omega_K$ be a finite place coprime to $l$ such that $\# \left(\mathcal{O}_K/\mathfrak{q} \right) \not \equiv 1 \bmod l$. Then $\mathfrak{q}$ is unramified in every finite extension $L/K$ inside $K^{\textup{pro}-l}$.
\end{proposition}

\begin{proof}
It suffices to prove the proposition locally at $\mathfrak{q}$. Take a positive integer $f$ and let $K_{\mathfrak{q}^f}$ be the unique unramified extension of $K_{\mathfrak{q}}$ of degree equal to $f$ with residue field denoted by $k_{\mathfrak{q}^f}$. Then if we have a cyclic totally ramified degree $l$ extension of $K_{\mathfrak{q}^f}$ it follows from local class field theory and the fact that $\mathfrak{q}$ is coprime to $l$
\[
\#k_{\mathfrak{q}^f} = \#k_{\mathfrak{q}}^f \equiv 1 \bmod l.
\]
Now, if $f$ is a power of $l$ itself, we conclude that $\#k_{\mathfrak{q}}$ is already congruent to $1$ modulo $l$ contrary to our assumption that $\# \left(\mathcal{O}_K/\mathfrak{q} \right) \not \equiv 1 \bmod l$.
\end{proof}

We remark that Proposition \ref{only with local roots of unity} certainly applies to any place $\mathfrak{q} \in \Omega_K - S_{\text{clean}}(l) - \widetilde{\Omega}_K(l)$. Let $\mathfrak{q} \in \widetilde{\Omega}_K(l)$. Thanks to Proposition \ref{generalized cleaners}, there exists a character 
$$
\chi_{\mathfrak{q}} \in H^1(G_K, \mathbb{F}_l)
$$
such that $\Phi_K(l, S_{\text{clean}}(l))(\chi_{\mathfrak{q}})$ has non-trivial coordinate precisely at $\mathfrak{q}$ and at no other places in $\Omega_K-S_{\text{clean}}(l)$. Fix once and for all such a choice of $\chi_{\mathfrak{q}}$ for each $\mathfrak{q} \in \widetilde{\Omega}_K(l)$. By construction
$$
\{\chi_{\mathfrak{q}} \}_{\mathfrak{q} \in \widetilde{\Omega}_K(l)}
$$
is a linearly independent set. Furthermore by Proposition \ref{finite kernel cokernel} we obtain that the subspace
$$
\langle \{\chi_{\mathfrak{q}} \}_{\mathfrak{q} \in \widetilde{\Omega}_K(l)} \rangle \subseteq H^1(G_K, \mathbb{F}_l)
$$
has finite index. Additionally, there exists a positive integer $t$ and a basis
$$
J := \{\chi_i\}_{i=1}^t \subseteq \text{ker}(\Phi_K(l, S_{\text{clean}}(l)))
$$
such that $J \cup \{\chi_{\mathfrak{q}} \}_{\mathfrak{q} \in \widetilde{\Omega}_K(l)}$ is a \emph{basis} of $H^1(G_K, \mathbb{F}_l)$. Fix once and for all such a choice of $J$. We denote by
$$
\mathcal{B}(K, l) := J \cup \{\chi_{\mathfrak{q}} \}_{\mathfrak{q}}
$$
this fixed choice of a basis. Put
$$
\mathcal{G}_K^{\text{pro}-l}:=\text{Gal}(K^{\text{pro}-l}/K).
$$
For each finite place $\mathfrak{q} \in \Omega_K$ we denote by
$$
I_{\mathfrak{q}}(l) := \text{proj}(G_K \to \mathcal{G}_K^{\text{pro}-l}) \circ i_q^{*}(I_{\mathfrak{q}}).
$$
We have the following basic fact. 

\begin{proposition} 
\label{general tame inertia procyclic: in general}
The group $I_{\mathfrak{q}}(l)$ is pro-cyclic for each finite place $\mathfrak{q} \in \Omega_K$ coprime to $l$.
\end{proposition}

\begin{proof}
Let $L$ be a non-archimedean local field of characteristic $0$ and write $p$ for its residue characteristic. Let $d$ be a positive integer coprime to $p$. Then every finite totally ramified extension of $L$ of degree equal to $d$ can be obtained as $L(\sqrt[d]{\pi})$ for $\pi$ a uniformizer of $L$. For an elementary proof of this well-known fact see \cite[Proposition A.5]{KP1}. Applying this repeatedly to all finite unramified extensions of $L$, we conclude that
$$
\frac{I_L}{I_L^{\text{wild}}}
$$
is a pro-cyclic group, where $I_L$ is the inertia subgroup and $I_L^{\text{wild}}$ is the wild inertia subgroup. Since $\mathfrak{q}$ is coprime to $l$, we conclude that $I_{\mathfrak{q}}(l)$ is a quotient of the pro-cyclic group $\frac{I_{\mathfrak{q}}}{I_{\mathfrak{q}}^{\text{wild}}}$, which gives in particular the desired conclusion. 
\end{proof}

\begin{remark}
Since the group $I_{\mathfrak{q}}(l)$ is a pro-cyclic pro-$l$ group, it is either isomorphic to a finite group of order a power of $l$ or isomorphic to $\mathbb{Z}_l$. In case a primitive $l$-th root of unity $\zeta_l$ is in $K$, then we claim that $I_{\mathfrak{q}}(l)$ is infinite. Indeed, we have the infinitely ramified subextension
$$
K(\sqrt[\infty]{1}, \sqrt[\infty]{\alpha})
$$
of $K^{\textup{pro}-l}/K$ (which is contained in $K^{\textup{pro}-l}/K$ exactly because $\zeta_l$ is in $K$) given by any $\alpha$ in $K$ with $v_{\mathfrak{q}}(\alpha)=1$. Therefore we conclude that if $K$ possesses a non-trivial $l$-th root of unity, then
$$
I_{\mathfrak{q}}(l) \simeq_{\textup{top.gr.}} \mathbb{Z}_l
$$
for every finite place $\mathfrak{q} \in \Omega_K$ coprime to $l$. Instead if $\zeta_l$ is not in $K$, we observe that Proposition \ref{only with local roots of unity} shows that the group $I_{\mathfrak{q}}(l)$ is trivial in case $\mathfrak{q}$ is a finite place of $\Omega_K$ that is coprime to $l$ and satisfies $\left(\mathcal{O}_K/\mathfrak{q} \right) \not \equiv 1 \bmod l$.
\end{remark}

We fix once and for all a topological generator $\sigma_{\mathfrak{q}}$ of $I_{\mathfrak{q}}(l)$ for all $\mathfrak{q} \in \widetilde{\Omega}_K(l)$ in the following manner. Observe that $\chi_{\mathfrak{q}}(\sigma_{\mathfrak{q}}) \neq 0$ for any topological generator of $I_{\mathfrak{q}}(l)$, since the character $\chi_{\mathfrak{q}}$ ramifies at $\mathfrak{q}$. Hence we can always pick a generator $\sigma_{\mathfrak{q}}$ with the normalization $\chi_{\mathfrak{q}}(\sigma_{\mathfrak{q}})=1$. We make such a choice of $\sigma_{\mathfrak{q}}$ once and for all. In case $\mathfrak{q} \in \Omega_K - S_{\text{clean}}(l) - \widetilde{\Omega}_K(l)$, then the group $I_{\mathfrak{q}}(l)$ is trivial by Proposition \ref{only with local roots of unity}, and we declare $\sigma_{\mathfrak{q}}:=\text{id}$. 

Now it follows by construction that
$$
\chi(\sigma_{\mathfrak{q}}) = \delta_{\chi_{\mathfrak{q}}}(\chi)
$$
for each $\chi \in \mathcal{B}(K, l)$ and $\mathfrak{q} \in \widetilde{\Omega}_K(l)$, where $\delta$ denotes the Kronecker delta function. Therefore we can complete the set 
$$
\{\sigma_{\mathfrak{q}} \}_{\mathfrak{q} \in \widetilde{\Omega}_K(l)}
$$
to a minimal set of generators
$$
\{\sigma_i\} \cup \{\sigma_{\mathfrak{q}} \}_{\mathfrak{q} \in \widetilde{\Omega}_K(l)},
$$
which is dual to the basis $\mathcal{B}(K, l)$, i.e. 
\begin{align*}
&\chi_i(\sigma_{\mathfrak{q}}) = 0 = \chi_{\mathfrak{q}}(\sigma_i) \quad && \text{ for each } i \in [t] \text{ and } \mathfrak{q} \in \widetilde{\Omega}_K(l) \\
&\chi_i(\sigma_j) = \delta_i(j) \quad&& \text{ for each } i, j \in [t] \\
&\chi_{\mathfrak{q}}(\sigma_{\mathfrak{q}'}) = \delta_{\mathfrak{q}}(\mathfrak{q}') \quad && \text{ for every } \mathfrak{q}, \mathfrak{q}' \in \widetilde{\Omega}_K(l).
\end{align*}
We denote by
$$
\mathcal{B}^{\vee}(K, l) := \{\sigma_i\}_{i=1}^{t} \cup \{\sigma_{\mathfrak{q}} \}_{\mathfrak{q} \in \widetilde{\Omega}_K(l)}
$$
this special set of topological generators of $\mathcal{G}_K^{\text{pro}-l}$. 

Let $L/K$ be a finite Galois extension inside $K^{\text{pro}-l}$ with Galois group $G := \text{Gal}(L/K)$. Take a $2$-cocycle $\theta$
$$
\theta: G^2 \to \mathbb{F}_l
$$
with the requirement that $\theta(\text{id}, \text{id}) = 0$. By the same argument as before, every class in $H^2(G, \mathbb{F}_l)$ can be represented by such a $2$-cocycle $\theta$. Consider the group
$$
(\mathbb{F}_l \times G,*_{\theta}),
$$
where the group law is
$$
(a_1, g_2)*_{\theta}(a_2, g_2)=(a_1 + a_2 + \theta(g_1, g_2), g_1g_2).
$$
Our assumption on $\theta$ ensures that $(0, \text{id})$ is the trivial element of $(\mathbb{F}_l \times G,*_{\theta})$. We have the following generalization of Proposition \ref{defining normalized phi}.

\begin{proposition} 
\label{defining normalized phi: in general}
Let $l$ be a prime number. Let $K$ be a number field and let $L$ be an extension with $G = \Gal(L/K)$ a finite $l$-group. Suppose that $\theta$ is non-trivial in $H^2(G, \mathbb{F}_l)$. \\
$(a)$ The natural projection map $\pi: G_K \twoheadrightarrow G$ can be lifted to a surjective homomorphism
$$
\psi: G_K \to (\mathbb{F}_l \times G,*_{\theta})
$$
if and only if $\theta$ is trivial in $H^2(G_{K_v}, \mathbb{F}_l)$ for each place $v$ that ramifies in $L/K$. Moreover, if $\psi$ is a lift, then the $\mathbb{F}_l$-coordinate of $\psi$ is a continuous $1$-cochain $\phi(\psi): G_K \rightarrow \FF_l$ with
$$
\diff(-\phi(\psi)) = \theta.
$$
Conversely, given any such continuous $1$-cochain $\phi: G_K \rightarrow \FF_l$ with $\diff(-\phi) = \theta$, the assignment
$$
\psi(\phi)(g) = (\phi(g), \pi(g))
$$
is an epimorphism lifting the canonical projection $\pi: G_K \to G$ to an epimorphism $G_K \to (\mathbb{F}_l \times G,*_{\theta})$. The two assignments are mutual inverses. \\
$(b)$ In case one has a lift $\psi$ as in part $(a)$, then there is a unique one satisfying
$$
\phi(\psi)(\sigma) = 0 \textup{ for all } \sigma \in \mathcal{B}^{\vee}(K, l).
$$
\end{proposition}

\begin{proof}
We start with part (a). We claim that a map $\phi(\psi): G_K \to \mathbb{F}_l$ is the first coordinate of a homomorphism
$$
\psi: G_K \twoheadrightarrow (\mathbb{F}_l \times G,*_{\theta}), \quad g \mapsto (\phi(\psi)(g), \pi(g))
$$
if and only if
$$
\diff(-\phi(\psi)) = \theta.
$$
Indeed, since $\psi$ is a homomorphism, we obtain
$$
(\phi(\psi)(g_1g_2), \pi(g_1g_2)) = \psi(g_1g_2) = \psi(g_1) \psi(g_2) = (\phi(\psi)(g_1) + \phi(\psi)(g_2) + \theta(g_1, g_2), \pi(g_1g_2)),
$$
which is equivalent to
$$
\diff(-\phi(\psi))(g_1, g_2) = \phi(\psi)(g_1g_2) - \phi(\psi)(g_1) - \phi(\psi)(g_2) = \theta(g_1, g_2)
$$
as claimed. 

Now suppose that there exists $\phi(\psi)$ with $\diff(-\phi(\psi)) = \theta$. We claim that $(\phi(\psi), \pi)$ is surjective. Let us first show that all characters $(\mathbb{F}_l \times G,*_{\theta}) \rightarrow \mathbb{F}_l$ must come from $G$. If not, then the kernel of such a hypothetical character provides a splitting of $\theta$, which implies that $\theta$ is trivial contrary to our assumptions. Hence, since $\psi$ is surjective, the image of $(\phi(\psi), \pi)$ generates modulo the Frattini of $(\mathbb{F}_l \times G,*_{\theta})$, and therefore equals $(\mathbb{F}_l \times G,*_{\theta})$. 

Furthermore, we see that the lifting $\psi$ exists if and only if the inflation of $\theta$ to $H^2(G_K, \mathbb{F}_l)$ is trivial if and only if $\theta$ is trivial in $H^2(G_{K_v}, \mathbb{F}_l)$ for every place $v$ of $K$. Thanks to \cite[Proposition 4.4]{KP1}, the vanishing at the finite places unramified in $L/K$ is already guaranteed: notice that in \cite[Section 4]{KP1} the number $l$ is assumed to be an odd prime but Proposition \cite[Proposition 4.4]{KP1} also holds for $l=2$ with an identical proof. If$v$ is an archimedean complex place the vanishing is authomatic. If $v$ is an archimedean real place and the extension $L/K$ is unramified at $v$, then this means that for each place $w$ of $L$ above $v$ we have that $L_w=K_v$ and thus the embedding problem is locally trivial at $v$. This ends the proof of part $(a)$. 

We now prove part $(b)$. The uniqueness follows at once from part $(a)$ combined with the fact that $\mathcal{B}^{\vee}(K, l)$ is a system of topological generators for $\mathcal{G}_K^{\text{pro}-l}$. Indeed, an epimorphism $\psi$ as in part $(a)$ is entirely determined by its values on a set of topological generators. We next show the existence: here we will take advantage of the fact that $\mathcal{B}^{\vee}(K, l)$ is a minimal set of topological generators. Take a map $\phi: G_K \rightarrow \mathbb{F}_l$ satisfying 
$$
\diff(-\phi) = \theta. 
$$
The resulting epimorphism $\psi(\phi):\mathcal{G}_K^{\text{pro}-l} \twoheadrightarrow (\mathbb{F}_l \times G, *_{\theta})$ corresponds to a finite extension. As such we conclude that $\phi(\sigma_{\mathfrak{q}}) = 0$ for all but finitely many $\mathfrak{q} \in \widetilde{\Omega}_K(l)$. 
Therefore the sum
$$
\sum_{i = 1}^t \phi(\sigma_i) \cdot \chi_i + \sum_{\mathfrak{q} \in \widetilde{\Omega}_K(l)} \phi(\sigma_{\mathfrak{q}}) \cdot \chi_{\mathfrak{q}}
$$
is a well-defined element of $H^1(G_K, \mathbb{F}_l)$. Hence, since $\mathcal{B}(K, l)$ and $\mathcal{B}^{\vee}(K, l)$ are dual to each other, we obtain that
$$
\phi - \sum_{i = 1}^t \phi(\sigma_i) \cdot \chi_i - \sum_{\mathfrak{q} \in \widetilde{\Omega}_K(l)} \phi(\sigma_{\mathfrak{q}}) \cdot \chi_{\mathfrak{q}}
$$
vanishes at $\sigma_i$ for all $i \in [t]$ and at $\sigma_\mathfrak{q}$ for all $\mathfrak{q} \in \widetilde{\Omega}_K(l)$. This ends the proof of part $(b)$.  
\end{proof}

We denote the unique $1$-cochain as in part $(b)$ of Proposition \ref{defining normalized phi: in general} by $\phi(G, \theta)$. In case $\theta$ is trivial as a $2$-cocycle, then we choose $\phi(G, \theta) := 0$. With this choice, we see that $\phi(G, \theta)$ satisfies part $(b)$ of Proposition \ref{defining normalized phi: in general}, since the trivial character is the unique cyclic degree $l$ character vanishing at all $\sigma \in \mathcal{B}^{\vee}(K, l)$.

Denote by $\mathcal{S}_l$ the set $\{0,1\}^{[t]} \times \mathcal{S}_l'$, where $\mathcal{S}_l'$ is the set of squarefree integral ideals in $\mathcal{O}_K$ entirely supported in $\widetilde{\Omega}_K(l)$. To an element $(T, \mathfrak{b}) \in \mathcal{S}_l$ we attach the character
$$
\chi_{(T, \mathfrak{b})} := \sum_{\substack{i \in [t] \\ \pi_i(T) = 1}} \chi_i + \sum_{\substack{\mathfrak{q} \mid \mathfrak{b} \\ \mathfrak{q} \in \widetilde{\Omega}_K(l)}} \chi_{\mathfrak{q}}.
$$
Two pairs $(T, \mathfrak{b}), (T', \mathfrak{b}') \in \mathcal{S}_l$ are said to be \emph{coprime} in case $\mathfrak{b}, \mathfrak{b'}$ are coprime ideals and there does not exist a $j \in [t]$ such that $\pi_j(T) = \pi_j(T') = 1$. Formulated differently, the two pairs are coprime exactly when 
$$
\{\sigma \in \mathcal{B}^{\vee}(K, l): \chi_{(T, \mathfrak{b})}(\sigma) \neq 0 \} \cap \{\sigma' \in \mathcal{B}^{\vee}(K, l): \chi_{(T', \mathfrak{b}')}(\sigma') \neq 0 \}=\emptyset.
$$

Let $V$ be any finite set. We denote by $\text{Prim}(\mathcal{S}_l^{\mathbb{F}_l^V - \{(0, \ldots, 0)\}})$ the subset of $\mathcal{S}_l^{\mathbb{F}_l^V - \{(0, \ldots, 0)\}}$ consisting of vectors possessing pairwise coprime coordinates. We conclude this subsection by giving a bijection 
$$
\text{Pow}_l(V): \text{Prim}(\mathcal{S}_l^{\mathbb{F}_l^V - \{(0, \ldots, 0)\}}) \to H^1(G_K, \mathbb{F}_l)^{V},
$$
which sends a vector $(v_g)_{g \in \mathbb{F}_l^V - \{(0, \dots, 0)\}}$ to
$$
\text{Pow}_l(V)((v_g)_{g \in \mathbb{F}_l^V - \{(0, \dots, 0)\}}) := \left(\sum_{g \in \mathbb{F}_l^V - \{(0, \ldots, 0)\}} \pi_j(g) \cdot \chi_{v_g} \right)_{j \in V},
$$
where $\pi_j$ is the projection map on the $j$-th coordinate. Let us prove that this map is indeed a bijection.

\begin{proposition} 
\label{it is a bijection}
Let $V$ be a finite set. Then the map $\textup{Pow}_l(V)$ is a bijection.
\end{proposition}

\begin{proof}
Assume without loss of generality that $V = [r]$. To a vector $(\chi_1, \ldots, \chi_r)$ in $H^1(G_K, \mathbb{F}_l)^r$ we attach a point
$$
\Pi_l(\chi_1, \ldots, \chi_r) := (v_g(1), v_g(2))_{g \in \mathbb{F}_l^r - \{(0, \ldots, 0)\}} 
$$
in $\text{Prim}(\mathcal{S}_l^{\mathbb{F}_l^r - \{(0, \ldots, 0)\}})$ as follows. For each $\mathfrak{q} \in \widetilde{\Omega}_K(l)$ we let $\mathfrak{q}$ divide the entry $v_g(2)$ if and only if
$$
(\chi_1(\sigma_{\mathfrak{q}}), \ldots, \chi_r(\sigma_{\mathfrak{q}})) = g.
$$
Likewise for each $j \in [t]$ we put $\pi_j(v_g(1)) = 1$ if and only if
$$
(\chi_1(\sigma_j), \ldots, \chi_r(\sigma_j)) = g.
$$
By construction $(v_g(1), v_g(2))_{g \in \mathbb{F}_l^r - \{(0, \ldots, 0)\}}$ is in $\text{Prim}(\mathcal{S}_l^{\mathbb{F}_l^r - \{(0, \ldots, 0)\}})$. Using that $\mathcal{B}^{\vee}(K, l)$ is a system of topological generators, we deduce that 
\[
\text{Pow}_l([r])((v_g(1), v_g(2))_{g \in \mathbb{F}_l^r - \{(0, \ldots, 0)\}}) = (\chi_1, \ldots, \chi_r),
\]
since the equality holds by construction when evaluated in an element of $\mathcal{B}^{\vee}(K, l)$. Conversely let $\sigma \in \mathcal{B}^{\vee}(K, l)$ and let $(v_g)_{g \in \mathbb{F}_l^r - \{(0, \ldots, 0)\}} \in \text{Prim}(\mathcal{S}_l^{\mathbb{F}_l^r - \{(0, \ldots, 0)\}})$. There exists at most one $g_0 \in \mathbb{F}_l^r - \{(0, \ldots, 0)\}$ such that $\chi_{v_{g_0}}(\sigma) \neq 0$. Suppose that such a $g_0$ exists. Then
$$
\text{Pow}_l([r])((v_g)_{g \in \mathbb{F}_l^r - \{(0, \ldots, 0)\}})(\sigma) = g_0,
$$
which implies that
\[
\Pi_l \circ \text{Pow}_l([r])((v_g)_{g \in \mathbb{F}_l^r - \{(0, \ldots, 0)\}}) = (v_g)_{g \in \mathbb{F}_l^r - \{(0, \ldots, 0)\}}.
\]
Hence $\Pi_l$ and $\text{Pow}_l([r])$ are mutual inverses, which finishes the proof of the proposition.
\end{proof}

We now have the necessary tools to generalize Proposition \ref{It is a parametrization} to general number fields and general nilpotent groups. We carry this out in the next subsection.

\subsection{The parametrization in general}
We recall the setup from Section \ref{Parametrization}. Let $r \in \Z_{\geq 1}$. A sequence of pairs
$$
\{(G_i, \theta_i)\}_{i \in [r]}
$$
is called an \emph{admissible sequence} if it satisfies the following inductive rules:
\begin{itemize}
\item $G_0$ is the trivial group by convention. Furthermore, $G_i$ is an $l$-group and $\theta_i: G_{i - 1}^2 \to \mathbb{F}_l$ is a $2$-cocycle with $\theta_i(\text{id}, \text{id}) = 0$ for each $i \in [r]$;
\item we have
\[
G_i = (\mathbb{F}_l \times G_{i - 1}, *_{\theta_i})
\]
for all $i \in [r]$;
\item $\theta_i$ is the zero map if and only if the class of $\theta$ in $H^2(G_{i - 1}, \mathbb{F}_l)$ is trivial. 
\end{itemize} 

For the remainder of this section we fix an admissible sequence $\{(G_i, \theta_i)\}_{i \in [r]}$. Set
$$
G := G_r. 
$$
The aim of this section is to construct a surjective map
$$
P_G: \text{Prim}(\mathcal{S}_l^{G - \{\text{id}\}}) \twoheadrightarrow \text{Epi}_{\text{top.gr.}}(\mathcal{G}_{K}^{\text{pro}-l}, G) \cup \{\bullet\},
$$
which restricts to a bijection between
$$
\text{Prim}(\mathcal{S}_l^{G - \{\text{id}\}})(\text{solv.}) := P_G^{-1}(\text{Epi}_{\text{top.gr.}}(\mathcal{G}_{K}^{\text{pro}-l}, G))
$$
and $\text{Epi}_{\text{top.gr.}}(\mathcal{G}_{K}^{\text{pro}-l}, G)$. Furthermore, we explain how to read the ramification data on the right hand side from the left hand side of this parametrization.

We start by defining a map
$$
\tilde{P}_G:H^1(G_K, \mathbb{F}_l)^r \to \text{Epi}_{\text{top.gr.}}(\mathcal{G}_{K}^{\text{pro}-l}, G) \cup \{\bullet\}
$$
as follows. Let $v := (\chi_1, \ldots, \chi_r)$ be an element of $H^1(G_K, \mathbb{F}_l)^{r}$. If $\chi_{1}$ is the trivial character, we declare $\tilde{P}_G(v) = \bullet$. So we assume from now on that $\chi_{1}$ is non-trivial. Equivalently,
$$
\chi_{1} \in \text{Epi}_{\text{top.gr.}}(\mathcal{G}_{K}^{\text{pro}-l}, G_1).
$$
Hence $\chi_{1}^*(\theta_2)$ is now a $2$-cocycle on $G_{K}$. If it is non-trivial, then we declare $\tilde{P}_G(v) = \bullet$. Now assume that $\chi_{1}^*(\theta_2)$ is zero in $H^2(G_{K}, \mathbb{F}_l)$; we distinguish two cases. If $\theta_2$ is already a trivial $2$-cocycle on $G_1$, we have that $\phi(G_1, \chi_{v_1}^{*}(\theta_2)) = 0$ and
$$
(\chi_2, \chi_1) \in \text{Epi}_{\text{top.gr.}}(\mathcal{G}_{K}^{\text{pro}-l}, G_2)
$$
if and only if $\chi_{1}$ and $\chi_{2}$ are linearly dependent. In case $\chi_{1}$ and $\chi_{2}$ are linearly dependent, we set $\tilde{P}_G(v) = \bullet$ and otherwise we proceed. If instead $\theta_2$ is a non-trivial $2$-cocycle on $G_1$, we always have that
$$
(\phi(G_1, \chi_{1}^{*}(\theta_2)) + \chi_{2}, \chi_{1}) \in \text{Epi}_{\text{top.gr.}}(\mathcal{G}_{K}^{\text{pro}-l}, G_2)
$$
by Proposition \ref{defining normalized phi: in general}.

Now we continue in this fashion inductively. At step $i < r$ we have either already assigned $v$ to $\bullet$, or we have obtained an epimorphism $\psi_i \in \text{Epi}_{\text{top.gr.}}(\mathcal{G}_{K}^{\text{pro}-l}, G_i)$. Then we get a $2$-cocycle $\psi_i^*(\theta_{i + 1})$, which gives a class in $H^2(G_{K}, \mathbb{F}_l)$. If this class is non-trivial in $H^2(G_{K}, \mathbb{F}_l)$, we send $v$ to $\bullet$. 

In case $\psi_i^*(\theta_{i + 1})$ is trivial in $H^2(G_K, \mathbb{F}_l)$, we distinguish two cases. If $\psi_i^*(\theta_{i + 1})$ is already trivial in $H^2(G_i, \mathbb{F}_l)$, we have that $\phi(G_i, \psi_i^*(\theta_{i + 1})) = 0$. Then
\[
(\chi_{i + 1}, \psi_i) \in \text{Epi}_{\text{top.gr.}}(\mathcal{G}_{K}^{\text{pro}-l}, G_{i + 1})
\]
if and only if $\chi_{i + 1}$ is linearly independent from the characters $\chi_{j}$ satisfying 
\[
\phi(G_{j - 1}, \psi_{j - 1}^*(\theta_j)) = 0.
\]
For a proof of this claim we refer the reader to the special case $K=\mathbb{Q}, l=2$, which we have discussed in detail in Section \ref{Parametrization}: the argument goes through without changes. 

If $\chi_{i + 1}$ is linearly dependent on these characters $\chi_j$, we send $v$ to $\bullet$. Otherwise we go to step $i + 1$.

Now suppose that $\theta_{i + 1}$ is a non-trivial class of $H^2(G_i, \mathbb{F}_l)$. Then we always obtain by means of Proposition \ref{defining normalized phi: in general} a new epimorphism
$$
(\phi(G_i, \psi_i^*(\theta_{i + 1})) + \chi_{i + 1}, \psi_i) \in \text{Epi}_{\text{top.gr.}}(\mathcal{G}_{K}^{\text{pro}-l}, G_{i + 1})
$$
and we go to step $i + 1$. Continuing in this fashion we obtain either $\bullet$ or an element of 
$$
\text{Epi}_{\text{top.gr.}}(\mathcal{G}_{K}^{\text{pro}-l}, G),
$$
which is by definition $\tilde{P}_G(v)$. We put
$$
P_G := \tilde{P}_G \circ \text{Pow}_l([r]). 
$$
We remark that $\phi(G_i, \theta)$ is only defined in case $G_i$ is a Galois group. Fortunately, this small abuse of notation does not present any issues. Indeed, the epimorphism $\psi_i$ realizes the implicit identification between $G_i$ and the corresponding Galois group. 

We additionally remark that in the construction of the map $P_G$ we have implicitly used that $G$ is set-theoretically defined to be $\mathbb{F}_l^r$ with the identity element being $(0, \ldots, 0)$, which is a consequence of our convention that $2$-cocycles vanish on $(\text{id}, \text{id})$. Hence it makes sense to invoke the map $\text{Pow}_l([r])$. 

\begin{proposition} 
\label{It is a parametrization: in general}
The map 
$$
P_G: \textup{Prim}(\mathcal{S}_l^{G - \{\textup{id}\}}) \twoheadrightarrow \textup{Epi}_{\textup{top.gr.}}(\mathcal{G}_{K}^{\textup{pro}-l}, G) \cup \{\bullet\},
$$
is a surjection, which restricts to a bijection between $\textup{Prim}(\mathcal{S}_l^{G - \{\textup{id}\}})(\textup{solv.})$ and surjective homomorphisms $\textup{Epi}_{\textup{top.gr.}}(\mathcal{G}_{K}^{\textup{pro}-l}, G)$.
\end{proposition}

\begin{proof}
This follows upon combining Proposition \ref{defining normalized phi: in general} and Proposition \ref{it is a bijection}.
\end{proof}

The parametrization $P_G$ allows us to read off very neatly the image of the elements $\{\sigma_{\mathfrak{q}} \}_{\mathfrak{q} \in \widetilde{\Omega}_K(l)}$ from the tuples of squarefree ideals. 

\begin{proposition} 
\label{Reading off inertia: in general}
Let $(v_g(1), v_g(2))_{g \in G - \{\textup{id}\}}$ be an element of $\textup{Prim}(\mathcal{S}_l^{G - \{\textup{id}\}})(\textup{solv.})$. Let $\mathfrak{q} \in \Omega_K - S_{\textup{clean}}(l)$. If $\mathfrak{q} \mid v_{g_0}(2)$ for a (necessarily) unique $g_0 \in G - \{\textup{id}\}$ then
$$
P_G((v_g(1), v_g(2))_{g \in G - \{\textup{id}\}})(\sigma_{\mathfrak{q}}) = g_0.
$$
If $\mathfrak{q}$ does not divide any of the elements of the vector $(v_g(2))_{g \in G - \{\textup{id}\}}$, then 
$$
P_G((v_g(1), v_g(2))_{g \in G - \{\textup{id}\}})(\sigma_{\mathfrak{q}}) = \textup{id}.
$$
\end{proposition}

\begin{proof}
For the case $\mathfrak{q} \in \widetilde{\Omega}_K(l)$ the conclusion follows immediately from the fact that the map $\Pi_l$ in the proof of Proposition \ref{it is a bijection} is inverse to the map $\text{Pow}_l([r])$. Otherwise, we have that $\mathfrak{q} \in \Omega_K - S_{\text{clean}}(l) - \widetilde{\Omega}_K(l)$ so that $\sigma_{\mathfrak{q}} = \text{id}$ by definition. By construction of $\mathcal{S}_l$ it follows that such $\mathfrak{q}$ do not divide any $v_g(2)$. Hence we are always in the second case of the proposition. Therefore the statement also holds for such $\mathfrak{q}$.
\end{proof}

We next read off the value of the discriminant under the bijection $P_G$. For any continuous homomorphism $\psi$ of $G_{K}$ with values in some finite group, we denote by $\text{Disc}(\psi)$ the relative discriminant (which is an ideal of $\mathcal{O}_K$) of the corresponding extension. For a non-zero integral ideal $\mathfrak{b}$ in $\mathcal{O}_K$ we write $\text{free}_S(\mathfrak{b})$ for the largest ideal dividing $\mathfrak{b}$ and entirely supported outside of $S$.

\begin{proposition} 
\label{reading off discriminants: in general}
Let $(v_g(1), v_g(2))_{g \in G - \{\textup{id}\}}$ be an element of $\textup{Prim}(\mathcal{S}_l^{G - \{\textup{id}\}})(\textup{solv.})$. Then
$$
\textup{free}_{S_{\textup{clean}}(l)}(\textup{Disc}(P_G((v_g)_{g \in G - \{\textup{id}\}}))) = \prod_{g \in G - \{\textup{id}\}} v_g(2)^{\#G \cdot (1-\frac{1}{\#\langle g \rangle})}.
$$
\end{proposition}

\begin{proof}
We show that the $\mathfrak{q}$-adic valuation matches for any prime $\mathfrak{q}$ of $\mathcal{O}_K$. This is certainly true for the places $\mathfrak{q}$ in $S_{\text{clean}(l)}$, but also for the places $\mathfrak{q}$ outside $\widetilde{\Omega}_K(l)$ by Proposition \ref{only with local roots of unity}. Now take a place $\mathfrak{q}$ in $\widetilde{\Omega}_K(l)$. Since $\mathfrak{q}$ is coprime to $l$ we know that  
$$
v_{\mathfrak{q}}(\textup{Disc}(P_G((v_g)_{g \in G - \{\textup{id}\}}))) = \frac{\#G}{\# \langle P_G((v_g)_{g \in G - \{\textup{id}\}})(\sigma_{\mathfrak{q}}) \rangle} \cdot \left(\# \langle P_G((v_g)_{g \in G - \{\textup{id}\}})(\sigma_{\mathfrak{q}}) \rangle - 1\right).
$$
Thanks to Proposition \ref{Reading off inertia: in general} we deduce that $P_G((v_g)_{g \in G - \{\textup{id}\}})(\sigma_{\mathfrak{q}})$ is trivial in case $\mathfrak{q}$ does not divide any $v_g$ and equals $g_0$ in case $\mathfrak{q}$ divides $v_{g_0}$. This is precisely the desired conclusion. 
\end{proof}

Our final goal for this subsection is to generalize Propositions \ref{It is a parametrization: in general}, \ref{Reading off inertia: in general} and \ref{reading off discriminants: in general} to arbitrary finite, nilpotent groups. Recall that a finite group $G$ is nilpotent if and only if it decomposes as a direct product of its Sylow subgroups. Let $c$ be a positive integer. Let $l_1, \ldots, l_c$ be distinct prime numbers. For each $j \in [c]$ fix an admissible sequence
$$
\{(G_i(l_j), \theta_i(l_j))\}_{i \in [r_j]}
$$
and write $G(l_j) := G_{r_j}(l_j)$ for the resulting $l_j$-group. We put 
$$
G := \prod_{j = 1}^c G(l_j).
$$
To parametrize $G$-extensions, we reduce to $l$-groups by means of the following proposition.

\begin{proposition} 
\label{epi iff cordinate are epi}
We have an identification  
$$
\textup{Epi}_{\textup{top.gr.}}(G_K, G) = \prod_{j \in [c]} \textup{Epi}_{\textup{top.gr.}}(\mathcal{G}_K^{\textup{pro}-l_j}, G(l_j))
$$
through the natural map.
\end{proposition} 

\begin{proof}
Let $H$ be any group and $\psi = (\psi_j)_{j \in [c]}:H \to G$ be any group homomorphism. We have to show that if $\psi_j$ is surjective for each $j \in [c]$, then $\psi$ is surjective. Observe that if each $\psi_j$ is surjective, then $\#\text{Im}(\psi)$ is divisible by $\#G(l_j)$ for each $j \in [c]$. Since these values are coprime, we find out that $\#\text{Im}(\psi)$ is divisible by $\#G$, which means precisely that $\psi$ is surjective.
\end{proof}

Thanks to Proposition \ref{epi iff cordinate are epi} we can now bundle together the various maps $P_{G(l_j)}$ into one map
$$
P_G: \prod_{j \in [c]} \text{Prim}(\mathcal{S}_{l_j}^{G(l_j) - \{\text{id}\}}) \twoheadrightarrow \text{Epi}_{\text{top.gr.}}(G_K, G) \cup \{\bullet\}
$$
by simply taking the product map. Put
\[
S := \bigcup_{j \in [c]} S_{\textup{clean}}(l_j).
\]
For every $j \in [c]$, let $\{\chi_{j, i}\}_{i \in [t_j]}$ be a basis for the space of characters $G_K \rightarrow \FF_{l_j}$ only ramified at $S$. Define $\mathcal{S}$ to be $\{0, 1\}^{[t_1 + \dots + t_c]} \times \mathcal{S}'$, where $\mathcal{S}'$ is the set of squarefree ideals supported outside $S$. Let
\[
\text{Prim}(\mathcal{S}^{G - \{\text{id}\}})
\]
be the set of tuples $(v_g(1), v_g(2))_{g \in G - \{\text{id}\}}$ satisfying the following properties
\begin{itemize}
\item writing $\pi_i$ for the natural projection map $[t_1 + \dots + t_c] \rightarrow [t_i]$, we have that the $\pi_i(v_g(1))$ are pairwise coprime;
\item the $v_g(2)$ are pairwise coprime;
\item if $\mathfrak{p}$ divides $v_g(2)$ and $l$ is a prime dividing the order of $g$, then
\[
\# (\mathcal{O}_K/\mathfrak{p}) \equiv 1 \bmod l.
\]
\end{itemize}
For an element 
\[
(v_{g, j}(1), v_{g, j}(2))_{j \in [c], g \in G(l_j) - \{\textup{id}\}} \in \prod_{j \in [c]} \text{Prim}(\mathcal{S}_{l_j}^{G(l_j) - \{\text{id}\}})
\]
and for $g := (g_1, \ldots, g_c) \in G - \{\text{id}\}$ we define
$$
v_g(1) := (v_{g_1, 1}(1), \dots, v_{g_c, c}(1)), \quad v_g(2) := \prod_{\substack{\mathfrak{p} \\ \forall j \in [c] \forall h \in G(l_j) - \{\text{id}\} : \mathfrak{p} \mid v_{h, j}(2) \Leftrightarrow g_j = h}} \mathfrak{p}.
$$
In this way we have created a very convenient bijection between
\[
\text{Prim}(\mathcal{S}^{G - \{\text{id}\}}) \cong \prod_{j \in [c]} \text{Prim}(\mathcal{S}_{l_j}^{G(l_j) - \{\text{id}\}}).
\]
We need one additional piece of notation, namely we define
$$
\textup{Prim}(\mathcal{S}^{G - \{\textup{id}\}})(\textup{solv.}) := \prod_{j \in [c]} \text{Prim}(\mathcal{S}_{l_j}^{G(l_j) - \{\text{id}\}})(\text{solv.}),
$$
which we shall often implicitly view as a subset of $\text{Prim}(\mathcal{S}^{G - \{\text{id}\}})$. The next proposition generalizes Proposition \ref{Reading off inertia: in general}. 

\begin{proposition} 
\label{Reading off inertia: in general plus}
Let $v := (v_{g, j}(1), v_{g, j}(2))_{j \in [c], g \in G(l_j) - \{\textup{id}\}}$ be an element of 
\[
\textup{Prim}(\mathcal{S}^{G - \{\textup{id}\}})(\textup{solv.}).
\]
Take some $\mathfrak{q} \in \Omega_K - S$. Let $T$ be the subset of $[c]$ such that $j \in T$ if and only if there exists a (necessarily) unique $g_0^j \in G(l_j) - \{\textup{id}\}$ with $\mathfrak{q} \mid v_{g_0^j, j}(2)$. Then we have
$$
P_G(v)(\sigma_{\mathfrak{q}}) = (g_0^j)_{j \in T} \times (\textup{id})_{k \in [c] - T}.
$$
In particular if $\mathfrak{q}$ does not divide any of the elements $v_{g, j}(2)$, i.e. $T=\emptyset$, then 
$$
P_G(v)(\sigma_{\mathfrak{q}}) = \textup{id}.
$$
\end{proposition}

\begin{proof}
This follows at once from Proposition \ref{It is a parametrization: in general}, applied to each $G(l_j)$-factor. 
\end{proof}

The next proposition generalizes Proposition \ref{reading off discriminants: in general}. 
In the new coordinates we have a rather simple formula for the discriminant.

\begin{proposition} 
\label{reading off discriminants: in general plus}
Notations as above. Let $(v_{g, j}(1), v_{g, j}(2))_{j \in [c], g \in G(l_j) - \{\textup{id}\}}$ be an element of $\textup{Prim}(\mathcal{S}^{G - \{\textup{id}\}})(\textup{solv.})$. Then
$$
\textup{free}_S(\textup{Disc}(P_G((v_{g, j}(1), v_{g, j}(2))_{j \in [c], g \in G(l_j) - \{\textup{id}\}}))) = \textup{free} _S \left( \prod_{g \in G - \{\textup{id}\}} v_g(2)^{\#G \cdot (1-\frac{1}{\#\langle g \rangle})} \right).
$$
\end{proposition}

\begin{proof}
This follows from Proposition \ref{Reading off inertia: in general plus} with exactly the same argument as used to establish Proposition \ref{reading off discriminants: in general} as a consequence of Proposition \ref{Reading off inertia: in general}. 
\end{proof}

\section{Local conditions and conjugacy classes}
\label{sLocal conditions}

\subsection{Some group theory}
\label{ssGroup}
Let $l$ be a prime number and let $G$ be a finite $l$-group given by an admissible sequence $\{(G_i, \theta_i)\}_{i \in [r]}$ with $G := G_r$. Our first goal is to study the formation of a conjugacy class in $G$, through the various groups $G_i$, with $i \in [r]$. For $g \in G$ we denote by $\text{Conj}_G(g)$ its conjugacy class. For each $0 \leq i \leq r$ we write
$$
\pi_i: G \to G_i
$$
for the natural projection map. 

For now we take any finite $l$-group $H$, a $2$-cocycle $\theta$ representing a class in $H^2(H, \mathbb{F}_l)$, with $\theta(\text{id}, \text{id}) = 0$, and an element $h \in H$. Denote the centralizer of $h$ by $\text{Cent}_H(h)$. Then lifting elements of $\text{Cent}_H(h)$ to $(\mathbb{F}_l \times H, *_{\theta})$ and taking the commutator with any lift $h_0$ of $h$ induces a homomorphism
$$
[-, h_0]_\theta: \frac{\text{Cent}_H(h)}{\langle h \rangle} \to \mathbb{F}_l,
$$
which does not depend on the choice of lifts. Put
$$
\widetilde{\text{Cent}}_H(h, \theta) := \text{ker}([-, h_0]_{\theta}),
$$
which is by definition a subgroup of $\frac{\text{Cent}_H(h)}{\langle h \rangle}$. Let
$$
\pi_{\theta}: (\mathbb{F}_l \times H, *_{\theta}) \to H
$$ 
be the natural projection map. Our next proposition describes the relationship between $\text{Cent}_H(h)$ and $\widetilde{\text{Cent}}_H(h, \theta)$.

\begin{proposition} 
\label{formation of conjugacy classes}
Let $H, h, \theta$ be as above this proposition. Then
$$
\left[\frac{\textup{Cent}_H(h)}{\langle h \rangle}:\widetilde{\textup{Cent}}_H(h, \theta)\right] \in \{1, l\},
$$
The index equals $1$ if and only if the elements in $\pi_{\theta}^{-1}(h)$ are pairwise non-conjugate in $(\mathbb{F}_l \times H, *_{\theta})$. The index equals $l$ if and only if the elements of $\pi_{\theta}^{-1}(h)$ sit inside a unique conjugacy class in $(\mathbb{F}_l \times H, *_{\theta})$.
\end{proposition}

\begin{proof}
Indeed, take a lift $h_0 := (a, h)$ in $\pi_{\theta}^{-1}(h)$. Observe that if we have
$$
[(b, h'), h_0]*_\theta h_0 = (b, h')*_\theta h_0*_\theta (b, h')^{-1} = (a', h)
$$
for some $a' \in \mathbb{F}_l$, then it follows that $h' \in \text{Cent}_H(h)$. From the left hand side we see that if the index is $l$, then $a'$ can take any possible value. If the index is instead equal to $1$, then $a'$ must be equal to $a$. 
\end{proof}

We also have a similar proposition for the exponent of an element. 

\begin{proposition} 
\label{formation of exponents}
Let $H, h, \theta$ be as above. Then either $\pi_{\theta}^{-1}(h)$ consists entirely of elements with order equal to $l \cdot \# \langle h \rangle$ or it consists entirely of elements with order equal to $\#\langle h \rangle$. 
\end{proposition}

\begin{proof}
The class $\theta$ restricted to $\langle h \rangle$ gives an element of $H^2(\langle h \rangle, \mathbb{F}_l) = \text{Ext}(\langle h \rangle, \mathbb{F}_l)$. If the class is $0$, then the sequence is split and we have that all the elements of $\pi_{\theta}^{-1}(h)$ have the same order as $h$. If the class is non-zero, then the sequence has the shape
$$
0 \to \mathbb{F}_l \to \mathbb{Z}/l \cdot \#\langle h \rangle \mathbb{Z} \to \langle h \rangle \to 0,
$$
and hence all elements of $\pi_{\theta}^{-1}(h)$ have order $l$ times bigger than that of $h$. 
\end{proof}

In case an $h$ as in Proposition \ref{formation of exponents} satisfies the second conclusion we say that $h$ is $\theta$-stable. We now return to our previous setup. To $g \in G$ we attach the following quantity
$$
j_G(g, \{(G_i, \theta_i)\}_{i \in [r]}) := \#\left\{i \in [r] : \frac{\textup{Cent}_{G_{i - 1}}(\pi_{i - 1}(g))}{\langle \pi_{i - 1}(g) \rangle} \neq \widetilde{\textup{Cent}}_{G_{i - 1}}(\pi_{i - 1}(g), \theta_i)\right\}.
$$
This quantity turns out to be the exponent of $l$ in the size of the conjugacy class $\text{Conj}_G(g)$. We call the $i$'s counted by $j_G(g, \{(G_i, \theta_i)\}_{i \in [r]})$ the \emph{breaks} for $g$ with respect to $\{(G_i, \theta_i)\}_{i \in [r]}$.

\begin{proposition} 
\label{number of breaks=conjugacy class}
We have
$$
\#\textup{Conj}_G(g) = l^{j_G(g, \{(G_i, \theta_i)\}_{i \in [r]})}.
$$
\end{proposition}

\begin{proof}
We have a filtration of subgroups
$$
G = \text{Cent}_G^0(g) \supseteq \text{Cent}_G^1(g) \supseteq \dots \supseteq \text{Cent}_G^r(g) = \text{Cent}_G(g),
$$
where we define
$$
\text{Cent}_G^i(g) := \pi_i^{-1}(\text{Cent}_{G_i}(\pi_i(g)))
$$
for every integer $0 \leq i \leq r$. We have that
$$
\#\text{Conj}_G(g) = \frac{\#G}{\#\text{Cent}_G(g)} = \prod_{0 \leq i \leq r - 1} \frac{\#\text{Cent}_G^i(g)}{\#\text{Cent}_G^{i + 1}(g)}.
$$
Observe that 
$$
[\text{Cent}_G^{i}(g):\text{Cent}_G^{i + 1}(g)] = \left[\frac{\textup{Cent}_{G_i}(\pi_i(g))}{\langle \pi_i(g) \rangle}:\widetilde{\textup{Cent}}_{G_i}(\pi_i(g), \theta_{i + 1})\right].
$$
Therefore the desired conclusion follows at once from Proposition \ref{formation of conjugacy classes}.
\end{proof}

The next proposition provides the crucial link between group theoretic data and the local conditions imposed, through Proposition \ref{defining normalized phi: in general}, on tuples in $\text{Prim}(\mathcal{S}_l^{G - \{\text{id}\}})$. We invoke the notation of Section \ref{sMain}. Let $L/K$ be a finite Galois extension inside $K^{\textup{pro}-l}/K$. Let now $G:=\textup{Gal}(L/K)$ and let $\theta$ be a $2$-cocycle representing a class in $H^2(\textup{Gal}(L/K), \mathbb{F}_l)$, with $\theta(\textup{id}, \textup{id}) = 0$. Let $\mathfrak{q} \in \Omega_K$ be a finite place coprime to $l$. Recall that 
$$
\frac{G_{K_{\mathfrak{q}}}}{I_{\mathfrak{q}}}
$$
is a pro-cyclic group, equipped with a canonical generator $\text{Frob}_{\mathfrak{q}}$. We will fix once and for all the $\textup{proj}(G_K \to \mathcal{G}_K^{\textup{pro}-l}) \circ i_{\mathfrak{q}}^{*}$-image of a lift to $G_{K_{\mathfrak{q}}}$ of such an element. In this way we obtain an element in $\mathcal{G}_K^{\text{pro}-l}$ that we will denote also by $\text{Frob}_{\mathfrak{q}}$: this slight abuse of notation will cause no confusion. As such we have naturally an element
$$
\textup{proj}(\mathcal{G}_K^{\textup{pro}-l} \to G)(\text{Frob}_{\mathfrak{q}}) \in \frac{N_G(\langle \textup{proj}(\mathcal{G}_K^{\textup{pro}-l} \to G)(\sigma_{\mathfrak{q}}) \rangle)}{ \langle \textup{proj}(\mathcal{G}_K^{\textup{pro}-l} \to G)(\sigma_{\mathfrak{q}})\rangle}.
$$
Here $N_G(-)$ denotes the normalizer of a subgroup in $G$. For any non-trivial finite group $G$, we denote by $l_G$ the smallest prime divisor of $\#G$ and by $I(G)$ the subset of $g \in G - \{\text{id}\}$ such that $g^{l_G} = \text{id}$.

\begin{proposition} 
\label{solvable iff commutative}
Let $G = \Gal(L/K)$ be a finite $l$-group and let $\mathfrak{q} \in \Omega_K$ be a finite place coprime to $l$. Assume that $\textup{proj}(\mathcal{G}_K^{\textup{pro}-l} \to G)(\sigma_{\mathfrak{q}})$ is an element of $I(G)$, which we shall also call $\sigma_{\mathfrak{q}}$. Then
$$
\textup{proj}(\mathcal{G}_K^{\textup{pro}-l} \to G)(\textup{Frob}_{\mathfrak{q}}) \in \frac{\textup{Cent}_G(\sigma_{\mathfrak{q}})}{\langle \sigma_{\mathfrak{q}} \rangle}.
$$
Moreover, if $\textup{proj}(\mathcal{G}_K^{\textup{pro}-l} \to G)(\sigma_{\mathfrak{q}})$ is $\theta$-stable, then 
$$
\textup{proj}(\mathcal{G}_K^{\textup{pro}-l} \to G)(\textup{Frob}_{\mathfrak{q}}) \in \frac{\widetilde{\textup{Cent}}_G(\sigma_{\mathfrak{q}}, \theta)}{\langle \sigma_{\mathfrak{q}} \rangle}
$$
if and only if
$$
\theta \textup{ is trivial in } H^2(G_{K_{\mathfrak{q}}}, \mathbb{F}_l). 
$$
\end{proposition}

\begin{proof}
Let $G$ be any finite non-trivial group. Then we claim that $g \in I(G)$ implies $N_G(\langle g \rangle) = \text{Cent}_G(g)$. Indeed, conjugation induces a homomorphism
$$
N_G(\langle g \rangle) \to \text{Aut}_{\text{gr.}}(\langle g \rangle) \simeq_{\textup{gr.}} \mathbb{F}_{l_G}^{*}.
$$
Since the latter group has size $l_G - 1$, and $l_G$ is the smallest prime divisor of $\#G$, we have that $\#G$ is coprime to $l_G-1$. Therefore the above homomorphism is actually trivial. This means exactly that $N_G(\langle g \rangle) = \text{Cent}_G(g)$ as claimed. Hence we have already shown the first part of this proposition, namely that  
$$
\textup{proj}(\mathcal{G}_K^{\textup{pro}-l} \to G)(\textup{Frob}_{\mathfrak{q}}) \in \frac{\textup{Cent}_G(\sigma_{\mathfrak{q}})}{\langle \sigma_{\mathfrak{q}} \rangle}.
$$
Assume now that $\textup{proj}(\mathcal{G}_K^{\textup{pro}-l} \to G)(\sigma_{\mathfrak{q}})$ is also $\theta$-stable. Observe that since $\sigma_{\mathfrak{q}}$ lands in $I(G)$, we in particular conclude that $\mathfrak{q}$ ramifies in $L/K$. It follows from Proposition \ref{only with local roots of unity} that
$$
q := \#\left(\mathcal{O}_K/\mathfrak{q} \right) \ \equiv 1 \ \text{mod} \ l.
$$
Also observe that the maximal pro-$l$ quotient of $G_{K_{\mathfrak{q}}}$ is isomorphic to
$$
\mathbb{Z}_l \rtimes q^{\mathbb{Z}_l},
$$
where $q$ acts by multiplication by $q$ on $\mathbb{Z}_l$. Here $I_{K_{\mathfrak{q}}}$ is sent to $\mathbb{Z}_l \rtimes \{1\}$, while a lift of $\text{Frob}_{\mathfrak{q}}$ is sent to $\{0\} \rtimes \{q\}$. Since $\textup{proj}(\mathcal{G}_K^{\textup{pro}-l} \to G)(\sigma_{\mathfrak{q}})$ is $\theta$-stable and lands in $I(G)$, we have that the lifting problem imposed by $\theta$ factors through
$$
l \cdot \mathbb{Z}_l \rtimes \{1\}.
$$
Since $q$ is $1$ modulo $l$, the resulting quotient is simply
$$
\mathbb{F}_l \times q^{\mathbb{Z}_l}.
$$
Recalling once more that $\textup{proj}(\mathcal{G}_K^{\textup{pro}-l} \to G)(\sigma_{\mathfrak{q}})$ is $\theta$-stable, we see that the lifting problem is solvable if and only if the restriction of $\theta$ to 
$$
\textup{proj}(G_K \to G) \circ i_{\mathfrak{q}^*}(G_{K_{\mathfrak{q}}})
$$
is in 
$$
\text{Ext}(\textup{proj}(G_K \to G) \circ i_{\mathfrak{q}^*}(G_{K_{\mathfrak{q}}}), \mathbb{F}_l).
$$
This is precisely equivalent to asking
$$
\textup{Frob}_{\mathfrak{q}} \in \frac{\widetilde{\textup{Cent}}_G(\sigma_{\mathfrak{q}}, \theta)}{\langle \sigma_{\mathfrak{q}} \rangle}
$$
as was to be shown.
\end{proof}

The final lemma of this subsection provides a simple way to compute the constant $b(G, K)$ for nilpotent $G$.

\begin{lemma}
\label{lbGKNil}
Let $G$ be a non-trivial, finite group and let $K$ be a number field. Then we have
\[
b(G, K) = \frac{\#\{C \in \textup{Conj}(G) : C \subseteq I(G)\}}{[K(\zeta_{l_G}) : K]}.
\]
\end{lemma}

\begin{proof}
Recall that there is a natural action of $\Gal(\overline{K}/K)$ on
\[
X := \{C \in \textup{Conj}(G) : C \subseteq I(G)\},
\]
which sends a conjugacy class $C$ to $C^{\chi(\sigma)}$ with $\chi: \Gal(\overline{K}/K) \rightarrow \hat{\Z}^\ast$ the cyclotomic character. By definition $b(G, K)$ equals the number of orbits of this group action. But our action clearly factors through $\Gal(K(\zeta_{l_G})/K)$. We claim that the induced action of $\Gal(K(\zeta_{l_G})/K)$ on $X$ is free, which implies the lemma.

So suppose that there exists $\sigma \in \Gal(K(\zeta_{l_G})/K)$ such that
\[
C^{\chi(\sigma)} = C.
\]
This implies that there exists a non-trivial $g \in C$ such that $g$ and $g^{\chi(\sigma)}$ are conjugate, say
\[
g = h^{-1} g^{\chi(\sigma)} h,
\]
and hence $h \in N_G(\langle g \rangle) = \text{Cent}_G(g)$ by the argument given at the start of Proposition \ref{solvable iff commutative}. We conclude that $g = g^{\chi(\sigma)}$, which forces $\sigma$ to be the identity as desired.
\end{proof}

\subsection{Interpretation of Malle's constant}
\label{sHeuristic}
The goal of this subsection is to give a heuristic supporting Malle's conjecture in the nilpotent case. Our heuristic is based on a combination of the parametrization given in Proposition \ref{It is a parametrization: in general} and the examination of the local conditions carried out in Subsection \ref{ssGroup}. To simplify the notation, we shall limit ourselves to the case where $G$ is an $l$-group. We leave it to the reader to generalize the material below to arbitrary nilpotent groups $G$. 

Ignoring the finitely many bad places in $S_{\text{clean}}(l)$, it follows from Proposition \ref{It is a parametrization: in general} and Proposition \ref{reading off discriminants: in general} that
\[
\#\{\psi \in \text{Epi}_{\text{top.gr.}}(\mathcal{G}_K^{\text{pro}-l}, G): |N_{K/\mathbb{Q}} \text{Disc}(\psi)| \leq X\}
\]
should have order of magnitude
\[
\#\{(v_g(1), v_g(2))_{g \in G - \{\text{id}\}} \in \text{Prim}(\mathcal{S}_l^{G - \{\text{id}\}})(\text{solv.}): \prod_{g \in G - \{\text{id}\}}(|N_{K/\mathbb{Q}}v_g(2)|)^{\#G \cdot (1-\frac{1}{\#\langle g \rangle})} \leq X\}.
\]
We now focus on the variables $(v_g(1), v_g(2))$ with $g \in I(G)$. Upon combining Proposition \ref{solvable iff commutative} and Proposition \ref{Reading off inertia: in general}, we see that the primes $\mathfrak{q}$ dividing $v_g(2)$ impose a local condition only at the breaks for $g$ in the admissible sequence $\{(G_i, \theta_i)\}_{i \in [r]}$. The local conditions at the points that are not breaks, are automatically satisfied in virtue of Proposition \ref{solvable iff commutative}. Now pretend that the values
$$
\text{Frob}_{\mathfrak{q}} \in \frac{\text{Cent}_{G_{i - 1}}(\sigma_{\mathfrak{q}})}{\langle \sigma_{\mathfrak{q}} \rangle }
$$
are jointly equidistributed at every break point $i$. Then, in virtue of Proposition \ref{solvable iff commutative}, we get the following sum
$$
\approx \sum_{\prod_{g \in G - \{\text{id}\}}(|N_{K/\mathbb{Q}}v_g(2)|)^{\#G \cdot (1-\frac{1}{\#\langle g \rangle})} \leq X} \left(\prod_{g \in I(G)} \frac{1}{l^{\omega(v_g(2)) j_G(g, \{(G_i, \theta_i)\}_{i \in [r]})}} \right),
$$
where the sum runs over all points $(v_g(1), v_g(2))$ in $\text{Prim}(\mathcal{S}_l^{G - \{\text{id}\}})$. Thanks to Proposition \ref{number of breaks=conjugacy class} the latter expression equals
\begin{align}
\label{eSimpleMalle}
\sum_{\prod_{g \in G - \{\text{id}\}}(|N_{K/\mathbb{Q}}v_g(2)|)^{\#G \cdot (1-\frac{1}{\#\langle g \rangle})} \leq X} \left(\prod_{g \in I(G)} \frac{1}{\#\text{Conj}_G(g)^{\omega(v_g(2))}} \right),
\end{align}
where the sum still ranges over all points $(v_g(1), v_g(2))$ in $\text{Prim}(\mathcal{S}_l^{G - \{\text{id}\}})$. Standard analytic techniques, see Theorem \ref{tSelbergDelange}, show that the sum in equation (\ref{eSimpleMalle}) is asymptotic to
$$
c(G, K) \cdot X^{a(G)} \cdot \log(X)^{\beta(G, K) - 1}, \quad \beta(G, K) := \sum_{g \in I(G)} \frac{1}{\# \text{Conj}_G(g) \cdot [K(\zeta_l) : K]},
$$
where $\beta(G, K)$ is the Malle constant by Lemma \ref{lbGKNil}. We remark that to turn this simple heuristic into an argument one also has to pay careful attention to the local conditions at the primes dividing variables outside of $I(G)$ and to the local conditions at the primes in $S_{\text{clean}}(l)$. These will affect the constant $c(G, K)$ in the asymptotic. We finish this section by explaining the interplay between this heuristic and the proofs of our main theorems.

During the proof of Theorem \ref{mt3: Upper bound general} we simply ignore the local conditions, at the cost of losing track of the conjugation in $G$: for this reason we get $i(G, K)-1$ instead of $b(G, K) - 1$. 

Correspondingly for those $G$ for which $i(G, K) - 1$ and $b(G, K) - 1$ coincide we have that all the elements of $I(G)$ are central, and consistently with Proposition \ref{solvable iff commutative} we have no local conditions coming from the variables in $I(G)$. In this case we are able to prove an asymptotic in Theorem \ref{mt4: Asymptotic}. 

Finally in the proof of Theorem \ref{mt5: right upper bound}, thanks to the fact that the elements of $I(G)$ are pairwise commuting, we have to control the behavior of $\text{Frob}_{\mathfrak{q}}$ in the quotient $\frac{G}{I(G) \cup \{\text{id}\}}$ for each $\mathfrak{q}$ dividing a variable in $I(G)$. This is very convenient, since the corresponding field is constructed out of the variables outside of $I(G)$, and those have a very large weight in the formula for the discriminant given in Proposition \ref{reading off discriminants: in general}. As such, they can almost be treated as fixed, and the required joint equidistribution of Frobenius elements is provable by appealing to the Chebotarev density theorem. Hence in this case we can partially turn the above heuristic into a rigorous argument: since we control only the local conditions at the places dividing variables in $I(G)$, we naturally end up with an upper bound of the correct order of magnitude.

\section{Analytic considerations}
\label{sAna}
In this section we provide the analytic tools used to prove our main theorems. The material in this section is a generalization of the material in Montgomery--Vaughan \cite[Section 7.4]{MoVa} and is an application of the Selberg--Delange method. Let $K$ be a number field, let $L$ be an abelian extension of $K$ and let $S \subseteq \Gal(L/K)$. Write $\mathcal{I}_K$ for the group of non-zero fractional ideals of $K$. For a squarefree ideal $I$ of $\mathcal{O}_K$ we define $\omega(I)$ for the number of prime divisors $\mathfrak{p}$ of $I$ and $\omega_S(I)$ for the number of prime divisors $\mathfrak{p}$ of $I$ that are unramified in $L$ and satisfy $\text{Frob}_{\mathfrak{p}} \in S$. Given a complex number $z$ and a collection of prime ideals $\mathcal{P}$, we are interested in the sum
\[
A_z(x) := \sum_{\substack{N_{K/\Q}(I) \leq x \\ \mathfrak{p} \mid I \Rightarrow \text{Frob}_\mathfrak{p} \in S \text{ and } \mathfrak{p} \not \in \mathcal{P}}} \mu^2(I) z^{\omega(I)} = \sum_{\substack{N_{K/\Q}(I) \leq x \\ \mathfrak{p} \mid I \Rightarrow \text{Frob}_\mathfrak{p} \in S \text{ and } \mathfrak{p} \not \in \mathcal{P}}} \mu^2(I) z^{\omega_S(I)},
\]
where $\mu$ is the M\"obius function of $\mathcal{O}_K$. Write
\begin{align}
\label{eCoeff}
F(s, z) = \sum_{\substack{I \in \mathcal{I}_K \\ \mathfrak{p} \mid I \Rightarrow \text{Frob}_\mathfrak{p} \in S \text{ and } \mathfrak{p} \not \in \mathcal{P}}} \frac{\mu^2(I) z^{\omega(I)}}{N_{K/\Q}(I)^s} = \sum_{n = 1}^\infty \frac{a_z(n)}{n^s}
\end{align}
for its Dirichlet series with coefficients $a_z(n)$. Then we have for $s = \sigma + it$
\begin{align*}
F(s, z) &= \prod_{\substack{\mathfrak{p} \not \in \mathcal{P} \\ \text{Frob}_{\mathfrak{p}} \in S}} \left(1 + \frac{z}{N_{K/\Q}(\mathfrak{p})^{s}}\right) \\
&= \prod_{\mathfrak{p} \not \in \mathcal{P}} \left(1 + \frac{z \sum_{\sigma \in S} \frac{1}{\# \Gal(L/K)}\sum_{\chi \in \Gal(L/K)^\vee} \chi(\text{Frob}_\mathfrak{p}) \overline{\chi(\sigma)}}{N_{K/\Q}(\mathfrak{p})^{s}}\right) \text{ for } \sigma > 1,
\end{align*}
where $\Gal(L/K)^\vee$ is by definition $\text{Hom}(\Gal(L/K), \mathbb{C}^\ast)$. We assume that the Euler product
\begin{align}
\label{ePsmall}
\prod_{\mathfrak{p} \in \mathcal{P}} \left(1 + \frac{1}{N_{K/\Q}(\mathfrak{p})^s}\right)
\end{align}
converges absolutely in the region $\sigma > 1 - \delta$ for some constant $\delta > 0$. Then we approximate the Dirichlet series $F(s, z)$ with
\[
G(s, z) := \prod_{\sigma \in S} \prod_{\chi \in \Gal(L/K)^\vee} L(s, \chi)^{\frac{z \overline{\chi(\sigma)}}{\# \Gal(L/K)}},
\]
where
\[
L(s, \chi) = \prod_{\mathfrak{p}} \left(1 - \frac{\chi(\text{Frob}_{\mathfrak{p}})}{N_{K/\Q}(\mathfrak{p})^s}\right)^{-1} \text{ for } \sigma > 1.
\]
We recall that $L(s, \chi)^z$ is by definition $e^{z \log L(s, \chi)}$. Note that $\log L(s, \chi)$ exists since the region $\sigma > 1$ is simply connected and $L(s, \chi)$ does not vanish in this region. We choose our determination of the logarithm in such a way that it agrees with the real logarithm for real $s$. It follows from equation (\ref{ePsmall}) that we have the fundamental relation
\[
F(s, z) = G(s, z) H(s, z),
\]
where $H(s, z)$ is defined by an absolutely convergent Euler product in the region $\sigma > 1 - \delta$ for some $\delta > 0$. In particular, if $|z| \leq R$, then there exists some constant $\delta(R) > 0$ such that $H(s, z)$ is a bounded non-zero holomorphic function on $\sigma > 1 - \delta(R)$.

\begin{theorem}
\label{tSelbergDelange}
Let $K$, $L$, $S$ and $\mathcal{P}$ as above. Then we have for all positive real numbers $R$ and all $|z| \leq R$
\[
A_z(x) = C x(\log x)^{\frac{z \# S}{\# \Gal(L/K)} - 1} + O_{R, K, L, S, \mathcal{P}}\left(x (\log x)^{\frac{\textup{Re}(z) \# S}{\# \Gal(L/K)} - 2}\right),
\]
where $C > 0$ is a real constant depending only on $z$, $K$, $L$, $S$ and $\mathcal{P}$.
\end{theorem}

\begin{proof}
Since the proof is similar to Montgomery--Vaughan \cite[Theorem 7.17]{MoVa}, we shall only sketch the necessary modifications. Set $a = 1 + 1/\log x$. An effective version of Perron's formula, see \cite[Corollary 5.3]{MoVa}, shows that
\[
A_z(x) - \frac{1}{2\pi i} \int_{a - iT}^{a + iT} F(s, z) \frac{x^s}{s} ds \ll \sum_{\frac{1}{2}x < n < 2x} |a_z(n)| \min\left(1, \frac{x}{T|x - n|}\right) + \frac{x^a}{T} \sum_{n = 1}^\infty |a_z(n)| n^{-a},
\]
where we recall that $a_z(n)$ is defined by equation (\ref{eCoeff}) and $T$ is a parameter at our disposal. We choose $T = \exp(\sqrt{\log x})$ and estimate the error terms as in \cite{MoVa}. To do so, we need to have a good estimate for the sum
\[
\sum_{|n - x| \leq \frac{x}{(\log x)^{(2R)^{[K : \Q]} + R + 1}}} |a_z(n)| \leq \sum_{|n - x| \leq \frac{x}{(\log x)^{(2R)^{[K : \Q]} + R + 1}}} (2R)^{[K : \Q] \omega(n)},
\]
where we assume without loss of generality that $R$ is an integer greater than $1$. The latter sum is estimated in \cite[Theorem 7.17]{MoVa} with Dirichlet's hyperbola method.

We next move the path of integration. Note that $F(s, z)$ has a branch point at $s = 1$ if $z$ is not an integer. For this reason, we move the path of integration in such a way to avoid this branch point. Put $b = 1 - c/\log T$, where $c$ is a small positive constant. Let $\mathcal{C}_1$ be the polygonal path with vertices $a - iT, b - iT, b - i/\log x$, let $\mathcal{C}_2$ be the line segment from $b - i/\log x$ to $1 - i/\log x$, followed by a semicircle $\{1 + e^{i \theta}/\log x: -\pi/2 \leq \theta \leq \pi/2\}$, and a line segment from $1 + i/\log x$ to $b + i/\log x$, and finally let $\mathcal{C}_3$ be the polygonal path with vertices $b + i/\log x, b + iT, a + iT$. 

Let $\mathcal{D}$ be the region enclosed by $\mathcal{C}_1, \mathcal{C}_2, \mathcal{C}_3$ and the line segment from $a - iT$ to $ a + iT$. If $c$ is sufficiently small, then $L(s, \chi)$ has no zeroes in the region $\mathcal{D}$ by \cite[Theorem 5.10]{IK}. Clearly, $H(s, z)$ also has no zeroes in $\mathcal{D}$ provided that $c$ is sufficiently small. Since the union of $\mathcal{D}$ with the region $\text{Re}(s) > 1$ is still simply connected, $\log L(s, \chi)$ and $\log H(s, z)$ are also well-defined in this region.

The main term comes from the integral over $\mathcal{C}_2$, and is extracted in exactly the same way as in the proof of \cite[Theorem 7.17]{MoVa}. Finally, Montgomery--Vaughan estimate the integrals on the paths $\mathcal{C}_1$ and $\mathcal{C}_3$ by appealing to bounds for $\zeta(s)$, see their \cite[Theorem 6.7]{MoVa}. Hence we need to supply similar bounds for $\zeta_K(s)$ and $L(s, \chi)$. These can be derived by following the proof of \cite[Theorem 6.7]{MoVa}, where we use \cite[Proposition 5.7, (2)]{IK} as a replacement for \cite[Lemma 6.4]{MoVa}.
\end{proof}

Write $\ast$ for the Dirichlet convolution on $\mathcal{I}_K$. In Section \ref{sTheorems} we will combine Theorem \ref{tSelbergDelange} with the following general lemma on convolutions.

\begin{lemma}
\label{lConvolution}
Let $f, g: \mathcal{I}_K \rightarrow \mathbb{R}$ be functions such that
\[
\sum_{N_{K/\Q}(I) \leq x} f(I) = C_1x (\log x)^A + O(x (\log x)^{A - \delta}), \sum_{N_{K/\Q}(I) \leq x} g(I) = C_2x (\log x)^B + O(x (\log x)^{B - \delta})
\]
for some real numbers $A, B > -1$, $C_1, C_2 > 0$ and $0 < \delta < 1$. Then there is $C_3 > 0$ such that
\[
\sum_{N_{K/\Q}(I) \leq x} (f \ast g)(I) = C_3 x (\log x)^{A + B + 1} + O(x (\log x)^{A + B + 1 - \delta}).
\]
\end{lemma}

\begin{proof}
It follows from Dirichlet's hyperbola method that
\[
\sum_{N_{K/\Q}(I) \leq x} (f \ast g)(I) = \sum_{N_{K/\Q}(IJ) \leq x} f(I) g(J)
\]
equals
\begin{multline}
\label{eDirichlet}
\sum_{N_{K/\Q}(I) \leq \sqrt{x}} \sum_{N_{K/\Q}(J) \leq \frac{x}{N_{K/\Q}(I)}} f(I) g(J) + \sum_{N_{K/\Q}(J) \leq \sqrt{x}} \sum_{N_{K/\Q}(I) \leq \frac{x}{N_{K/\Q}(J)}} f(I) g(J)  \\
- \sum_{N_{K/\Q}(I) \leq \sqrt{x}} \sum_{N_{K/\Q}(J) \leq \sqrt{x}} f(I) g(J).
\end{multline}
The latter sum is at most $O(x (\log x)^{A + B})$. Since the first two sums in equation (\ref{eDirichlet}) play a symmetric role, we shall only treat the first sum. The first sum equals
\begin{align}
\label{eMainTerm}
\sum_{N_{K/\Q}(I) \leq \sqrt{x}} \frac{C_2xf(I)}{N_{K/\Q}(I)} \left(\log \frac{x}{N_{K/\Q}(I)}\right)^B = C_2x \hspace{-0.6cm} \sum_{N_{K/\Q}(I) \leq \sqrt{x}} \frac{f(I)}{N_{K/\Q}(I)} \left(\log x - \log N_{K/\Q}(I)\right)^B
\end{align}
up to an error of size bounded by
\begin{align}
\label{eEasyError}
O\left(x \sum_{N_{K/\Q}(I) \leq \sqrt{x}} \frac{f(I)}{N_{K/\Q}(I)} \left(\log x - \log N_{K/\Q}(I)\right)^{B - \delta}\right).
\end{align}
We shall give an asymptotic formula for equation (\ref{eMainTerm}), from which it will also be clear how to treat the error term in equation (\ref{eEasyError}). Put
\[
F(t) := \sum_{N_{K/\Q}(I) \leq t} f(I),
\]
so that we have the formula
\begin{align}
\label{eFtAsym}
F(t) = C_1t (\log t)^A + O(t (\log t)^{A - \delta})
\end{align}
by assumption. Partial summation shows that
\begin{multline*}
\sum_{N_{K/\Q}(I) \leq \sqrt{x}} \frac{f(I)}{N_{K/\Q}(I)} \left(\log x - \log N_{K/\Q}(I)\right)^B = \\
\frac{F(\sqrt{x}) \cdot (\frac{1}{2} \log x)^B}{\sqrt{x}} - \int_1^{\sqrt{x}} F(t) d\left(\frac{(\log x - \log t)^B}{t}\right).
\end{multline*}
The first term is $O((\log x)^{A + B})$. Plugging in equation (\ref{eFtAsym}) shows that the second term above equals
\[
C_1 \int_1^{\sqrt{x}} \frac{(\log t)^A (\log x - \log t)^B}{t} dt + O\left((\log x)^{A + B + 1 - \delta}\right).
\]
Recall the Taylor expansion, valid for $-\log x < \log t < \log x$
\[
(\log x - \log t)^B = (\log x)^B \left(1 - \frac{\log t}{\log x}\right)^B = (\log x)^B \sum_{k = 0}^\infty \binom{B}{k} \left(\frac{-\log t}{\log x}\right)^k,
\]
where $\binom{B}{k}$ is the generalized binomial coefficient. Since the Taylor expansion converges uniformly for $1 \leq t \leq \sqrt{x}$, we may switch the infinite sum and the integral to obtain
\[
(\log x)^B \sum_{k = 0}^\infty \frac{(-1)^k \binom{B}{k}}{(\log x)^k} \int_1^{\sqrt{x}} \frac{(\log t)^{A + k}}{t} dt = (\log x)^{A + B + 1} 2^{-A - 1} \sum_{k = 0}^\infty \frac{(-1)^k \binom{B}{k}}{2^k(A + k + 1)},
\]
where we used that $A > -1$ to compute the integral. We conclude that equation (\ref{eMainTerm}) equals
\[
C_1C_2 2^{-A - 1} \sum_{k = 0}^\infty \frac{(-1)^k \binom{B}{k}}{2^k(A + k + 1)} x (\log x)^{A + B + 1} + O\left(x (\log x)^{A + B + 1 - \delta}\right).
\]
Set
\[
C_3 := C_1C_2 \left(2^{-A - 1} \sum_{k = 0}^\infty \frac{(-1)^k \binom{B}{k}}{2^k(A + k + 1)} + 2^{-B - 1} \sum_{k = 0}^\infty \frac{(-1)^k \binom{A}{k}}{2^k(B + k + 1)}\right).
\]
It remains to show that $C_3 > 0$. But we have the lower bound
\[
\int_1^{\sqrt{x}} \frac{(\log t)^A (\log x - \log t)^B}{t} dt \gg \int_1^{\sqrt{x}} \frac{(\log t)^{A + B}}{t} dt \gg (\log x)^{A + B + 1},
\]
and this completes the proof.
\end{proof}

Finally, we will need the following version of the Siegel--Walfisz theorem.

\begin{theorem}
\label{tCheb}
Let $A > 0$ be a given real number and let $K$ be a fixed number field. Then we have for all $X > 2$, all Galois extensions $L/K$ with $[L : K] < A$ and $N_{K/\Q}(\Delta(L/K)) \leq (\log X)^A$ and all conjugacy classes $C$ of $\Gal(L/K)$
\[
\frac{\#\{\mathfrak{p} \in \Omega_K : N_{K/\Q}(\mathfrak{p}) \leq X, \mathfrak{p} \textup{ unr. in } L, \textup{Frob}_\mathfrak{p} = C\}}{\#\{\mathfrak{p} \in \Omega_K : N_{K/\Q}(\mathfrak{p}) \leq X\}} = \frac{\# C}{\# \Gal(L/K)} \textup{Li}(X) + O\left(\frac{X}{(\log X)^A}\right),
\]
where the implied constant depends only on $A$ and $K$.
\end{theorem}

\begin{proof}
This follows immediately from \cite[Theorem 1.1]{TZ}, were it not for potential Siegel zeroes. To control a potential Siegel zero of $\zeta_L(s)$, we apply the ineffective Brauer--Siegel theorem, which yields for every $\epsilon > 0$
\[
\text{Res}_{s = 1} \zeta_L(s) \gg_{\epsilon, K} \frac{1}{\Delta(L/\Q)^\epsilon}.
\]
Picking $\epsilon$ sufficiently small in terms of $A$ gives the desired lower bound for $1 - \beta$ by \cite[Theorem 1]{Louboutin}.
\end{proof}

\section{Proof of main theorems}
\label{sTheorems}
In this section we prove our three main results in complete generality. Recall that $l_G$ is the smallest prime divisor of $\# G$ and that $I(G)$ is the subset of $g \in G$ with order equal to $l_G$. Set $H(G) := I(G) \cup \{\text{id}\}$. We begin with the generalization of Theorem \ref{Main1}. Recall that
$$
i(G, K) := \frac{\#I(G)}{[K(\zeta_{l_G}) : K]}.
$$ 

\subsection{Proof of Theorem \ref{tUpper} and Theorem \ref{tUpper2}}
In this subsection we prove Theorem \ref{tUpper} and Theorem \ref{tUpper2}.

\begin{theorem} 
\label{mt3: Upper bound general}
Let $G$ be a finite non-trivial nilpotent group and let $K$ be a number field. Then there exists a constant $c \in \mathbb{R}_{>0}$ such that 
$$
\#\{\psi \in \textup{Epi}_{\textup{top.gr.}}(G_K, G) : |N_{K/\mathbb{Q}}(\textup{Disc}(\psi))| \leq X \} \leq c \cdot X^{a(G)} \cdot \log(X)^{i(G, K) - 1}
$$
for all $X \in \mathbb{R}_{> 2}$.
\end{theorem}

\begin{proof}
Write 
$$
e_g := \#G \left(1 - \frac{1}{\#\langle g \rangle} \right).
$$
It follows from Proposition \ref{reading off discriminants: in general plus} and Proposition \ref{epi iff cordinate are epi} that it suffices to bound
$$
\#\left\{(v_{g, j}(1), v_{g, j}(2))_{j \in [c], g \in G(l_j) - \{\textup{id}\}} : \left|N_{K/\mathbb{Q}} \left(\prod_{g \in G - \{\text{id}\}} v_g(2)^{e_g}\right)\right| \leq X\right\},
$$
where
\[
v_g(2) := \prod_{\substack{\mathfrak{p} \\ \forall j \in [c] \forall h \in G(l_j) - \{\text{id}\} : \mathfrak{p} \mid v_{h, j}(2) \Leftrightarrow g_j = h}} \mathfrak{p}
\]
with $g = (g_1, \dots, g_c)$. Here the variables $(v_{g, j}(1), v_{g, j}(2))_{g \in G(l_j) - \{\textup{id}\}}$ are in $\text{Prim}(\mathcal{S}_{l_j}^{G(l_j) - \{\text{id}\}})$ for every $j \in [c]$. We now drop the following conditions
\begin{itemize}
\item we drop the condition that $v_{g, j}(2)$ is supported outside $S_{\text{clean}}(l_j)$. We also drop the condition that $v_{g, j}(2)$ is supported in $\widetilde{\Omega}_K(l_j)$ except if $j = 1$ and $g \in I(G)$;
\item we drop the coprimality conditions between the $v_{g, j}(2)$, except that we remember that $v_{g, 1}(2)$ and $v_{g', 1}(2)$ are squarefree and coprime for $g, g' \in I(G)$.
\end{itemize} 
Since there are only finitely many possibilities (depending on $G$ and $K$) for $v_{g, j}(1)$, the above set has size bounded by
\[
\ll_{G, K} \sum_{\left|N_{K/\mathbb{Q}} \left(\prod_{g \in G - \{\text{id}\}} v_g(2)^{e_g}\right)\right| \leq X} 1.
\]
Here the sum runs over variables $v_g(2)$ for $g \in G - \{\text{id}\}$, where 
\begin{itemize}
\item for $g \in G - H(G)$, the variable $v_g(2)$ is an arbitrary integral ideal of $\mathcal{O}_K$;
\item for $g \in I(G)$, the variables $v_g(2)$ are squarefree and supported in $\widetilde{\Omega}_K(l_G)$;
\item for distinct $g, g' \in I(G)$, the variables $v_g(2)$ and $v_{g'}(2)$ are coprime.
\end{itemize}
We pull out the variables $v_g(2)$ for which $g$ is not in $I(G)$. For such variables $v_g(2)$ we know that 
\begin{align}
\label{eLargeVariables}
e_g > \#G \left(1 - \frac{1}{l_G}\right) = a(G)^{-1}.
\end{align}
Hence we get an upper bound
\begin{align}
\label{eSmallVariablesIn}
\ll \hspace{-0.5cm} \sum_{|N_{K/\Q}(\prod_{g \in G - H(G)} v_g(2)^{e_g})| \leq X} \sum_{\substack{|N_{K/\Q}(\prod_{h \in I(G)} v_h(2)^{a(G)^{-1}})| \leq \frac{X}{|N_{K/\Q}(\prod_{g \in G - H(G)} v_g(2)^{e_g})|} \\ v_h(2) \text{ squarefree and pairwise coprime} \\ v_h(2) \text{ supported in } \widetilde{\Omega}_K(l_G)}} \hspace{-1cm} 1.
\end{align}
Letting $I$ be the product of $v_h(2)$ with $h \in I(G)$, we see that $I$ is a squarefree ideal supported in $\widetilde{\Omega}_K(l_G)$. Equivalently, all prime divisors of $I$ split in the extension $K(\zeta_{l_G})/K$. By Theorem \ref{tSelbergDelange}, the inner sum is bounded by
\[
\sum_{\substack{|N_{K/\Q}(I^{a(G)^{-1}})| \leq \frac{X}{|N_{K/\Q}(\prod_{g \in G - H(G)} v_g(2)^{e_g})|} \\ I \text{ supported in } \widetilde{\Omega}_K(l_G)}} \mu^2(I) \# I(G)^{\omega(I)} \ll_{G, K} \frac{X^{a(G)} (\log X)^{i(G, K) - 1}}{|N_{K/\Q}(\prod_{g \in G - H(G)} v_g(2)^{e_g a(G)})|}.
\]
Plugging this back in equation (\ref{eSmallVariablesIn}) yields
\begin{multline*}
X^{a(G)} \cdot \log(X)^{i(G, K) - 1} \sum_{|N_{K/\Q}(\prod_{g \in G - H(G)} v_g(2)^{e_g})| \leq X} \frac{1}{|N_{K/\Q}(\prod_{g \in G - H(G)} v_g(2)^{e_g a(G)})|} \ll \\
X^{a(G)} \cdot \log(X)^{i(G, K) - 1} \prod_{g \in G - H(G)} \left(\sum_{|N_{K/\Q}(v_g(2)^{e_g})| \leq X} \frac{1}{|N_{K/\Q}(v_g(2)^{e_g a(G)})|}\right).
\end{multline*}
The inner sum converges since $e_g a(G) > 1$ by equation (\ref{eLargeVariables}), and this completes the proof.
\end{proof}

Now let $G := G(l_1) \times \dots \times G(l_c)$ be a finite non-trivial nilpotent group such that the elements of $I(G)$ are pairwise commuting. In this case we are going to prove an upper bound matching the prediction of Malle's conjecture. As before we assume that the elements of $\{l_1, \ldots, l_c\}$ are increasingly ordered, so that $l_1$ equals to $l_G$. Under these assumptions $H(G) = I(G) \cup \{\text{id}\}$ is a $\mathbb{F}_{l_1}$ vector space and furthermore characteristic (hence normal) in $G(l_1) \subseteq G$. Write $h(G)$ for the $\mathbb{F}_{l_1}$-dimension of $H(G)$. Observe that $\frac{G}{H(G)}$ naturally acts on $H(G)$ by conjugation. For each $h \in H(G)$ we denote by
$$
\text{Stab}_{\frac{G}{H(G)}}(h)
$$
the stabilizer of $h$ under the above action, which is a subgroup of $\frac{G}{H(G)}$. Observe that
$$
\#\text{Conj}_G(h) = \left[\frac{G}{H(G)} : \text{Stab}_{\frac{G}{H(G)}}(h)\right].
$$

For each $2 \leq j \leq c$, we filter $G(l_j)$ by any admissible sequence
$$
\{(G_{i_j}(l_j), \theta_{i_j})\}_{i_j \in [r_j]}.
$$
Instead for $l_1$ we filter $G(l_1)$ by an admissible sequence 
$$
\{(G_{i_1}(l_1), \theta_{i_1}) \}_{i_1 \in [r_1]}
$$
such that the kernel of the projection map from $G(l_1) = G_{r_1}(l_1)$ to $G_{r_1-h(G)}(l_1)$ coincides with $H(G)$. In other words $H(G)$ equals the subset of vectors in $\mathbb{F}_{l_1}^{r_1}$ with last $r_1 - h(G)$ coordinates equal to $0$. 

Fix now 
$$
\psi \in \text{Epi}_{\text{top.gr.}}(G_K, \frac{G}{H(G)}).
$$
Let us denote by
$$
\textup{Prim}(\mathcal{S}^{G - \{\textup{id}\}})(\textup{solv.})(\psi)
$$
the subset of
$$
\textup{Prim}(\mathcal{S}^{G - \{\textup{id}\}})(\textup{solv.}) 
$$
such that the induced epimorphism to $\frac{G}{H(G)}$, by means of the canonical projection, coincides with $\psi$. Observe that if $v := (v_{g, j}(1),v_{g, j}(2))_{j \in [c], g \in G(l_j)-\{\textup{id}\}} \in \textup{Prim}(\mathcal{S}^{G - \{\textup{id}\}})(\textup{solv.})(\psi)$, and $h \in H(G)$ and $\mathfrak{q} \mid v_{h}(2)$, then $\mathfrak{q}$ is unramified in the $\frac{G}{H(G)}$-extension given by $\psi$. Hence $\psi(\text{Frob}_{\mathfrak{q}})$ is well-defined; we remind the reader that $\text{Frob}_{\mathfrak{q}}$ depends on the choice of the embedding $i_{\mathfrak{q}}$ fixed so far in the paper. 

\begin{proposition} 
\label{contained in product set} 
Notation as immediately above this proposition. Then for each
$$
v := (v_{g, j}(1),v_{g, j}(2))_{j \in [c], g \in G(l_j)-\{\textup{id}\}} \in \textup{Prim}(\mathcal{S}^{G - \{\textup{id}\}})(\textup{solv.})(\psi),
$$
for each $h \in H(G)$ and each $\mathfrak{q} \mid v_{h}(2)$ we have that $\psi(\textup{Frob}_{\mathfrak{q}}) \in \textup{Stab}_{\frac{G}{H(G)}}(h)$.
\end{proposition}

\begin{proof}
This is an immediate consequence of Proposition \ref{epi iff cordinate are epi} and Proposition \ref{solvable iff commutative}.
\end{proof}

As mentioned above $\text{Frob}_\mathfrak{q}$ depends on the choice of embedding $i_\mathfrak{q}$. Unfortunately this means that $\text{Frob}_\mathfrak{q}$ might not be equidistributed (as $\mathfrak{q}$ varies) for some choices of the embeddings $i_\mathfrak{q}$. But since we are free to choose the embeddings as we like, we are able to work around this.

\begin{theorem} 
\label{mt5: right upper bound} 
Let $G$ be a finite non-trivial nilpotent group and let $K$ be a number field. Suppose that all elements of $I(G)$ commute with each other. Then there exists a constant $c \in \mathbb{R}_{>0}$ such that 
$$
\#\{\psi \in \textup{Epi}_{\textup{top.gr.}}(G_K, G): |N_{K/\mathbb{Q}}(\textup{Disc}(\psi))| \leq X \} \leq c \cdot X^{a(G)} \cdot \textup{log}(X)^{b(G, K) - 1}
$$
for all $X \in \mathbb{R}_{> 2}$.
\end{theorem}

\begin{proof}
We start as in the proof of Theorem \ref{mt3: Upper bound general} and we get the same upper bound as in equation (\ref{eSmallVariablesIn}) except that the $v_h(2)$ are now also such that $\mathfrak{q} \mid v_h(2)$ implies $\psi'(\textup{Frob}_{\mathfrak{q}}) \in \textup{Stab}_{\frac{G}{H(G)}}(h)$, where $\psi': G_K \rightarrow G/H(G)$ is the composition of $\psi$ with the quotient map $G \twoheadrightarrow G/H(G)$. More precisely, we see that
\[
\#\{\psi \in \textup{Epi}_{\textup{top.gr.}}(G_K, G): |N_{K/\mathbb{Q}}(\textup{Disc}(\psi))| \leq X \}
\]
is bounded by
\begin{align}
\label{eStart}
\sum_{|N_{K/\Q}(\prod_{g \in G - H(G)} v_g(2)^{e_g})| \leq X} \sum_{\substack{|N_{K/\Q}(\prod_{h \in I(G)} v_h(2)^{a(G)^{-1}})| \leq \frac{X}{|N_{K/\Q}(\prod_{g \in G - H(G)} v_g(2)^{e_g})|} \\ v_h(2) \text{ squarefree and pairwise coprime} \\ v_h(2) \text{ supported in } \widetilde{\Omega}_K(l_G) \\ \mathfrak{p} \mid v_h(2) \Rightarrow \psi'(\text{Frob}_\mathfrak{p}) \in \text{Stab}_{G/H(G)}(h)}} \hspace{-1cm} 1,
\end{align}
where $\psi'$ is the map associated to the tuple $(v_g(1), v_g(2))$ as $g$ runs through the elements that are zero on the coordinates corresponding to $H(G)$ (of course ignoring all tuples $(v_g(1), v_g(2))$ that map to $\bullet$).

We split the sum depending on
\begin{align}
\label{eSmallK}
\left|N_{K/\Q}\left(\prod_{g \in G - H(G)} v_g(2)^{e_g}\right)\right| \leq (\log X)^{A_1}.
\end{align}
If we pick $A_1 > 0$ large enough (depending on $G$ and $K$) , the terms with
\[
\left|N_{K/\Q}\left(\prod_{g \in G - H(G)} v_g(2)^{e_g}\right)\right| > (\log X)^{A_1}
\]
can be bounded as in the proof of Theorem \ref{mt3: Upper bound general}. Hence it remains to bound the terms satisfying equation (\ref{eSmallK}). This implies that $\text{Disc}(\psi') \leq (\log X)^{A_2}$ for a constant $A_2$ depending only on $G$ and $K$. The Rosser--Iwaniec sieve \cite[Lemma 3]{Coleman} gives some $A_3 > 0$ such that
\begin{align}
\label{eSelberg}
\sum_{\substack{N_{K/\Q}(I) \leq X \\ \mathfrak{p} \mid I \Rightarrow \mathfrak{p} \in \widetilde{\Omega}_K(l_G) \\ \mathfrak{p} \mid I \Rightarrow \psi'(\text{Frob}_\mathfrak{p}) \in \text{Stab}_{G/H(G)}(h)}} \mu^2(I) \ll_{G, K} \frac{X}{\log X} \cdot \prod_{\substack{N_{K/\Q}(\mathfrak{p}) \leq X^{A_3} \\ \mathfrak{p} \in \widetilde{\Omega}_K(l_G) \\ \psi'(\text{Frob}_\mathfrak{p}) \in \text{Stab}_{G/H(G)}(h)}} \left(1 + \frac{1}{N_{K/\Q}(\mathfrak{p})}\right).
\end{align}
Let $M$ be the compositum of all the extensions corresponding to a map $\psi': G_K \rightarrow G/H(G)$ with $\text{Disc}(\psi') \leq (\log X)^{A_1}$. Let $\text{Emb}(X)$ be the set of functions $f$ that send a place $\mathfrak{p} \in \Omega_K$ with $N_{K/\Q}(\mathfrak{p}) \leq X$ to a place of $M$ above $\mathfrak{p}$. If a function $f$ is given, it makes sense to speak of $\text{Frob}_\mathfrak{p}$ as an element of $\Gal(M/K)$ and its quotients.

We now call a function $f \in \text{Emb}(X)$ $A_4$-unfavorable in case there exists $\psi': G_K \rightarrow G/H(G)$ with $\text{Disc}(\psi') \leq (\log X)^{A_1}$ and there exists $g \in G/H(G)$ with
\begin{align}
\label{eEmb}
\left|\frac{\#\{\mathfrak{p} \in \widetilde{\Omega}_K(l_G) : N_{K/\Q}(\mathfrak{p}) \leq X, \psi'(\textup{Frob}_\mathfrak{p}) = g\}}{\#\{\mathfrak{p} \in \Omega_K : N_{K/\Q}(\mathfrak{p}) \leq X\}} - \frac{\textup{Li}(X)}{[K(\zeta_{l_G}) : K] \cdot \# \Gal(L/K)}\right| \geq \frac{X}{(\log X)^{A_4}},
\end{align}
where $L$ is the field corresponding to $\psi'$. Recall that $\mathfrak{p} \in \widetilde{\Omega}_K(l_G)$ is equivalent to $\mathfrak{p}$ splitting in $K(\zeta_{l_G})$, except for finitely many bad primes. Now we apply Theorem \ref{tCheb} with a very large $A$. Since the set of $\psi': G_K \rightarrow G/H(G)$ with $\text{Disc}(\psi') \leq (\log X)^{A_1}$ is bounded by $(\log X)^{A_5}$ for some $A_5 > 0$ depending only on $G$ and $K$ (see Theorem \ref{mt3: Upper bound general} for example), it follows from Theorem \ref{tCheb} and Hoeffding's inequality that 
\[
\frac{\#\{f \in \text{Emb}(X) : f \text{ is } A_4\text{-unfavorable}\}}{\#\{f \in \text{Emb}(X)\}}
\]
is small for any fixed $A_4 > 0$. In particular, we can fix one choice of embeddings that is not, say, $100$-unfavorable. It follows from partial summation and equations (\ref{eSelberg}) and (\ref{eEmb}) that
\[
\sum_{\substack{N_{K/\Q}(I) \leq X \\ \mathfrak{p} \mid I \Rightarrow \mathfrak{p} \in \widetilde{\Omega}_K(l_G) \\ \mathfrak{p} \mid I \Rightarrow \psi'(\text{Frob}_\mathfrak{p}) \in \text{Stab}_{G/H(G)}(h)}} \mu^2(I) \ll_{G, K} \frac{X}{(\log X)^{1 - \frac{\# \text{Stab}_{G/H(G)}(h)}{[K(\zeta_{l_G}) : K] \cdot \#(G/H(G))}}}.
\]
We now drop the condition that the $v_h(2)$ are pairwise coprime in equation (\ref{eStart}). The theorem then follows from Lemma \ref{lbGKNil} and repeatedly applying Lemma \ref{lConvolution} to equation (\ref{eStart}).
\end{proof}

\subsection{Poitou--Tate duality}
In order to prove Theorem \ref{tMalle} it will be convenient to pick a favorable choice of the characters $\chi_{\mathfrak{q}}$. In particular we would like to show the following. Let $K$ be a number field, let $l$ be a prime number and let $S_{\text{clean}}(l)$ be as in Section \ref{sMain}. Then there exists an extension $L/K$ such that $\text{Frob}_{L/K}(\mathfrak{q}) = \text{Frob}_{L/K}(\mathfrak{q'})$ implies that the characters $\chi_{\mathfrak{q}}, \chi_{\mathfrak{q}'}: G_K \rightarrow \mathbb{F}_l$ can be chosen in such a way that their restriction to
\[
\bigoplus_{\mathfrak{q} \in S_{\text{clean}}(l)} \frac{H^1(G_{K_{\mathfrak{q}}},\mathbb{F}_l)}{H_{\text{unr}}^1(G_{K_{\mathfrak{q}}},\mathbb{F}_l)}
\]
is the same. Here we recall that $\chi_{\mathfrak{q}}$ is a character satisfying the following two properties: the place $\mathfrak{q}$ ramifies in the field corresponding to $\chi_{\mathfrak{q}}$, and furthermore any other ramified place must be in $S_{\text{clean}}(l)$. We will use the following form of Poitou--Tate duality to achieve our goal. Let us start by giving some background material on Selmer groups.

Let $M$ be a finite, discrete $G_K$-module. We define for each place $v$ the unramified classes to be
\[
H^1_{\text{unr}}(G_{K_v}, M) := \text{ker}\left(H^1(G_{K_v}, M) \rightarrow H^1(G_{K_v^{\text{unr}}}, M)\right)
\]
with $K_v^{\text{unr}}$ the maximal, unramified extension of $K_v$. A Selmer structure for $M$ is then a collection $\mathcal{L} = \{\mathcal{L}_v\}_v$, where each $\mathcal{L}_v$ is a subgroup of $H^1(G_{K_v}, M)$ such that $\mathcal{L}_v = H^1_{\text{unr}}(G_{K_v}, M)$ for all but finitely many places. The associated Selmer group $\text{Sel}_\mathcal{L}(G_K, M)$ is then the kernel of the map
\[
H^1(G_K, M) \rightarrow \prod_{v \in \Omega_K} H^1(G_{K_v}, M)/\mathcal{L}_v.
\]
Define $M^\ast = \text{Hom}(M, \Q/\Z(1))$, where $\Q/\Z(1)$ is the Tate twist of $\Q/\Z$. We have the local Tate pairing
\[
H^1(G_{K_v}, M) \times H^1(G_{K_v}, M^\ast) \rightarrow H^2(G_{K_v}, \Q/\Z(1)) \cong \text{Br}(K_v) \rightarrow \Q/\Z
\]
given by the cup product and the local invariant map. The dual Selmer structure is then defined to be the orthogonal complement of $\mathcal{L}_v$ under the local Tate pairing, which gives subspaces $\{\mathcal{L}_v^\ast\}_v$ of $H^1(G_{K_v}, M^\ast)$. If $v$ does not divide $|M|$ and the inertia group $I_{K_v}$ acts trivially on $M$, then it is known that $H^1_{\text{unr}}(G_{K_v}, M)$ and $H^1_{\text{unr}}(G_{K_v}, M^\ast)$ are orthogonal complements under the local Tate pairing. The dual Selmer group is defined as the kernel of the map
\[
H^1(G_K, M^\ast) \rightarrow \prod_{v \in \Omega_K} H^1(G_{K_v}, M^\ast)/\mathcal{L}_v^\ast.
\]

\begin{theorem}[Poitou--Tate duality]
\label{tPT}
Let $\mathcal{L}$ and $\mathcal{F}$ be Selmer structures such that $\mathcal{L}_v \subseteq \mathcal{F}_v$ for each $v$. Let $\Omega$ be a finite set of places such that $\mathcal{L}_v = \mathcal{F}_v$ for all $v \not \in \Omega$. Then we have exact sequences
\[
0 \rightarrow \textup{Sel}_\mathcal{L}(G_K, M) \rightarrow \textup{Sel}_\mathcal{F}(G_K, M) \rightarrow \bigoplus_{v \in \Omega} \mathcal{F}_v/\mathcal{L}_v
\]
and
\[
0 \rightarrow \textup{Sel}_{\mathcal{F}^\ast}(G_K, M^\ast) \rightarrow \textup{Sel}_{\mathcal{L}^\ast}(G_K, M^\ast) \rightarrow \bigoplus_{v \in \Omega} \mathcal{L}_v^\ast/\mathcal{F}_v^\ast.
\]
Now consider the pairing
\[
\bigoplus_{v \in \Omega} \mathcal{F}_v/\mathcal{L}_v \times \bigoplus_{v \in \Omega} \mathcal{L}_v^\ast/\mathcal{F}_v^\ast \rightarrow \Q/\Z,
\]
which is by definition the sum of the local Tate pairings at each $v \in \Omega$. Then the images of $\textup{Sel}_\mathcal{F}(G_K, M)$ in $\oplus_{v \in \Omega} \mathcal{F}_v/\mathcal{L}_v$ and $\textup{Sel}_{\mathcal{L}^\ast}(G_K, M^\ast)$ in $\oplus_{v \in \Omega} \mathcal{L}_v^\ast/\mathcal{F}_v^\ast$ are orthogonal complements.
\end{theorem}

\begin{proof}
See \cite[Theorem 2.3.4]{MR}.
\end{proof}

We remark that it is only the last part of the theorem that is deep. With this theorem in hand, it is now easy to control the local behavior at the places in $S_{\text{clean}}(l)$.

\begin{theorem}
\label{tCharLoc}
There exists an extension $L/K$ with the following property. Take any two primes $\mathfrak{p}, \mathfrak{p}' \in \widetilde{\Omega}_K(l)$. Assume that 
\[
\textup{Frob}_{L/K}(\mathfrak{p}) = \textup{Frob}_{L/K}(\mathfrak{p'}). 
\]
Then we can choose characters $\chi_{\mathfrak{p}}$ and $\chi_{\mathfrak{p}'}$ such that they have the same restriction to 
\[
\bigoplus_{\mathfrak{q} \in S_{\textup{clean}}(l)} \frac{H^1(G_{K_{\mathfrak{q}}}, \mathbb{F}_l)}{H_{\textup{unr}}^1(G_{K_{\mathfrak{q}}}, \mathbb{F}_l)}.
\]
\end{theorem}

\begin{remark}
By a choice for $\chi_{\mathfrak{p}}$ we mean a character from $G_K$ to $\mathbb{F}_l$ that is ramified at $\mathfrak{p}$, and unramified at all other places except possibly for those in $S_{\textup{clean}}(l)$.
\end{remark}

\begin{proof}
We apply Theorem \ref{tPT} as follows. We take $M = \mathbb{F}_l$ so that $M^\ast \cong \langle \zeta_l \rangle$. For now fix a finite place $w$ outside $S_{\textup{clean}}(l)$ but in $\widetilde{\Omega}_K(l)$. We take $\mathcal{L} = \{\mathcal{L}_v\}_v$ with $\mathcal{L}_v = H^1_{\text{unr}}(G_{K_v}, M)$ for all places $v$. Furthermore, we take $\mathcal{F}^w = \{\mathcal{F}_v^w\}_v$ with
\[
\mathcal{F}_v^w =
\left\{
\begin{array}{ll}
H^1(G_{K_v}, M)  & \mbox{if } v \in S_{\textup{clean}}(l) \cup \{w\} \\
H^1_{\text{unr}}(G_{K_v}, M) & \mbox{otherwise.}
\end{array}
\right.
\]
Finally, we take $\Omega$ to be the union of $S_{\textup{clean}}(l)$ with $\{w\}$. 

By Kummer theory we know that $H^1(G_K, M^\ast)$ can be identified with $K^\ast/K^{\ast l}$. A computation then shows that $\text{Sel}_{\mathcal{L}^\ast}(G_K, M^\ast)$ is given by the elements $\alpha \in K^\ast/K^{\ast l}$ that have valuation divisible by $l$ at all finite places $v$. Now take the field $L$ to be
\[
L := K(\zeta_l, \{\sqrt[l]{\alpha} : \alpha \in \text{Sel}_{\mathcal{L}^\ast}(G_K, M^\ast)\}).
\]
Suppose now that we are given two primes $\mathfrak{p}, \mathfrak{p}' \in \widetilde{\Omega}_K(l)$ with
\begin{align}
\label{eEqualFrob}
\textup{Frob}_{L/K}(\mathfrak{p}) = \textup{Frob}_{L/K}(\mathfrak{p'}).
\end{align}
Take two non-zero elements
\[
x_{\mathfrak{p}} \in \frac{H^1(G_{K_{\mathfrak{p}}}, \mathbb{F}_l)}{H_{\textup{unr}}^1(G_{K_{\mathfrak{p}}}, \mathbb{F}_l)}, \quad x_{\mathfrak{p}'} \in \frac{H^1(G_{K_{\mathfrak{p}'}}, \mathbb{F}_l)}{H_{\textup{unr}}^1(G_{K_{\mathfrak{p}'}}, \mathbb{F}_l)}.
\]
Observe that any character $\chi \in \text{Sel}_{\mathcal{F}^\mathfrak{p}}(G_K, M)$ restricting to $x_{\mathfrak{p}}$ is a valid choice of $\chi_{\mathfrak{p}}$, and similarly for $x_{\mathfrak{p}'}$ and $\chi_{\mathfrak{p}'}$. By construction of $S_{\textup{clean}}(l)$ we can find some 
\[
y \in \bigoplus_{\mathfrak{q} \in S_{\textup{clean}}(l)} \frac{H^1(G_{K_{\mathfrak{q}}}, \mathbb{F}_l)}{H_{\textup{unr}}^1(G_{K_{\mathfrak{q}}}, \mathbb{F}_l)}
\]
such that the pair $(x_\mathfrak{p}, y)$ is in the image of $\text{Sel}_{\mathcal{F}^{\mathfrak{p}}}(G_K, M)$. To complete the proof, we will show that there exists $\lambda \in \mathbb{F}_l^\ast$ such that $(\lambda x_{\mathfrak{p}'}, y)$ is in the image of $\text{Sel}_{\mathcal{F}^{\mathfrak{p}'}}(G_K, M)$. Consider the linear functionals $\varphi_\mathfrak{p}, \varphi_{\mathfrak{p}'} : \text{Sel}_{\mathcal{L}^\ast}(G_K, M^\ast) \rightarrow (\frac{1}{l}\Z_l/\Z_l)^n$ given respectively by
\[
\alpha \mapsto \text{inv}_\mathfrak{p}(x_\mathfrak{p} \cup \text{res}_\mathfrak{p}(\alpha)), \quad \alpha \mapsto \text{inv}_{\mathfrak{p}'}(x_{\mathfrak{p}'} \cup \text{res}_{\mathfrak{p}'}(\alpha)),
\]
where $\text{res}$ denotes the natural restriction map. We claim that $\text{ker}(\varphi_\mathfrak{p}) = \text{ker}(\varphi_{\mathfrak{p}'})$. But indeed, this follows from equation (\ref{eEqualFrob}) and
\[
\text{inv}_\mathfrak{p}(x_\mathfrak{p} \cup \text{res}_\mathfrak{p}(\alpha)) = 0 \Longleftrightarrow \mathfrak{p} \text{ splits completely in } K(\zeta_l, \sqrt[l]{\alpha}).
\]
We observe that $\text{ker}(\varphi_\mathfrak{p}) = \text{ker}(\varphi_{\mathfrak{p}'})$ implies that there exists $\lambda \in \mathbb{F}_l^\ast$ such that $\varphi_\mathfrak{p} = \lambda \varphi_{\mathfrak{p}'}$. Finally, recall that
\[
\text{inv}_\mathfrak{p}(x_\mathfrak{p} \cup \text{res}_\mathfrak{p}(\alpha))
\]
is also the local Tate pairing of $x_\mathfrak{p}$ with $\alpha$. The theorem now follows from Poitou--Tate duality and the fact that $\varphi_\mathfrak{p} = \lambda \varphi_{\mathfrak{p}'}$.
\end{proof}

\subsection{Proof of Theorem \ref{tMalle}}
Recall that the quantities $i(G, K)$ and $b(G, K)$ coincide if and only if $I(G)$ is entirely contained in the center of $G$. In this subsection we shall establish an asymptotic for $N(G, K, X)$ for such groups $G$. 

We start by generalizing Proposition \ref{Product structure}. Recall that for a finite group we denote by $Z(G)$ the center of $G$. Let now $G := G(l_1) \times \dots \times G(l_c)$ be a finite non-trivial group with each $G(l_i)$ an $l_i$-group. We assume that $I(G) \subseteq Z(G)$ and we order the primes $l_1, \dots, l_c$ such that $l_1 < \dots < l_c$, so that $l_1 = l_G$. 

In particular we see that $H(G) := I(G) \cup \{\text{id}\}$ is a vector space over $\mathbb{F}_{l_1}$, we denote by $h(G)$ its dimension. For each $2 \leq j \leq c$, we filter $G(l_j)$ by any admissible sequence
$$
\{(G_{i_j}(l_j), \theta_{i_j}) \}_{i_j \in [r_j]}.
$$
Instead for $l_1$ we filter $G(l_1)$ by an admissible sequence
$$
\{(G_{i_1}(l_1), \theta_{i_1}) \}_{i_1 \in [r_1]}
$$
such that the kernel of the projection map from $G(l_1) = G_{r_1}(l_1)$ to $G_{r_1 - h(G)}(l_1)$ coincides with $H(G)$. In other words $H(G)$ equals the subset of vectors in $\mathbb{F}_{l_1}^{r_1}$ with last $r_1 - h(G)$ coordinates equal to $0$. Let us denote by
$$
\text{Prim}(\mathcal{S}^{G - H(G)})(\text{solv.}) \subseteq \text{Prim}(\mathcal{S}_{l_1}^{G(l_1) - H(G)}) \times \prod_{2 \leq j \leq c} \text{Prim}(\mathcal{S}_{l_j}^{G(l_j) - \{\text{id}\}})
$$
the image of the projection map $\pi$ from $\text{Prim}(\mathcal{S}^{G - \{\text{id}\}})(\text{solv.})$ that drops the coordinates in $H(G)$. We also denote by $\left(\prod_{j \in [c]} \textup{Prim} \left(\mathcal{S}_{l_j}^{G(l_j) - \{\textup{id}\}} \right)\right)^{\circ}$ the subset of vectors having non-trivial coordinate in $H(G)$. 

\begin{proposition} 
\label{Product structure plus} 
We have
\begin{multline*}
\textup{Prim}(\mathcal{S}^{G - \{\textup{id}\}})(\textup{solv.}) \supseteq \\
\left(\textup{Prim} \left(\mathcal{S}_{l_1}^{H(G) - \{\textup{id}\}} \right) \times \textup{Prim}(\mathcal{S}^{G-H(G)})(\textup{solv.})\right) \bigcap \left(\prod_{j \in [c]} \textup{Prim} \left(\mathcal{S}_{l_j}^{G(l_j) - \{\textup{id}\}} \right)\right)^{\circ}.
\end{multline*}
Furthermore, we have
$$
\textup{Prim}(\mathcal{S}^{G - \{\textup{id}\}})(\textup{solv.}) \subseteq (\textup{Prim}(\mathcal{S}_{l_1}^{H(G) - \{\textup{id}\}}) \times \textup{Prim}(\mathcal{S}^{G - H(G)})(\textup{solv.})) \cap \prod_{j \in [c]} \textup{Prim} \left(\mathcal{S}_{l_j}^{G(l_j) - \{\textup{id}\}} \right).
$$
\end{proposition}

\begin{proof}
The proof is identical to the one given for Proposition \ref{Product structure}. 
\end{proof}

We are now ready to prove our main theorem.

\begin{theorem} 
\label{mt4: Asymptotic}
Let $G$ be a finite non-trivial nilpotent group such that $I(G)$ is entirely contained in the center of $G$. Then there exists a constant $c > 0$ such that
$$ 
\#\{\psi \in \textup{Epi}_{\textup{top.gr.}}(G_K, G) : |N_{K/\mathbb{Q}}(\textup{Disc}(\psi))| \leq X \} \sim c \cdot X^{a(G)} \cdot \textup{log}(X)^{b(G, K) - 1}.
$$
\end{theorem}

\begin{proof}
We start by labelling the elements of $\text{Prim}(\mathcal{S}^{G - H(G)})(\text{solv.})$ as $x_1, x_2, x_3, \dots$ and we write $L$ for the length of the sequence, where $L$ is possibly infinite. Recall that $\pi$ denotes the natural projection map from $\text{Prim}(\mathcal{S}^{G - \{\text{id}\}})(\text{solv.})$ to $\text{Prim}(\mathcal{S}^{G - H(G)})(\text{solv.})$. We have the decomposition
\[
\#\{\psi \in \textup{Epi}_{\textup{top.gr.}}(G_K, G) : |N_{K/\mathbb{Q}}(\textup{Disc}(\psi))| \leq X \} = \sum_{i = 1}^L \sum_{\substack{y \in \text{Prim}(\mathcal{S}^{G - \{\text{id}\}})(\text{solv.}) \\ \pi(y) = x_i \\ \text{Disc}(y) \leq X}} 1.
\]
We claim that for all $i$ there is a constant $b_{G, x_i} > 0$ such that
\begin{align}
\label{eSlowVarFixed}
\sum_{\substack{y \in \text{Prim}(\mathcal{S}^{G - \{\text{id}\}})(\text{solv.}) \\ \pi(y) = x_i \\ \text{Disc}(y) \leq X}} 1 \sim b_{G, x_i} \cdot X^{a(G)} \cdot \textup{log}(X)^{b(G, K) - 1}.
\end{align}
Once the claim is established, the remainder of the proof is identical to the proof of Theorem \ref{Main2}. To establish the claim, we take a more close look at the conditions imposed on $y$ for a given $x_i$. Write $x_i$ as a tuple $(v_{g'}(1), v_{g'}(2))_{g' \in G - H(G)}$. Let $(v_g(1), v_g(2))_{g \in I(G)}$ be in $\textup{Prim}(\mathcal{S}_{l_1}^{H(G) - \{\textup{id}\}})$. Fix,  for every $g = (g_1, \dots, g_c) \in I(G)$, any choice of 
\[
v_g(1) = (v_{g_1, 1}(1), \dots, v_{g_c, c}(1))
\]
such that $v_{h_j, j}(1)$ and $v_{h'_j, j}(1)$ are coprime (as defined in Section \ref{sMain}) for all distinct $h_j, h'_j \in G(l_j)$ and for all $j \in [c]$.

Then by Proposition \ref{Product structure plus} we see that $\{v_g(2)\}_{g \in I(G)}$ (together with $x_i$ and the variables $v_g(1)$ for $g \in I(G)$) is in $\textup{Prim}(\mathcal{S}^{G - \{\textup{id}\}})(\textup{solv.})$ if the $v_g(2)$ are squarefree non-trivial ideals and pairwise coprime, coprime to the $v_{g'}(2)$ for $g' \in G - H(G)$ and all prime divisors of $v_g(2)$ lie in $\widetilde{\Omega}_K(l_G)$. Conversely, if $\{v_g(2)\}_{g \in I(G)}$ (together with $x_i$ and the variables $v_g(1)$ for $g \in I(G)$) is in $\textup{Prim}(\mathcal{S}^{G - \{\textup{id}\}})(\textup{solv.})$, then the $v_g(2)$ are squarefree and pairwise coprime, coprime to the $v_{g'}(2)$ for $g' \in G - H(G)$ and all prime divisors of $v_g(2)$ lie in $\widetilde{\Omega}_K(l_G)$.

However, we would like to achieve some finer control. Indeed, Proposition \ref{reading off discriminants: in general plus} does not give us full control over the discriminant, but only over the part outside of some finite set of bad places $S$. To remedy this, we apply Theorem \ref{tCharLoc} with this set of places $S$. Let $L/K$ be the extension guaranteed by Theorem \ref{tCharLoc}; tracing through the proof, we see that we can take the field $L$ to be an elementary abelian extension of $K(\zeta_{l_G})$ of degree a power of $l_G$. Let $S_1, \dots, S_k$ be the conjugacy classes of $\Gal(L/K)$ that project trivially to the identity in $\Gal(K(\zeta_{l_G})/K)$. Now choose the characters $\chi_{\mathfrak{p}}$ as in Theorem \ref{tCharLoc}. Concretely this means that
\[
\text{Frob}_{L/K}(\mathfrak{p}) = \text{Frob}_{L/K}(\mathfrak{p}')
\]
implies that the characters $\chi_{\mathfrak{p}}$ and $\chi_{\mathfrak{p}'}$ have the same restriction to
\[
\bigoplus_{\mathfrak{q} \in S} \frac{H^1(G_{K_\mathfrak{q}}, \mathbb{F}_{l_G})}{H^1_{\text{unr}}(G_{K_\mathfrak{q}}, \mathbb{F}_{l_G})}.
\]
We recall that we have already fixed the finitely many possibilities of $\{v_g(1)\}_{g \in I(G)}$. Motivated by Theorem \ref{tCharLoc}, we further split the sum in equation (\ref{eSlowVarFixed}) depending on the values of 
\[
a_{j, g} := \omega_{S_j}(v_g(2)) \bmod l_G
\] 
for $g \in I(G)$. Indeed, if two vectors $\{v_g(2)\}_{g \in I(G)}, \{w_g(2)\}_{g \in I(G)}$ are such that
\[
\omega_{S_j}(v_g(2)) \equiv \omega_{S_j}(v_g(2)) \bmod l_G
\]
for all $j$ and all $g$, then the $S$-part of the discriminant is the same for the two extensions corresponding to respectively $\{v_g(2)\}_{g \in I(G)}$ and $\{w_g(2)\}_{g \in I(G)}$. Hence the sum in equation (\ref{eSlowVarFixed}) can be upper bounded by finitely many sums of the shape
\begin{align}
\label{eSumOverK}
\sum_{\substack{|\prod_{g \in I(G)} N_{K/\Q}(v_g(2))| \leq C(x_i, \mathbf{a}) X^{a(G)} \\ \omega_{S_j}(v_g(2)) \equiv a_{j, g} \bmod l_G \\ v_g(2) \text{ squarefree and pairwise coprime } \\ \mathfrak{p} \mid v_g(2) \Rightarrow \mathfrak{p} \in \widetilde{\Omega}_K(l_G) \\ \gcd(v_g(2), \mathfrak{p}) = 1 \ \forall \mathfrak{p} \in \mathcal{P}(x_i)}} 1,
\end{align}
where $\mathbf{a}$ is any vector in $\mathbb{F}_{l_G}^{[k] \times I(G)}$ with entries $a_{j, g}$, $C(x_i, \mathbf{a})$ is a positive real number depending only on $K$, $G$, $x_i$, $\{v_g(1)\}_{g \in I(G)}$, $\mathbf{a}$ and $\mathcal{P}(x_i)$ is a finite set depending only on $K$, $G$, $x_i$, $\{v_g(1)\}_{g \in I(G)}$. Similarly, the sum in equation (\ref{eSlowVarFixed}) can be lower bounded by finitely many sums of the shape
\begin{align}
\label{eSumOverK2}
\sum_{\substack{|\prod_{g \in I(G)} N_{K/\Q}(v_g(2))| \leq C(x_i, \mathbf{a}) X^{a(G)} \\ v_g(2) \neq (1) \\ \omega_{S_j}(v_g(2)) \equiv a_{j, g} \bmod l_G \\ v_g(2) \text{ squarefree and pairwise coprime } \\ \mathfrak{p} \mid v_g(2) \Rightarrow \mathfrak{p} \in \widetilde{\Omega}_K(l_G) \\ \gcd(v_g(2), \mathfrak{p}) = 1 \ \forall \mathfrak{p} \in \mathcal{P}(x_i)}} 1.
\end{align}
We shall give an asymptotic for equation (\ref{eSumOverK}). From the proof it shall be clear how to extract a matching asymptotic for equation (\ref{eSumOverK2}), which implies the claimed equation (\ref{eSlowVarFixed}). It remains to give an asymptotic for equation (\ref{eSumOverK}).

Our first step is to pass to $K(\zeta_{l_G})$. We define $\mathcal{P}'(x_i)$ as the set of finite places $w$ of $K(\zeta_{l_G})$ such that the place $v$ of $K$ below $w$ is not in $\mathcal{P}(x_i)$ and splits completely in $K(\zeta_{l_G})/K$. Then equation (\ref{eSumOverK}) becomes
\begin{align}
\label{eSumOverKzeta}
\sum_{\substack{|\prod_{g \in I(G)} N_{K(\zeta_{l_G})/\Q}(v_g(2))| \leq C(x_i, \mathbf{a}) X^{a(G)} \\ \omega_{S_j}(v_g(2)) \equiv a_{j, g} \bmod l_G \\ v_g(2) \text{ squarefree and pairwise coprime } \\ \gcd(v_g(2), \mathfrak{p}) = 1 \ \forall \mathfrak{p} \in \mathcal{P}'(x_i)}} \frac{1}{[K(\zeta_{l_G}) : K]^{\sum_{g \in I(G)} \omega(v_g(2))}},
\end{align}
where the $v_g(2)$ are now ideals of $K(\zeta_{l_G})$. To evaluate this sum, define $f_j(I)$ to be the function on $\mathcal{I}_{K(\zeta_{l_G})}$ that sends $I$ to zero if $I$ is divisible by a square, by a prime $\mathfrak{p} \in \mathcal{P}'(x_i)$ or any $\mathfrak{p}$ with $\text{Frob}_\mathfrak{p} \not \in S_j$. If instead $I$ is a squarefree ideal entirely supported on primes $\mathfrak{p}$ with $\text{Frob}_\mathfrak{p} \in S_j$ and $\mathfrak{p} \not \in \mathcal{P}'(x_i)$, we define
\[
f_j(I) = \frac{\# \{(I_g)_{g \in I(G)} \in \mathcal{I}_{K(\zeta_{l_G})}^{I(G)} : \prod_{g \in I(G)} I_g = I, \omega_{S_j}(I_g) \equiv a_{j, g} \bmod l_G\}}{[K(\zeta_{l_G}) : K]^{\omega(I)}}.
\]
Having defined $f_j(I)$, we see that equation (\ref{eSumOverKzeta}) is simply
\begin{align}
\label{eConvolution}
\sum_{N_{K/\Q}(I) \leq C(x_i, \mathbf{a}) X^{a(G)}} (f_1 \ast \dots \ast f_k)(I).
\end{align}
We will now approximate $f_j(I)$ using some probability theory. Assume, for now, that $I$ is a squarefree ideal entirely supported on primes $\mathfrak{p}$ with $\text{Frob}_\mathfrak{p} \in S_j$ and $\mathfrak{p} \not \in \mathcal{P}'(x_i)$. First define for $(b_{j, g})_{g \in I(G)} \in \Z_{\geq 0}^{I(G)}$
\[
g_{j, (b_{j, g})_{g \in I(G)}}(I) = \frac{\# \{(I_g)_{g \in I(G)} \in \mathcal{I}_{K(\zeta_{l_G})}^{I(G)} : \prod_{g \in I(G)} I_g = I, \omega_{S_j}(I_g) = b_{j, g}\}}{\# I(G)^{\omega(I)}}
\]
Let us recall the multinomial distribution. As input it takes an integer $n$, which is the sample size, an integer $k$, which are the number of mutually exclusive events $E_1, \dots, E_k$, and real numbers $0 \leq p_1, \dots, p_k \leq 1$ such that $p_i$ is the probability of the event $E_i$. We further demand
\[
p_1 + \dots + p_k = 1.
\]
The multinomial distribution is then a vector $X = (X_1, \dots, X_k)$, where $X_i$ indicates the number of outcomes of event $E_i$ in $n$ independent samples. We now take $n = \omega(I)$, $k = \# I(G)$ and $p_1 = \dots = p_k = 1/k$. Then $g_{j, (b_{j, g})_{g \in I(G)}}(I)$ is the probability density function of the resulting multinomial distribution. Concretely,
\[
\mathbb{P}(X_g = b_{j, g} \text{ for all } g \in I(G)) = g_{j, (b_{j, g})_{g \in I(G)}}(I) = \frac{\omega(I)!}{\prod_{g \in I(G)} b_{j, g}!} \cdot \frac{1}{\# I(G)^{\omega(I)}},
\]
where we have relabelled the variables $X_i$ as $X_g$ with $g \in I(G)$. Now observe that 
\[
\frac{\omega(I)!}{\prod_{g \in I(G)} b_{j, g}!}
\]
are precisely the coefficients when one expands
\[
\left(\sum_{g \in I(G)} p_g\right)^{\omega(I)}
\]
as a polynomial in the variables $p_g$. Evaluating this polynomial at $p_g = \zeta_{l_G}^{c_g}/\# I(G)$ as $c_g$ runs through all vectors in $\mathbb{F}_{l_G}^{I(G)}$ shows that
\begin{align}
\label{eProbTrick}
\left(\frac{[K(\zeta_{l_G}) : K]}{\#I(G)}\right)^{\omega(I)}  f_j(I) &= l_G^{-\#I(G)} \sum_{(c_g)_{g \in I(G)} \in \mathbb{F}_{l_G}^{I(G)}} \left(\prod_{g \in I(G)} \zeta_{l_G}^{-c_g a_{j, g}}\right) \cdot \left(\sum_{g \in I(G)} \frac{\zeta_{l_G}^{c_g}}{\# I(G)}\right)^{\omega(I)} \nonumber \\
&= l_G^{-\#I(G) + 1} \cdot \mathbf{1}_{\omega(I) \equiv \sum_{g \in I(G)} a_{j, g} \bmod l_G} + O(\delta^{\omega(I)})
\end{align}
with $\delta < 1$. We apply Theorem \ref{tSelbergDelange} with 
\[
z = \frac{\#I(G)}{[K(\zeta_{l_G}) : K]} \cdot \zeta_{l_G}^a
\]
as $a$ runs through $0, \dots, l_G - 1$. We conclude that
\begin{multline}
\label{efjI}
\sum_{N_{K/\Q}(I) \leq X} f_j(I) = C X (\log X)^{\frac{\#I(G) \# S_j}{[K(\zeta_{l_G}) : K] \# \Gal(L/K(\zeta_{l_G}))} - 1}  + \\
O\left(X (\log X)^{\frac{\#I(G) \# S_j}{[K(\zeta_{l_G}) : K] \# \Gal(L/K(\zeta_{l_G}))} - 1 - \delta'}\right)
\end{multline}
for some constant $C > 0$ and some $\delta' > 0$. The theorem now follows from equation (\ref{efjI}), equation (\ref{eConvolution}) and repeated application of Lemma \ref{lConvolution}.
\end{proof}


\begin{thebibliography}{39}
\bibitem{Alberts1}
B. Alberts.
Statistics of the First Galois Cohomology Group: A Refinement of Malle's Conjecture.
\textit{arXiv preprint}, 1907.06289v2.

\bibitem{Alberts2}
B. Alberts.
The Weak Form of Malle's Conjecture and Solvable Groups.
\textit{Res. Number Theory} 6, Paper No. 10, 2020.

\bibitem{AD}
B. Alberts and E. O'Dorney.
Harmonic Analysis and Statistics of the First Galois Cohomology Group.
\textit{arXiv preprint}, 2102.11223.

\bibitem{ASVW}
S. Ali Altu\u{g}, A. Shankar, I. Varma and K.H. Wilson.
The number of quartic $D_4$-fields ordered by conductor.
\textit{arXiv preprint}, 1704.01729.
	
\bibitem{BhargavaI}
M. Bhargava.
The density of discriminants of quartic rings and fields. 
\textit{Ann. of Math. (2)} 162:1031-1063, 2005. 	

\bibitem{BhargavaII}
M. Bhargava.
The density of discriminants of quintic rings and fields.
\textit{Ann. of Math. (2)} 172:1559-1591, 2010.

\bibitem{BSTTTZ}
M. Bhargava, A. Shankar, T. Taniguchi, F. Thorne, J. Tsimerman and Y. Zhao.
Bounds on $2$-torsion in class groups of number fields and integral points on elliptic curves.
\textit{J. Amer. Math. Soc.} 33:1087-1099, 2020. 
	
\bibitem{BST}
M. Bhargava, A. Shankar and J. Tsimerman.
On the Davenport-Heilbronn theorems and second order terms.
\textit{Invent. Math.} 193:439-499, 2013.

\bibitem{BhargavaWood}
M. Bhargava and M.M. Wood. 
The density of discriminants of $S_3$-sextic number fields. 
\textit{Proc. Amer. Math. Soc.} 136:1581-1587, 2008.

\bibitem{CDO}
H. Cohen, F. Diaz y Diaz and M. Olivier.
Enumerating quartic dihedral extensions of $\mathbb{Q}$. 
\textit{Compositio Math.} 133:65-93, 2002. 

\bibitem{CDO2}
H. Cohen, F. Diaz y Diaz and M. Olivier.
On the density of discriminants of cyclic extensions of prime degree.
\textit{J. Reine Angew. Math.} 550:169-209, 2002. 

\bibitem{cohen--lenstra}
H. Cohen and H.W. Lenstra.
Heuristics on class groups of number fields.  
Number theory, {N}oordwijkerhout 1983, 
\textit{Lecture Notes in Math.},
\textit{Springer, Berlin}, 1984,
33-62.

\bibitem{Coleman}
M. D. Coleman.
The Rosser--Iwaniec sieve in number fields, with an application.
\textit{Acta Arith.} 65:53-83, 1993.

\bibitem{Couveignes}
J.-M. Couveignes. 
Enumerating number fields. 
\textit{Ann. of Math. (2)} 192:487-497, 2020.

\bibitem{DatskovskyWright}
B. Datskovsky and D.J. Wright.
Density of discriminants of cubic extensions. 
\textit{J. Reine Angew. Math.} 386:116-138, 1988.

\bibitem{DH}
H. Davenport and H. Heilbronn. 
On the density of discriminants of cubic fields, II. 
\textit{Proc. Roy. Soc. Lond. A} 322:405-420, 1971. 

\bibitem{ETW}
J. Ellenberg, T. Tran and C. Westerland.
Fox-Neuwirth-Fuks cells, quantum shuffle algebras, and Malle's conjecture for function fields.
\textit{arXiv preprint}, 1701.04541.

\bibitem{EV}
J. Ellenberg and A. Venkatesh.
The number of extensions of a number field with fixed degree and bounded discriminant.
\textit{Ann. of Math. (2)} 163:723-741, 2006.

\bibitem{EVW}
J. Ellenberg, A. Venkatesh and C. Westerland.
Homological stability for Hurwitz spaces and the Cohen-Lenstra conjecture over function fields.
\textit{Ann. of Math. (2)} 183:729-786, 2016.
 
\bibitem{FK1}
\'E. Fouvry and J. Kl\"uners. 
Cohen-Lenstra heuristics of quadratic number fields.
Algorithmic number theory, 
\textit{Lecture Notes in Comput. Sci.}, 4076, 
\textit{Springer, Berlin}, 2006,
40-55.

\bibitem{fouvry--kluners}
\'E. Fouvry and J. Kl\"uners.
On the 4-rank of class groups of quadratic number fields.
\textit{Invent. Math.} 167:455-513, 2007.

\bibitem{FK}
\'E. Fouvry and P. Koymans.
Malle's conjecture for nonic Heisenberg extensions.
\textit{arXiv preprint}, 2102.09465.

\bibitem{Gerth}
F. Gerth.
The $4$-class ranks of quadratic fields. 
\textit{Invent. Math.} 77:498-515, 1984. 

\bibitem{HB}
D.R. Heath-Brown. 
The size of Selmer groups for the congruent number problem, II. 
\textit{Invent. Math.} 118:331-370, 1994.

\bibitem{HSV}
W. Ho, A. Shankar and I. Varma.
Odd degree number fields with odd class number.
\textit{Duke Math. J.} 167:995-1047, 2018. 

\bibitem{IK}
H. Iwaniec and E. Kowalski.
Analytic number theory.
\textit{Colloquium Publications}, 53. American Mathematical Society, Providence, RI, 2004.

\bibitem{KlunersCounter}
J. Kl\"uners.
A counterexample to Malle's conjecture on the asymptotics of discriminants. 
\textit{C. R. Math. Acad. Sci. Paris} 340:411-414, 2005.

\bibitem{KlunersHab}
J. Kl\"uners.
\textit{\"Uber die Asymptotik von Zahlk\"orpern mit vorgegebener Galoisgruppe}. 
Shaker, 2005.

\bibitem{KlunersUpper}
J. Kl\"uners.
The asymptotics of nilpotent Galois groups.
\textit{arXiv preprint}, 2011.04325.

\bibitem{KlunersMalle}
J. Kl\"uners and G. Malle.
Counting nilpotent Galois extensions. 
\textit{J. Reine Angew. Math.} 572:1-26, 2004.

\bibitem{KlunersWang}
J. Kl\"uners and J. Wang.
$\ell$-torsion bounds for the class group of number fields with an $\ell$-group as Galois group.
\textit{arXiv preprint}, 2003.12161.

\bibitem{KP1}
P. Koymans and C. Pagano.
On the distribution of $\text{Cl}(K)[l^{\infty}]$ for degree $l$ cyclic fields. 
\textit{arXiv preprint}, 1812.06884.

\bibitem{KP2}
P. Koymans and C. Pagano.
Higher genus theory. 
\textit{Int. Math. Res. Not.}, 2020, rnaa196, \url{https://doi.org/10.1093/imrn/rnaa196}.

\bibitem{KP3}
P. Koymans and C. Pagano.
A sharp upper bound for the $2$-torsion of class groups of multiquadratic fields.
\textit{arXiv preprint}, 2009.08399.

\bibitem{Louboutin}
S. Louboutin. 
Explicit upper bounds for residues of Dedekind zeta functions and values of $L$-functions at $s = 1$, and explicit lower bounds for relative class numbers of CM-fields. 
\textit{Canad. J. Math.} 53:1194-1222, 2001.

\bibitem{OlivierThorne}
R.J. Lemke Oliver and F. Thorne.
Upper bounds on number fields of given degree and bounded discriminant.
\textit{arXiv preprint}, 2005.14110.

\bibitem{Malle1}
G. Malle. 
On the distribution of Galois groups. 
\textit{J. Number Theory} 92:315-329, 2002.

\bibitem{Malle2}
G. Malle.
On the distribution of Galois groups. II. 
\textit{Experiment. Math.} 13:129-135, 2004. 

\bibitem{MTTW}
R. Masri, F. Thorne, W.-L. Tsai and J. Wang.
Malle's Conjecture for $G \times A$, with $G = S_3, S_4, S_5$.
\textit{arXiv preprint}, 2004.04651v2.

\bibitem{MR}
B. Mazur and K. Rubin.
Kolyvagin systems.
\textit{Mem. Amer. Math. Soc.} 168, No. 799, 2004.

\bibitem{MoVa} 
H.L. Montgomery  and R.C. Vaughan.
Multiplicative number theory. I. Classical theory. 
\textit{Cambridge Studies in Advanced Mathematics}, 97. Cambridge University Press, Cambridge, 2007.
	
\bibitem{Pagano--Sofos}
C. Pagano and E. Sofos.
$4$-ranks and the general model for statistics of ray class groups of imaginary quadratic fields. 
\textit{arXiv preprint}, 1710.07587.

\bibitem{Schmidt}
W.M. Schmidt.
Number fields of given degree and bounded discriminant.
\textit{Ast\'erisque} 228:189-195, 1995. 

\bibitem{Shafarevich}
I.R. Shafarevich. 
Extensions with given points of ramification (Russian).
\textit{Inst. Hautes \'Etudes Sci. Publ. Math.} 18, 71-95, 1963.
English translation in: Collected mathematical papers. Springer-Verlag, Berlin, 1989.

\bibitem{Siad}
A. Siad.
Monogenic fields with odd class number Part I: odd degree.
\textit{arXiv preprint}, 2011.08834.

\bibitem{Siad2}
A. Siad.
Monogenic fields with odd class number Part II: even degree.
\textit{arXiv preprint}, 2011.08842.

\bibitem{Smith1}  
A. Smith.
Governing fields and statistics for {$4$}-Selmer groups, {$8$}-class groups.
\textit{arXiv preprint}, 1607.07860.

\bibitem{Smith}
A. Smith.
$2^\infty$-Selmer Groups, $2^\infty$-class groups, and Goldfeld's conjecture.
\textit{arXiv preprint}, 1702.02325v2.

\bibitem{TT}
T. Taniguchi and F. Thorne.
Secondary terms in counting functions for cubic fields. 
\textit{Duke Math. J.} 162:2451-2508, 2013.

\bibitem{TZ}
J. Thorner and A. Zaman.
A unified and improved Chebotarev density theorem. 
\textit{Algebra Number Theory} 13:1039-1068, 2019.

\bibitem{Turkelli}
S. T\"urkelli.
Connected components of Hurwitz schemes and Malle's conjecture. 
\textit{J. Number Theory} 155:163-201, 2015.

\bibitem{Wang}
J. Wang.
Malle's Conjecture for $S_n \times A$ for $n = 3, 4, 5$.
\textit{arXiv preprint}, 1705.00044v3.

\bibitem{Wright}
D.J. Wright.
Distribution of discriminants of abelian extensions. 
\textit{Proc. London Math. Soc. (3)} 58:17-50, 1989.
\end{thebibliography}
\end{document}